\documentclass[10pt,leqno]{article}
\usepackage[pagewise]{lineno}
\usepackage[toc,page]{appendix}
\usepackage{amsmath,amssymb,amsfonts,graphicx,bm,amsthm,geometry}
\usepackage{mathtools,stackengine}
\usepackage{mathrsfs}
\usepackage{enumitem}
\usepackage{latexsym}
\usepackage{rotating}
\usepackage{xspace}
\usepackage[usenames]{xcolor}
\usepackage{stmaryrd}
\usepackage[colorlinks=true]{hyperref}
\definecolor{ultramarine}{rgb}{0 , 0, 0.3}
\definecolor{bulgarianrose}{rgb}{0.3, 0, 0}
\hypersetup{citecolor={ultramarine},
urlcolor={ultramarine},linkcolor={bulgarianrose}}
\usepackage[capitalise]{cleveref}

\usepackage{soul}
\usepackage{bussproofs}
\usepackage{pstricks}
\usepackage[square]{natbib}
\usepackage{aliascnt}

\usepackage[a]{esvect}

\numberwithin{equation}{section}
\newcommand{\myref}{/home/mojtaba/texmf/tex/latex/local/bibliography.bib}
\newcommand{\mojref}{\bibliography{\myref}%
\bibliographystyle{apalike}%
\end{document}%
}

\def\lcal{\mathcal{L}}

\def\lcalz{\lcal_{0}}

\def\lcalza{\lcalz(\abold)}

\def\lcalzvar{\lcalz(\varr)}

\def\lcalzpar{\lcalz(\parr)}
\def\kcal{\mathcal{K}}

\def\kcalh{\hat{\mathcal{K}}}

\def\IPC{{\sf IPC}}
\def\ipc{{\sf IPC}}

\def\HA{\hbox{\sf HA}{} }

\def\K4{\hbox{\sf K4}{} }

\newcommand{\ttbrace}[1]{[\![{#1}]\!]}

\newcommand{\tbnpr}[1]{\ttbrace{#1}_{_{\nnilpar}}}

\newcommand{\Ax}{\AxiomC}
\newcommand{\UI}{\UnaryInfC}
\newcommand{\BI}{\BinaryInfC}

\newcommand{\LLa}{\LeftLabel}
\newcommand{\RLa}{\RightLabel}
\newcommand{\DP}{\DisplayProof}

\newcommand{\uparan}[1]{\textup{(}#1\textup{)}}

\newcommand{\bs}[1]{\vv{#1}}
\def\xbold{{\bs{x}}}

\def\abold{{\bs{a}}}
\def\bbold{{\bs{b}}}
\def\ybold{{\bs{y}}}
\def\pbold{{\bs{p}}}

\def\qbold{{\bs{q}}}

\def\IPC{{\sf IPC}}
\def\HA{\hbox{\sf HA}{} }

\def\NNIL{{\sf N}}

\def\nnil{{\sf N}}

\def\K4{\hbox{\sf K4}{} }

\newcommand{\pres}[2]{\mathrel{\setlength{\unitlength}{1ex}
                    \begin{picture}(2,2)
                    \put(0,0){$\pr$}
                    \put(-.1,1.3){\fontsize{4}{0} \selectfont ${#1}$}
                    \put(-.1,-.8){\fontsize{4}{0} \selectfont ${#2}$}
                    \end{picture}}
                    }
\newcommand{\preslow}[2]{\mathrel{\setlength{\unitlength}{1ex}
                    \begin{picture}(2,2)
                    \put(0,0){$\pr$}
                    \put(-.1,1.3){\fontsize{4}{0} \selectfont ${#1}$}
                    \put(-.1,-1.2){\fontsize{4}{0} \selectfont ${#2}$}
                    \end{picture}}
                    }

\newcommand{\adsm}[2]{\mathrel{\setlength{\unitlength}{1ex}
                    \begin{picture}(2,2)
                    \put(0,0){$\ar$}
                    \put(-.3,1.2){\fontsize{4}{0} \selectfont ${#1}$}
                    \put(-.3,-.5){\fontsize{4}{0} \selectfont ${#2}$}
                    \end{picture}}
                    }

\newcommand{\vdashsub}[1]{\mathrel{\setlength{\unitlength}{1ex}
                    \begin{picture}(2,2)
                    \put(0,0){$\vdash$}
                    \put(0.1,1.2){\fontsize{4}{0} \selectfont 
                    $ $}
                    \put(0.1,-.4){\fontsize{4}{0} \selectfont ${#1}$}
                    \end{picture}}
                    }                                                  
\def\pr{\mid\!\approx}

\def\prtgp{\mathrel{ ^{^*}\!\!\!\prtg}}
\newcommand{\submodel}[1]{\preceq^{{#1}}}
\newcommand{\submodelp}[1]{\preceq_1^{{#1}}}
\newcommand{\submodeli}[1]{\subseteq_{}^{{#1}}}
\def\prtdg{\pres{\sft}{\dgamma}}
\def\prtgv{\pres{\sft}{\Gamma^\vee}}

\def\prtdgv{\pres{\sft}{\dgamma^{\!\vee}}}

\def\prpnpar{\mathrel{\pres{}{\pnnilpar}\ \ \ }}

\def\prnpar{\mathrel{\pres{}{\nnilpar}\ \ \ }}

\def\pce{\preccurlyeq}
\def\sce{\succcurlyeq}

\newcommand{\V}{\mathrel{V}}

\newcommand{\ra }{\rightarrow }

\newcommand{\lr}{\leftrightarrow}

\newtheorem{theorem}{Theorem}[section]
\newtheorem{question}{Question}

\newaliascnt{lemma}{theorem}  
\newtheorem{lemma}[lemma]{Lemma}  
\aliascntresetthe{lemma}  
  
\newaliascnt{definition}{theorem}  
\newtheorem{definition}[definition]{Definition}  
\aliascntresetthe{definition}  
 
\newaliascnt{corollary}{theorem}  
\newtheorem{corollary}[corollary]{Corollary}  
\aliascntresetthe{corollary}  
 
\newaliascnt{proposition}{theorem}  
  
\aliascntresetthe{proposition}  
 
\newaliascnt{remark}{theorem}  
\newtheorem{remark}[remark]{Remark}  
\aliascntresetthe{remark}  
 
\newaliascnt{notation}{theorem}  
  
\aliascntresetthe{notation}  
 
\newaliascnt{example}{theorem}  
\newtheorem{example}[example]{Example}  
\aliascntresetthe{example}  
 
\newaliascnt{conjecture}{theorem}  
  
\aliascntresetthe{conjecture}  
 
\newaliascnt{fact}{theorem}  
  
\aliascntresetthe{fact}  
 
\newaliascnt{claim}{theorem}  
  
\aliascntresetthe{claim}  
 
\newcommand{\sub}[1]{{\sf sub}(#1)}

\def\scrk{\mathscr{C}}
\def\scrf{\mathscr{F}}

\def\scrkn{\langle\scrk\rangle_n}

\newcommand{\Mod}[1]{\textup{Mod}(#1)}

\def\txa{\theta^\xbold_{\!\!A}}
\newcommand{\ta}[1]{{\theta^{#1}_{\!\!A}}}

\def\NNILa{\NNIL(\abold)}

\def\nnilpar{{\nnil(\parr)}}
\def\NNILpar{{\NNIL(\parr)}}
\def\NNILp{{\NNIL(\pbold)}}

\def\ex{\exists\,}
\def\fa{\forall\,}
\def\mipc{{modulo $\IPC$-provable equivalence}{} }

\newcommand{\charac}[2]{{\ttbrace{#1}_{#2}}}
\newcommand{\charn}[1]{\charac{#1}{n}}

\def\charkn{\charn{\kcal}}

\def\bat{B^{A'}_\theta}
\def\ar{            \mathrel{\setlength{\unitlength}{1ex}
                    \begin{picture}(1.65,1.65)           
                    \put(0,0){\line(0,1){1.65}}                  
                    \put(0,.3){\textup{\footnotesize{$\sim$}}}
                    \end{picture}}\!
        }
\def\arn{\mathrel{\adsm{}{\nnilpar}\ \ }} 
\def\arpn{\mathrel{\adsm{}{\pnnil}}}

\def\arg{\mathrel{\adsm{}{\Gamma}}}
\def\argt{\mathrel{\adsm{\sft}{\Gamma}}}
\def\prtg{\mathrel{\pres{\sft}{\Gamma}}}

\def\prg{\mathrel{\pres{}{\Gamma}}}
\def\artg{\argt}
\def\artgv{\mathrel{\adsm{\sft}{\Gamma^\vee}}\ }
\def\ARN{{{\sf AR}_{\parr}}\!\,}

\def\ARNN{{{\sf AR}^+_{\parr}}\!\,}

\def\BART{\BAR\sft}
\newcommand{\BAR}[1]{[#1]}

\def\parb{\ninetaghir{\sf par}}

\def\AR{{\sf AR}}
\def\VAR{{\sf V}^\parr_{\!\text{\fontsize{6}{0}\selectfont\AR}}}

\def\VARBP{{\sf V}'_{\!\text{\fontsize{6}{0}\selectfont\AR}}}
\def\VARBPP{{\sf V}''_{\!\text{\fontsize{6}{0}\selectfont\AR}}}

\newcommand{\itp}[2]{\itv{#1}{#2}{\parr}}

\newcommand{\ita}[2]{\itv{#1}{#2}{\parb}}

\newcommand{\itap}[2]{\itvp{#1}{#2}{\parb}}
\newcommand{\itapp}[2]{\itvpp{#1}{#2}{\parb}}
\newcommand{\itv}[3]{({#1}\xrightarrow{#3}{#2})}
\newcommand{\itvp}[3]{({#1}\mathrel{\xrightarrow{#3}{\!'}}{#2})}
\newcommand{\itvpp}[3]{({#1}\mathrel{\xrightarrow{#3}{\!''}}{#2})}

\newcommand{\nita}[2]{\nitv{#1}{#2}{\parb}}

\newcommand{\nitap}[2]{\nitvp{#1}{#2}{\parb}}
\newcommand{\nitapp}[2]{\nitvpp{#1}{#2}{\parb}}
\newcommand{\nitv}[3]{{#1}\xrightarrow{#3}{#2}}
\newcommand{\nitvp}[3]{{#1}\mathrel{\xrightarrow{#3}{\!'}}{#2}}
\newcommand{\nitvpp}[3]{{#1}\mathrel{\xrightarrow{#3}{\!''}}{#2}}

\def\eqp{\equiv_\parr}

\def\eqg{\equiv_\Gamma}
\def\mont{\textup{Mont}}

\def\montd{\mont(\Delta)\xspace}

\def\what{\hat{w}}
\def\pnnil{{\sf PN}}\def\pNNIL{{\sf P}\NNIL}
\def\pNNILp{\pNNIL(\pbold)}
\def\pnnilpar{\pnnil(\parr)}\def\pNNILpar{\pNNIL(\parr)}

\newcommand{\vdashm}{\Vdash^{^{\!\!-}}}

\newcommand{\nin}{\not\in}

\newcommand\xtwoheadrightarrow[2]{\ensurestackMath{\mathrel{%
			\stackengine{1pt}{%
				\stackengine{0pt}{\twoheadrightarrow}{\scriptstyle#2}{O}{c}{F}{F}{S}%
			}{\scriptstyle#1}{U}{c}{F}{F}{S}%
}}}

\newcommand\xrightarrowtail[2]{\ensurestackMath{\mathrel{%
			\stackengine{1pt}{%
				\stackengine{0pt}{{\rightarrowtail}\!\!\!\!\!{\twoheadrightarrow}}{\scriptstyle#2}{O}{c}{F}{F}{S}%
			}{\scriptstyle#1}{U}{c}{F}{F}{S}%
}}}

\newcommand{\xrat}[2]{\xrightarrowtail{#2}{#1}}

\newcommand{\xratth}{\xrightarrowtail{}{\theta}}

\newcommand{\xtra}[2]{\xtwoheadrightarrow{#2}{#1}}
\newcommand{\xtrath}{\xtwoheadrightarrow{}{\theta}}

\def\sft{{\varLambda}}
\def\sfn{{\sf N}}

\def\vdasht{\vdashsub\sft}

\def\tvdash{\sft\vdash}

\def\VA{{\sf sub}}

\def\cto{c_{\!_\to\!}}

\newcommand{\suba}[1]{{\atom}(#1)}

\def\atom{{\sf atom}}
\newcommand{\varr}{{\sf var}}
\newcommand{\parr}{{\sf par}}

\def\NI{{\sf NI}}
\newcommand{\cfrak}[1]{\mathfrak{c}(#1)}
\newcommand{\cfrakz}[1]{\bar{\mathfrak{c}}(#1)}
\newcommand{\ofrak}[1]{\mathfrak{o}(#1)}
\newcommand{\ifrak}[1]{\bar{\mathfrak{i}}(#1)}
\newcommand{\Ifrak}[1]{{\mathfrak{I}}(#1)}

\newcommand{\vfrak}[1]{\mathfrak{v}(#1)}

\def\nat{\mathbb{N}}

\newcommand{\ap}[3]{\lfloor{#3}\rfloor^{^{\!#2}}_{_{\!{#1}}}}

\def\dgamma{{\dar{\,}\Gamma}}
\def\dtgamma{{{\downarrow}^{\!\sft}\Gamma}}
\def\dgv{{{\downarrow}\Gamma^\vee}}
\newcommand{\darrow}[1]{{\downarrow}{#1}}
\def\fsubeq{\subseteq_{\text{fin}}}
\newcommand{\dar}[1]{{\downarrow^{\!\!^{#1}}}}

\newcommand{\sixtaghir}[1]{{#1}}
\newcommand{\seventaghir}[1]{{#1}}
\newcommand{\eighttaghir}[1]{{#1}}
\newcommand{\ninetaghir}[1]{{#1}}
\newcommand{\tentaghir}[1]{{#1}}

\def\altNNILpar{\seventaghir{\NNIL}}
\def\altnnilpar{\altNNILpar}
\def\altdnnilpar{\seventaghir{{\dar{\,}\NNIL}}}
\def\altpnnilpar{\seventaghir{\pNNIL}}
\def\altpNNILpar{\seventaghir{\pNNIL}}
\def\altarpn{\mathrel{\adsm{}{\altpnnilpar}\ \ \ }}
\def\altdNNILpar{{\dar{\,}\altNNILpar}}
\def\altarn{\mathrel{\adsm{}{\altnnilpar}\!\!}} 
\def\altprdnpar{\mathrel{\pres{}{\altdnnilpar}}}
\def\altarpn{\mathrel{\adsm{}{\altpnnilpar}}}
\def\altdpnnilpar{{\dar{\,}\altpnnilpar}}
\def\altprdpnpar{\mathrel{\pres{}{\altdpnnilpar}\ }}
\newcommand{\basisworlds}[1]{{#1}^{\mathfrak{b}}}

\newcommand{\mspan}[2]{\bgroup\def\arraystretch{1}
\begin{tabular}{l} 
{#1}
\\ 
{#2}
\end{tabular}\egroup}
\begin{document}
\title{Relative unification in intuitionistic logic:\\
towards the provability logic of Heyting Arithmetic
}
\author{
Mojtaba Mojtahedi\thanks{
Department of Mathematics, Statistics and Computer Science, 
College of Sciences, University of Tehran, Iran;
and Department of Mathematics WE16, Ghent University. Krijgslaan 281-S8, B9000 Ghent, 
Belgium.  
Email: \url{mojtahedy@gmail.com}, Home page: \url{http://mmojtahedi.ir}
}
}
\maketitle
\begin{abstract}
This paper studies relative 
unification and admissibility in intuitionistic logic. 
We generalize results of
\citep{Ghil99,Iemhoff-admissibility} and prove them relative to 
$\altNNILpar$ formulae, the class of forlmulae with 
No Nested Implications to the Left. 
The main application of this generalization is to 
characterize the provability logic of Heyting Arithmetic $\HA$
and prove its decidability  \citep{PLHA}.
\end{abstract}
\tableofcontents

\section{Introduction}
Silvio Ghilardi \citep{Ghil99,Ghil2000modal} studies 
unification in propositional logics. More precisely, he 
describes all solutions for
$ A(x_1,\ldots,x_n)\lr \top$
with respect to a background logic like intuitionistic logic $\IPC$ or a modal logic containing ${\sf K4}$. By a solution we mean a substitution 
$\theta$ such that $\theta(A \lr \top)$ holds.

On the other hand, we have a related question for 
decidability/characterization 
of admissible rules of $\IPC$.
A rule $A/B$ \textit{is admissible for a logic} ${\sf L}$ if 
${\sf L}\vdash\theta(A)$ implies ${\sf L}\vdash \theta(B)$
for every substitution $\theta$. 
In contrast to the case of classical logic, 
in which every admissible rule is also derivable,
the cases of modal logic and intuitionistic logic are not trivial. 
Probably the first such underivable admissible rule for 
$\IPC$ is the following \citep{Harrop60}: 
\begin{center}
\Ax{$\neg A\to (B\vee C)$}
\UI{$(\neg A\to B)\vee(\neg A\to C)$}
\DP
\end{center}
Using the tools and results in \citep{Ghil99}, Rosalie Iemhoff
proves the completeness of a base for 
\textit{all} admissible rules of $\IPC$
\citep{IemhoffT,Iemhoff-admissibility}, which was previously conjectured by 
de Jongh and Visser. 
Decidability of admissibility for $\IPC$ was already known
\citep{Rybakov_1987,rybakov_1992,Rybakov_Book}. 
There are 
similar results for some modal logics extending  ${\sf K4}$ both for 
unification \citep{Ghil2000modal} and 
admissibility \citep{jevrabek2005admissible,iemhoff2009proof}.

There is yet another related notion, \textit{preservativity}, 
an intuitionistic alternate for the classical notion of interpretability
or conservativity \citep{Iemhoff.Preservativity,Visser02}.  
Preservativity is a binary relation $A\pres{\sft}{\Gamma}B$
defined as ``for every $E\in\Gamma$, 
$\tvdash E\to A$ implies 
$\tvdash E\to  B$".
Albert Visser in \citep{Visser02} shows that $\NNIL$-preservativity and admissibility are tightly related, in which $\NNIL$, is the class of No
Nested Implications to the Left, introduced in 
\citep{Visser-Benthem-NNIL} and elaborated further in \citep{Visser02}. 
This class of formulas  proves to be helpful in the realm of 
intuitionistic logic. A crucial
result concerning $\NNIL$ appeared in \citep{Visser02}. He provides
an algorithm that takes   $A\in\lcalz$ and returns its best 
$\NNIL$ approximation $A^*$ from below, i.e., $  \vdash A^*\ra A$ and
for all $\NNIL$ formulas $B$ such that $  \vdash B\ra A$, we have
$  \vdash B\ra A^*$. 
Later in \Cref{sec-pres-2} we also provide an algorithm 
which computes $A^\star$, the best $\NNILpar$-approximation of $A$ from below.

The main work of the current paper is to extend 
\citep{Ghil99,Iemhoff-admissibility} and prove their results relative in 
$\altNNILpar$ formulas, the class of No Nested Implications to the Left \citep{Visser-Benthem-NNIL} \seventaghir{in the parametric language}.
First we imitate  \citep{Ghil99} and study 
projectivity and extendibility  relative in $\altNNILpar$ (\Cref{Theorem-Ghil-Ext}). This 
will lead us to a relativised version of projective approximations
(\Cref{Theorem-IPC-nnilp-Finitary}). 
Then we take a route similar to \citep{Iemhoff-admissibility} and 
provide a base called $\ARN$, for $\altNNILpar$-admissibility of $\IPC$
and prove its completeness (\Cref{Characterization-admissibility}).
This last result together with \citep{Sigma.Prov.HA,mojtahedi2021hard},
 lead us to the characterization and decidability 
of  provability logic of Heyting arithmetic $\HA$, 
which is contained in  \citep{PLHA}.

Finally we axiomatize two interesting preservativity predicates
$\pres{\IPC}{\Gamma}\ $: first when $\Gamma$ is considered as the set of 
$\altNNILpar$-projective formulas (this is same as projectivity relative in $\altNNILpar$, as defined in \Cref{Gamma-proj-nonmodal}), and second when $\Gamma:=\NNILpar$.

\section{Preliminary definitions and facts}
This section is devoted to preliminaries and conventions.  
Among other well-known notions, we 
define $\NNIL$ formulas, admissibility, preservativity and greatest lower bounds.
\subsection{Propositional language}\label{lang}
The propositional  language  $ \lcalz $ includes connectives
$ \vee$, $ \wedge $, $ \to $ and $ \bot $. Negation $ \neg $ is defined 
as $ \neg A:= A\to\bot $ and $\top:=\neg\bot$. 
\tentaghir{We assume that $ \lcalz $ includes
a set of atomic variables $ \varr $ 
and also  a set of atomic parameters 
$ \parr $.}
The union $ \varr\cup\parr $ is 
annotated as $ \atom $, the set of atoms. 
\tentaghir{Unless said otherwise, we assume that the  $\atom$ is finite.}
We use $\pbold$ and $\qbold$ as a finite set or list of parameters
and $\xbold$ and $\ybold$ for a finite set or list of variables.  
Finite lists or sets of atoms are annotated by $\abold$ and $\bbold$.
We use $x$, $y$ and $z$ (possibly with subscripts) as  meta-variables for variables and also $p$, $q$ and $r$ (possibly with subscripts) for parameters. Also $a$, $b$ and $c$
(again possibly with subscripts) are used for both atomic variables and parameters.

Let $\bs{a}=a_1,\ldots,a_n$ be a list of atoms and 
$\bs{B}=B_1,\ldots,B_n$. Then $A[\bs{a}:\bs{B}]$ 
indicate the simultaneous substitution of 
$B_i$ for $a_i$ in $A$.

We also use the notation $ \lcalz(X) $ to indicate the language of all \sixtaghir{B}oolean combinations of formulas in $ X $. 
We use $ \IPC $ for 
intuitionistic propositional logic
\citep{TD} and 
$\vdash $ indicates derivability in $\IPC$. 
\subsection{Substitutions}
A substitution $\theta$ is a function on propositional  language $\lcalz$
which commutes with all connectives, i.e.
\begin{itemize}
\item $\theta(B\circ C)=\theta(B)\circ\theta(C)$ for every $\circ\in\{\vee,\wedge,\to\}$.
\item $\theta(\bot)=\bot$.
\end{itemize}
\textit{By default we assume that all substitutions are identity on the set 
$\parr$ of parameters.} We say that a substitution is \textit{general},
if we relax this condition on $\parr$ and allow the parameters to be substituted as well.

\sixtaghir{A propositional logic is a set 
$\sft\subseteq \lcalz$ containing all
theorems of intuitionistic logic and
closed under intuitionistic derivability
and general substitutions, 
i.e.~$A\in \sft$ implies $\theta(A)\in\sft$ for every general
substitution $\theta$ and 
$\sft\vdash A$ implies $A\in\sft$. 
We always reserve 
$\sft$ as a meta-variable for logics.
Furthermore, for a given logic $\sft$, we define $\Gamma\vdasht A$ as a shorthand for 
$\Gamma\cup\sft\vdash A$.}
\subsection{Kripke models for intuitionistic logic}
A Kripke model for intuitionistic logic, is a triple
$ \kcal=(W,\prec,\V) $ with following properties:
\begin{itemize}
	\item $ W\neq\emptyset $.
	\item $ (W, \prec) $ is a partial order (transitive and irreflexive). We write $ \pce $ for the reflexive closure of $ \prec $.
	\item $ \V $ is the valuation on atoms, 
	i.e.~$ V\subseteq W\times\atom  $.
	\item $ w\pce u $ and $ w\V a $ implies $ u\V a $ for every $ w,u\in W $ and $ a\in\atom $.
\end{itemize}
The valuation $ \V $ may be extended to include all  formulas as follows:
\begin{itemize}
	\item $ \kcal,w\Vdash a $ iff $ w\V a $, for $ a\in\atom $.
	\item $ \kcal,w\Vdash A\wedge B $ iff $ \kcal,w\Vdash A $ and 
	$ \kcal,w\Vdash B $.
	\item $ \kcal,w\Vdash A\vee B $ iff $ \kcal,w\Vdash A $ or 
	$ \kcal,w\Vdash B $.
	\item $ \kcal,w\Vdash A\to B $ iff for every $ u\sce w $ if we have 
	$ \kcal,w\Vdash A $ then
	$ \kcal,w\Vdash B $.
\end{itemize}
We also define the following notions for Kripke models:
\begin{itemize}
	\item \textit{Finite:} if $ W $ is a finite set. 
	\item \textit{Rooted:} if there is some node $ w_0\in W $
	such that $ w_0\pce w $ for every $ w\in W $.
	\item \textit{Tree:} if for every $ w\in W $ the set 
	$ \{u\in W: u\pce w\} $ is finite linearly  ordered 
	(by $ \pce $) set. 
\end{itemize}
By default we assume that \textit{all Kripke models of $\IPC$ in this paper 
are finite, rooted and tree.} As we will see in \Cref{sec-ARmod}, 
some other sort of Kripke semantics are used, called $\ARN$-models,
which might not be finite or tree.
Given $A\in\lcalz$, we define $\Mod{A}$ as the class of all 
(finite rooted tree) Kripke models of $A$. 
\subsection{No Nested Implications in the Left}\label{sec-NNIL-def}
The class of {\em No Nested Implications to the Left}, 
$\NNIL$ formulas, was discovered by Albert Visser and first 
published in
\citep{Visser-Benthem-NNIL}, and more explored in \citep{Visser02,NNIL-rev}. 
For simplicity of notations, we may write ${\sf N}$ for $\NNIL$.
The crucial
result of \citep{Visser02} is to provide
an algorithm that  
takes   $A\in\lcalz$ and returns its best 
$\NNIL$ approximation $A^*$ from below, i.e., $  \vdash A^*\ra A$ and
for all $\NNIL$ formulas $B$ such that $  \vdash B\ra A$, we have
$  \vdash B\ra A^*$. 
Later in this paper we define another algorithm $A^\star$
which calculates the best $\NNILpar$-approximation of $A$ from below
(\Cref{sec-pres-2}).
The classes $\NNIL$ and $ \NI $ of formulas 
in $\lcalz$ are defined inductively: 
\begin{itemize}
\item $ a\in\NNIL $ and $a\in\NI $ for every $ a\in\atom $.
\item $ B\circ C\in \NNIL $ if $ B,C\in\NNIL $. 
Also $ B\circ C\in \NI $ if $ B,C\in\NI $. ($ \circ\in\{\vee,\wedge\} $)
\item $ B\to C\in\NNIL $ if $ B\in\NI $ and $ C\in\NNIL $.
\end{itemize}
\subsection{Notations on sets of formulas}\label{notation-set}
In rest of the paper we deal with several sets of formulas and following notations  make life easier.
Given $A\in\lcalz$, let $ \sub{A} $ be the set of all subformulas of $ A $. 
For simplicity of notations, we write 
${ X_1\ldots X_n}$ for ${ X_1}\cap\ldots\cap{ X_n}$, when 
${ X_i}$  are sets of formulas. 
For a set $\Gamma$ of formulas define 
\begin{itemize}
	\item $\Gamma^\vee:=\{\bigvee \Delta:\Delta\fsubeq \Gamma
	\text{ and } \Delta\neq\emptyset
	\}.$ ($X\fsubeq Y$ indicates that $X$ is a finite subset of $Y$)
    \item $ \Gamma(X) $ indicates the set $ \Gamma\cap\lcalz(X) $. 
	\item $ \dtgamma:= $ the class of all 
	$\Gamma$-projective formulas in $\sft$. 
	We say that a formula $ A $ is \textit{$ \Gamma $-projective in $\sft$}, 
	if there is some substitution $ \theta $  and \seventaghir{$ B\in\Gamma(\parr) $} such that 
	$ \tvdash \theta(A)\lr B $ and $ A\vdasht x\lr \theta(x) $ for every $ x\in\varr $ (see \Cref{Gamma-proj-nonmodal}).  Remember that all substitutions are identity 	on $\parr$ and hence this definition implies 
	that $ A\vdasht a\lr \theta(a) $ for every $ a\in\atom $.
	Whenever $\sft=\IPC$, we may omit the superscript $ \sft $ and simply write $ \darrow \Gamma$. 
Also note that with this definition we have 
$\seventaghir{\Gamma(\parr)}\subseteq \dtgamma$, as witnessed by identity substitution.	
\end{itemize}
Also define 
\begin{itemize}
	\item $ \sfn:= $ as defined in \Cref{sec-NNIL-def}.
	\item ${\sf P}^X:={\sf Prime}^X:=$ the set of all 
	$X$-prime formulas, i.e.~the set of formulas $A$
	such that for every $B,C$ with $X\vdash A\to (B\vee C)$
	either we have $X\vdash A\to B$ or $X\vdash A\to C$.
Whenever $X=\emptyset$,  we may omit 
	the $X$ from notations.
\end{itemize}
And finally we assume that
$ (.)^\vee $ has the  lowest precedence after $ \darrow{(.)} $. This means that 

$$\darrow{XY}^\vee:=\left(\darrow{\left(XY\right)}\right)^\vee. $$  
\subsection{Admissibility and preservativity}\label{pres-admis}
Given a logic $\sft$, the binary relation $\adsm{\sft}{}$
is defined to hold for those pairs $A$ and $B$ such that the inference 
rule $A/B$ is admissible. More precisely $A\adsm{\sft}{}B$ iff 
for every substitution $\theta$, $\tvdash \theta(A)$ implies 
$\tvdash \theta(B)$.
The admissibility relationship is trivial when one considers classical
propositional logic, since every admissible $A/B$ is also derivable. 
However this relationship is highly nontrivial when one considers 
a modal logic or intuitionistic logic. Probably the first known 
non-derivable admissible rule is the following rule \citep{Harrop60}:
\begin{prooftree}
\Ax{$\neg A\to (B\vee C)$}
\UI{$(\neg A\to B)\vee(\neg A\to C)$}
\end{prooftree}
Harvey Friedman asked in 1975 for decidability of admissibility 
in the intuitionistic propositional logic.  Then 
\citep{Rybakov_1987,rybakov_1992,Rybakov_Book}  answers to this question 
positively. Although it was shown that no finite base exists for all 
admissible rules of intuitionistic logic $\IPC$ \citep{Rybakov87},
de Jongh and Visser introduced a recursive base and conjectured it to generate all admissible rules of $\IPC$. Then Iemhoff proved this 
conjecture \citep{Iemhoff-admissibility,IemhoffT}. 

Here in this paper, we consider a relativised version  of admissibility. 
Given a logic $\sft$ and a set $\Gamma$ of formulas define 
the $\Gamma$-admissibility relation in $\sft$ as follows
\begin{center}
	$A\argt B$ iff for every  substitution $\theta$
	and $C\in\seventaghir{\Gamma(\parr)}$: 
	$ \tvdash C\to\theta(A)$ 
	implies $ \tvdash C\to \theta(B)$.
\end{center}
Note that there is a hidden role for the language $\lcalz$
in the definition of $\argt$, when we consider substitution $\theta$.
However since almost everywhere in the paper we fix the language $\lcalz$,
by default we assume substitutions over this fixed language and we do not explicitly mention $\lcalz$. 

There is also another binary relation on formulas, called 
preservativity. 
The $\Gamma$-preservativity relation in $\sft$ is defined 
as follows:

\begin{equation*}
A\prtg B \quad \text{ iff } 
\quad \forall\,E\in\Gamma(\tvdash E\to A \Rightarrow \tvdash E\to B).
\end{equation*}
Preservativity could be considered as intuitionistic analogue of classical
interpretability or conservativity. 
This notion for the intuitionistic propositional logic,  well studied in
\citep{Visser02} and  \citep{Iemhoff.Preservativity} provided 
Kripke semantics for it. \citep{Zhou-PhD,Iemhoff2005} include some 
more elaboration on preservativity and provability, including fixed-point theorem and Beth property.

Following theorem says that $\argt$ and $\prtg$ are \sixtaghir{descending} 
on $\Gamma$. All over this paper we may 
use this fact without mentioning. 
\begin{theorem}\label{pres-admis-asce}
If $\Gamma\subseteq \Gamma'$  
then $ {\pres{\sft}{\Gamma'}}\subseteq {\prtg}$ and 
${\adsm{\sft}{\Gamma'}}\subseteq {\argt}$.
\end{theorem}
\begin{proof}
Left to the reader.
\end{proof}
\begin{theorem}\label{pres-admis-rel}
$A\argt B$ implies $A\preslow{\sft}{\dtgamma}B$.
\end{theorem}
\begin{proof}
Let $A\argt B$ and $E\in\dtgamma$ such that $\tvdash E\to A$.
Since $E\in\dtgamma$ there is some $\theta$ and $E^\dagger\in
\seventaghir{\Gamma(\parr)}$ such that 
$E\vdasht \theta(a)\lr a$ for every $a\in\atom$ and 
$\tvdash \theta(E)\lr E^\dagger$.
\seventaghir{%
Hence by $\tvdash E\to A$ we have $\tvdash \theta(E)\to\theta(A)$ and thus
$\tvdash E^\dagger \to \theta(A)$. 
Since $E^\dagger\in\lcalzpar$ we have $\theta(E^\dagger)=E^\dagger$ and thus
we may conclude $\tvdash \theta(E^\dagger\to A)$. 
Then by $A\argt B$
we get $\tvdash  \theta(E^\dagger\to B)$ and thus 
$E\vdasht \theta(E^\dagger\to B)$. Since $E\vdasht \theta(a)\lr a$,
we may conclude $E\vdasht E^\dagger\to B$. Then since $\tvdash \theta(E)\lr E^\dagger$, we have $E\vdasht \theta(E)\to B$. Again from $E\vdasht \theta(a)\lr a$ we may conclude $E\vdasht E\to B$ and thus $\tvdash E\to B$.}
\end{proof}
\begin{question}
What can be said about the other direction of \Cref{pres-admis-rel}?
\end{question}

Later in this paper we axiomatize $\prtg$  and $\argt$
for several pairs $(\sft,\Gamma)$. Before we continue with this, let us see some basic axioms. 

Let $\sft$ be a  logic.
The logic $ \BART $
proves statements $ A\rhd B $ for $A,B\in\lcalz$
and has the following axioms and rules: 
\\[4mm]
\textbf{Axioms}
\begin{itemize}[leftmargin=1.5cm]
	\item[${\sf Ax}:$] \quad $A\rhd B$, 
	for every $  \tvdash A\to B$. 
\end{itemize}	
\textbf{Rules}
\begin{center}
	\bgroup
	\begin{tabular}{c c}
		\Ax{$A\rhd B$}
		\Ax{$A\rhd C$}
		\RLa{Conj}
		\BI{$A\rhd B\wedge C$}
		\DP \quad \quad \quad 
		&
		\Ax{$A\rhd B$}
		\Ax{$B\rhd C$}
		\RLa{Cut}
		\BI{$A\rhd C$}
		\DP \quad \quad \quad 
	\end{tabular}
	\egroup
\end{center}
\vspace{4mm}
The above mentioned axiom and rules are not interesting, because 
$\BART\vdash A\rhd B$ iff $\tvdash A\to B$. However we define several interesting additional rules:
\begin{center}
	\bgroup
	\def\arraystretch{1}
	\begin{tabular}{c c}
		\quad \quad \quad \Ax{$B\rhd A$}
		\Ax{$C\rhd A$}
		\RLa{Disj}
		\BI{$B\vee C\rhd A$}
		\DP
		&
		\quad \quad \quad 
		\Ax{$A\rhd B$}
		\Ax{($ C\in\Delta $)}
		\RLa{$\montd$}
		\LLa{}
		\BI{$ C\to A\rhd  C\to B$}
		\DP
	\end{tabular}
	\egroup
\end{center}
Remember (\Cref{notation-set}) that $A$ is called 
$X$-prime, if  
$X\vdash A\to (B\vee C)$
		implies either  $X\vdash A\to B$ or $X\vdash A\to C$ for every
		 $B,C$.  Moreover a set $\Gamma$ of formulas is called \textit{$X$-prime}, if every $A\in\Gamma$ is 
		 $X$-prime.
		Also we say that $\Gamma$ is 
		 \textit{closed under 
		$\Delta$-conjunctions}, if 
		$A\in\Gamma$ and $B\in\Delta$ implies $A\wedge B\in \Gamma$ 
		\uparan{up to $\sft$-provable equivalence}.  
\begin{theorem}[\textbf{Soundness}]\label{gen-pres-sound}
	$\BART$ is sound for  relative admissibility 
	interpretations,
	i.e.~$\BART\vdash A\rhd B$ implies $A\argt B$ 
	  for every set 
	$\Gamma$ of formulas and every  logic $\sft$. Moreover
	\begin{enumerate}
		\item \seventaghir{if $\Gamma$ only includes $\sft$-prime  formulas
		then Disj is also sound,}
		\item if $\Gamma$ is closed under $\Delta$-conjunctions, then $\montd$ 
		is sound.
	\end{enumerate}
\end{theorem}
\begin{proof}
	Easy induction on the complexity of proof $\BART\vdash A\rhd B$
	and left to the reader. 
\end{proof}
\begin{theorem}\label{vee-pres}
	${\prtg}={\prtgv\ }$ and  ${\artg}={\artgv}$ .
\end{theorem}
\begin{proof}
	We only  show $A\prtg B$ iff $A\prtgv B$ and leave the similar argument for $A\artg B$ iff $A\artgv B$ to the reader. 
	The right-to-left direction holds since 
	$\Gamma\subseteq \Gamma^\vee$. For the other direction assume that 
	$A\prtg B$ and let $E\in\Gamma^\vee$ such that 
	$\tvdash E\to A$. Then $E=\bigvee_i E_i$ with $E_i\in\Gamma$. 
	Hence for every $i$ we have $\tvdash E_i\to A$. 
	Then $A\prtg B$ implies $\tvdash E_i\to B$. Thus 
	$\tvdash E\to B$, as desired.
\end{proof}
\noindent\textbf{Notation.} Whenever $\sft=\IPC$ we may omit the $\sft$
form notations $\argt$ and $\prtg$ and simply write 
$\arg$ and $\prg$ for them. Also if $\Gamma:=\{\top,\bot\}$ we may omit 
$\Gamma$ from notations.
\subsection{Greatest lower bounds}\label{glb}
Given a set $\Gamma\cup\{A\}$ 
of formulas, and a  logic 
$\sft$,  we say that $B$ is a lower bound for
$A$ w.r.t.~$(\Gamma,\sft)$, if the following conditions met:
\begin{enumerate}
	\item $B\in\Gamma$,
	\item $\tvdash B\to A$.
\end{enumerate}
Moreover we say that $B$ is the greatest 
lower bound (glb)   for $A$ w.r.t.~$(\Gamma,\sft)$, 
if it is a lower bound for $A$ w.r.t.~$(\Gamma,\sft)$ and
for every lower bound $B'$ for $A$ w.r.t.~$(\Gamma,\sft)$
we have $\tvdash B'\to B$. Note that
\sixtaghir{if such glb exists}, 
up to $\sft$-provable equivalence, it is unique 
and  we annotate it as 
$\ap{\Gamma}{\sft}{A}$.
\\
We say that $ (\Gamma,\sft) $ is downward compact, if 
every $A\in\lcalz$ has glb w.r.t.~$ (\Gamma,\sft) $.
%

\begin{theorem}\label{Gamma-approx-preserv}
	$B$ is the glb for $ A $ w.r.t.~$(\Gamma,\sft)$, iff 
	\begin{itemize}
		\item $B\in\Gamma$,
		\item $\tvdash B\to A$,
		\item $A\prtg B$.
	\end{itemize}
	Hence we have
	$A\prtg \ap\Gamma\sft A$. 
\end{theorem}
\begin{proof}
	Left to the reader.
\end{proof}
\begin{question}
	As we saw in \Cref{Gamma-approx-preserv}, the glb may be expressed via 
	preservativity relation $ \prtg $. One may think of $\Gamma$-conservativity relation:
	$$ A\prtgp B \quad \text{ iff }\quad \forall\,E\in\Gamma(\tvdash A\to E \Rightarrow \tvdash B\to E).$$
	\citep[Corollary 7.2]{Visser02} axiomatizes $\prtg$ for $\sft=\IPC$ and $\Gamma=\NNIL$. 	
	We ask for an axiomatization for $ \prtgp $ when we let $ \sft=\IPC $ and 
	$ \Gamma=\NNIL $. 
\end{question}
\begin{corollary}\label{Cor-Gamma-approx-preserv}
	If $\ap\Gamma\sft A$ exists, then for every $B\in\lcalz$ we have 
	
	$$\tvdash \ap\Gamma\sft A \to B\quad \text{ iff } 
	\quad 
	A\prtg B.
	$$
\end{corollary}
\begin{proof}
	First assume that $\tvdash \ap\Gamma\sft A \to B$. Also let $E\in\Gamma$
	such that $\tvdash E\to A$. \Cref{Gamma-approx-preserv} implies 
	$A\prtg\ap\Gamma\sft A$ and hence $\tvdash E\to \ap\Gamma\sft A$. 
	Then by $\tvdash \ap\Gamma\sft A\to B$ we get 
	$\tvdash E\to B$, as desired. \\
	For the other direction let 
	$A\prtg B$. By definition we have $\ap\Gamma\sft A\in\Gamma$
	and $\tvdash \ap\Gamma\sft A\to A$. Hence by $A\prtg B$
	we get $\tvdash \ap\Gamma\sft A\to B$, as desired. 
\end{proof}

\section{$\altNNILpar$-fication: unification to $\NNIL$}
\label{Sec-Ghil}
Silvio Ghilardi, in \citep{Ghil99} characterizes 
projective formulas in the language $\lcalz(\varr)$
with the aid of Kripke semantics. Then he uses this characterization to prove that the unification type of
$\IPC$ is finitary. Afterwards, Rosalie Iemhoff 
\citep{IemhoffT,Iemhoff-admissibility} uses this 
result together with a special sort of Kripke models, 
called AR-models, 
to characterize the admissible rules of $\IPC$. 
In this section we consider a  relativised version of those results.
The difference from the previous version is that 
we are not allowed substitute for parameters (a reserved set of atoms), and also instead of unification, we expect to 
simplify the formula to a 
$\altNNILpar$ formula, called $\altNNILpar$-fication. 
In fact, previous results will be special cases of ours when 
$\parr=\emptyset$ and hence $\NNILpar=\{\top,\bot\}$.
The methods of our proof follow  main roads taken in \citep{Ghil99,Iemhoff-admissibility}.
\\
We start with a relativised version of projective unification 
(\Cref{Gamma-proj-nonmodal}) and the 
extension property (\Cref{relative-ext-prop}). 
Then (\Cref{rel-ext-proj})
we prove a correspondence between relativised 
projectivity and extendibility.
Having such a Kripke semantical characterization 
in hand, 
then we prove that every formula has a finitary projective approximation (\Cref{proj-res}). 
Actually we prove something more: every formula 
has a finitary projective resolution (\Cref{def-proj-res}). 
Finally at the end of this section (\Cref{NNIL-resol}), we prove that 
in the specific  case, when $A\in\NNIL$, this finitary projective 
resolution takes an elegant form. 

\subsection{Relative projectivity}\label{Gamma-proj-nonmodal}

Given  $A\in\lcalz$, a substitution $\theta$ is 
called \textit{$A$-identity} (in $\IPC$)
if 

\begin{equation}
\text{For all atomic $a$ we have } \quad 
A\vdash  a\lr \theta(a).
\end{equation}
When one considers unification for propositional logics, 
projectivity is proved to be of great help \citep{Ghilardi97}.
As we will see, our study is not an exception.     
\\ 
If  $\Gamma\subseteq\seventaghir{\lcalz}$, a substitution $\theta$
is a \textit{$\Gamma$-fier}  (as a generalization for uni-fier) for 
$A$, if
$$
 \vdash \theta(A)\in\seventaghir{\Gamma(\parr)} 
\quad 
\text{i.e.~$\theta(A)$ is $\IPC$-equivalent to some 
$E\in\seventaghir{\Gamma(\parr)}$.}
$$
In this case we use the notation $ A\xtrath \Gamma $. 
$\theta$ is a unifier for $A$ if it is $\{\top\}$-fier for $A$. 
We say that a substitution  $ \theta $ projects $ A $ to $ \Gamma $
(notation: $  A\xratth \Gamma $)
if $ \theta  $ is $ A $-identity and  $ \Gamma $-fier. 
 We say that $ A $ is $ \Gamma $-projective (notation 
 $ A\xrat{}{} \Gamma $) if there is some $ \theta $ such that 
 $ A\xratth\Gamma $.
We say that 
$A$ is projective, if it is $\{\top \}$-projective.
Also $\dar{\,}{\Gamma}$
indicates the set of all formulas which are 
$\Gamma$-projective.
\sixtaghir{In all above notations, when the set $\Gamma$
is a singleton set $\{B\}$, we may skip braces
and simply write $B$ for $\Gamma$.}
\seventaghir{For example $A\xrat{\theta}{} B$ means that $\theta$ is $A$-identity and $B\in\lcalzpar$ and $\vdash\theta(A)\lr B$.}

\subsubsection*{Uniqueness of $\Gamma$-projections}
\seventaghir{Let $A\xrat{\theta}{} A'$ and  $A\xrat{\tau}{}A''$.}
Since   $\theta$ and $\tau$ are both $A$-identity,
for every atomic $a $ we have $A\vdash  \theta(a)
\lr\tau(a)$. Hence $A\vdash  \theta(A)\lr\tau(A)$
and then $A\vdash  A'\lr A''$. By applying 
$\theta$ to both sides of this derivation, 
we have $\theta(A)\vdash  \theta(A')\lr\theta(A'')$.
Since $\theta$ is  identity over parameters and $A',A''\in\lcalzpar$, 
we have $A'\vdash  A'\lr A''$. 
Hence $ \vdash A'\to A''$. Similarly we have 
$ \vdash A''\to A'$, and hence $ \vdash A'\lr A''$.
This argument shows that for every $\Gamma$-projective 
$A$, there is a unique (modulo $\IPC$-provable 
equivalence) $A^\dagger\in\Gamma$ such that 
for some $A$-identity $\theta$
we have $ \vdash \theta(A)\lr A^\dagger$. Such unique 
$A^\dagger$ is called the \textit{$\Gamma$-projection}
of $A$. 

Given a formula in $A$, the uniform post-interpolant of $A$ with respect to the set $\parr$
of atomic formulas, is a formula $B\in\lcalzpar$ with following properties:
\begin{itemize}
\item $\vdash A\to B$,
\item for every $C\in\lcalzpar$ such that $\vdash A\to C$ we have $\vdash B\to C$.
\end{itemize}
This means that $B$ is the least upper bound  for $A$ in $\lcalzpar$. 
The existance of such formula is non-trivial, 
however for the intuitionistic logic it always exists \citep{visser1996uniform}.
\begin{theorem}\label{Gamma-projectivity-pres}
\seventaghir{Let $A$ be $\Gamma$-projective. Then
$ A^\dagger$ is the uniform post-interpolant of $A$ with respect to the set $\parr$ of atomic formulas.}
\end{theorem}
\begin{proof}
Let $\theta$ be the $A$-identity $\Gamma$-fier for 
$A$, i.e.~$A\vdash B\lr \theta(B)$ for every $B$, 
and $  \vdash \theta(A)\lr A^\dagger$. 
Hence we have $A\vdash A\lr \theta(A)$ and thus $A\vdash A\lr A^\dagger$. 
This implies $\vdash A\to A^\dagger$. 
Next assume that $E\in\lcalzpar$ such that $\vdash A\to E$. Then $\vdash \theta(A\to E)$
and hence $\vdash A^\dagger\to \theta(E)$. Since \seventaghir{$E\in\lcalzpar$}
we get $\theta(E)=E$ and thus 
$\vdash A^\dagger\to E$. Hence we may conclude that $A^\dagger$ is the  uniform 
post interpolant of $A$ with respect to $\parr$.
\end{proof}

\begin{lemma}\label{proj-pres-admiss}
If  \seventaghir{$ A\xratth A^\dagger$}
 and $ B $ is an arbitrary formula,
then we have 
$$ \vdash A\to B \quad \text{ iff } \quad \vdash A^\dagger\to\theta( B).$$
\end{lemma}
\begin{proof}
	The left-to-right direction is obvious. For   other direction, 
	let $\vdash A^\dagger\to\theta( B)$. Hence 
	$A\vdash \theta(A^\dagger\to B)$ and then $A\vdash A^\dagger\to B$.
	\Cref{Gamma-projectivity-pres} implies $\vdash A\to A^\dagger$
	and thus $\vdash A\to B$.
\end{proof}

\subsection{Relative extendibility}\label{relative-ext-prop}
Given a Kripke model $\kcal=(W,\pce,\V)$ and $w\in W$,
$\kcal_w$ indicates the restriction of $\kcal$ to 
the nodes $u\succcurlyeq w$. 
\sixtaghir{%
For a set $\abold\subseteq\atom$ and  
$\kcal'=(W',\pce',\V')$ and 
$\kcal=(W,\pce,\V)$, 
define $\kcal'\submodel\abold \kcal$,
}
if
 there  exists a relation $R\subseteq W'\times W$ such 
 that 
 \begin{itemize}
 \item $\kcal',w'\Vdash a$ iff $\kcal,w\Vdash a$, 
 for every  $a\in \abold$ and $(w',w)\in R$.
 \item $v'\sce'w'\mathrel{R} w$ implies   
 $\ex v\in W\ (v'\mathrel{R}v\sce w)$, for every $w\in W$
 and $w',v'\in W'$.
 \item $\fa w'\in W'\ex w\in W(w'\mathrel{R} w)$. 
 \end{itemize}
Also we say that $\kcal'\submodelp\abold\kcal$
if the above relation $R$, is function. In this case the 
second condition takes a more readable face:
\begin{itemize}
\item $w'\pce' v'$ implies $f(w')\pce f(v')$.
\end{itemize}
\sixtaghir{%
Moreover we say that 
$\kcal'$ is a $\abold$-submodel of $\kcal$,
denoted as $\kcal'\submodeli\abold \kcal$, if   
the above function $f$  is injective.
}
Also $\scrk\submodel\abold\kcal$   for 
a class $\scrk$ of Kripke models and a Kripke model
$\kcal$ indicates that for every $\kcal'\in\scrk$
we have  $\kcal'\submodel\abold\kcal$.   We have similar notations 
for $\scrk\submodeli\abold\kcal$ and  $\scrk\submodelp\abold\kcal$.
\sixtaghir{
Furthermore,  we define $\abold$-bisimilarity 
$\kcal'\sim^\abold\kcal$ as follows.
$\kcal'\sim^\abold\kcal$ if the following conditions hold: 
\begin{itemize}
\item For every $a\in\abold$, $\kcal\Vdash a$ iff $\kcal'\Vdash a$.
\item For every $w\in W$ other than the root,
there is some $w'\in W'$ such that $\kcal_w\sim^\abold \kcal'_{w'}$.
\item For every $w'\in W'$ other than the root,
there is some $w\in W$ such that $\kcal_w\sim^\abold \kcal'_{w'}$.
\end{itemize}
For $\kcal\sim^\abold\kcal'$,
it can be easily proved by induction on complexity of a formula $A\in\lcalza$ that 
$\kcal\Vdash A$ iff $\kcal'\Vdash A$.
}	

Since we only consider Kripke models  with finite rooted tree frames, we have the equivalency of $\submodelp\abold$  and 
$\submodel\abold$:
\begin{lemma}\label{Remark-embed-sub}
$\kcal'\submodel\abold \kcal$
is equivalent to $\kcal'\submodelp\abold\kcal$. 
\end{lemma}
\begin{proof}
We only reason for the left-to-right direction. 
Let $\kcal'\submodel\abold \kcal$ and $R$ is a relation with above mentioned properties.   
We define $f (w')$ by induction on the distance of 
a $w'\in W'$
from the minimal elements such that 
$w'\mathrel{R} f(w')$ and for every 
$v'\pce' w'$ we have  $f(v')\pce f(w')$.  

For every $\pce'$-minimal node $w'\in W'$, 
define $f(w')$ an arbitrary $w\in W$ 
with $w' \mathrel{R} w$. 
Note that such a $w$ always exists and satisfies all
required properties.   

Then we define $f(w')$ for $w'\in W'$ 
which is not minimal.  As induction hypothesis, 
we assume that for every 
$u'\prec' w'$ we already defined $f(u')\in W$ 
such that $u' \mathrel{R} f(u')$ and $v'\pce' u'$
implies $f(v')\pce f(u')$. 
Since $\kcal'$ is tree,  
there is a unique predecessor 
$u'\prec' w'$. 
Then by the induction hypothesis and properties of $R$, 
there is some 
$w\in W$ such that $f(u')\pce w $ and $w'\mathrel{R}w$. For some such $w$, define $f(w'):=w$. 
\end{proof}
Although $\submodeli\abold$ is not equivalent to 
$\submodelp\abold$, we have the following \sixtaghir{correspondence}.

\begin{lemma}\label{Remark-embed-sub2}
\sixtaghir{$\kcal_0\submodelp\abold\kcal_1$ iff 
$\exists\,\kcal_2\ (\kcal_0\submodeli\abold\kcal_2\ \&\ \kcal_2\sim^\abold\kcal_1 )$.}
\end{lemma}
\begin{proof}
%
\sixtaghir{We first prove the left-to-right direction.}
The proof is almost identical to the proof of theorem 6.9 in 
\citep{Visser-Benthem-NNIL} and we refer the 
reader to it for more details. 
Let $\kcal_i=(W_i,\pce_i,\V_i)$ for $i\in\{0,1\}$ 
and $f:W_0\to W_1$ be the embedding of $\kcal_0$ in $\kcal_1$. 
Moreover we may assume that $f$ is surjective, 
otherwise we add a copy of $\kcal_1$ to $\kcal_0$ 
with a fresh root in beneath of them and then extend 
the embedding to the new nodes. 

Define $\kcal_2:=(W_2,\pce_2,\V_2)$ as follows:
\begin{itemize}
\item $W_2:=\{ (w_0, w_1): w_0\in W_0 \text{ and } f(w_0)\pce_1 w_1\in W_1\}$.
\item $(w_0, w_1)\pce_2 (w'_0, w'_1)$ iff 
either of the following holds:
\begin{itemize}
\item $w_0\pce_0 w'_0$ and $w_1=f(w_0)$,
\item $w_0=w'_0$ and $w_1\pce_1 w'_1$.
\end{itemize}
\item $(w_0, w_1)\V_2 a$ iff $w_1\V_1 a$. 
\end{itemize}
It is straightforward to show that $\kcal_2$ is a finite rooted 
tree-frame Kripke model with root 
$(\rho, f(\rho))$ 
in which $\rho$ is the root of $\kcal_0$.
Also 
$\kcal_1$ and $\kcal_2$ are bisimilar. 
Moreover one may easily show that $g$ as defined in the following, 
is an injective embedding of $\kcal_0$ into $\kcal_2$:
$g(w_0):=(w_0, f(w_0))$.

\sixtaghir{
For the proof of other direction, assume that  $\kcal_0\submodelp \abold \kcal_2\sim^\abold \kcal_1$ and $\kcal_i=(W_i,\pce_i,\V_i)$ for $i\in\{0,1,2\}$. 
Let $f:W_0\longrightarrow W_2$ be the injective function witnessing $\kcal_0\submodelp\abold \kcal_2$. By induction on $w\in W_2$ ordered by $\prec_2$, 
we also define the function $g:W_2\longrightarrow W_1$ such that 
$(\kcal_2)_w\sim^\abold (\kcal_1)_{g(w)}$ for every $w\in W_2$.
\begin{itemize}
\item As the basic step of induction, for the root $w$ of $W_2$, define $g(w)$ as the root of $W_1$.
\item For any  $w\in W_2$ other than  the root, we do as follows.
Let $w$ be the immediate successor of $v\in W_2$. This means that $v\prec_2 w$ and 
there is no $v'$ such that $v\prec_2 v'\prec_2 w$. By induction hypothesis, $g(v)$
is already defined and $(\kcal_2)_v\sim^\abold(\kcal_1)_{g(v)}$.
Hence there is some $u\in W_1$ such that $(\kcal_2)_w\sim^\abold(\kcal_1)_{u}$. 
Define $g(w)$ as some such $u$.
\end{itemize}
Then it can easily observed that the composition function 
$g\circ f:W_0\longrightarrow W_1$ is a function witnessing 
$\kcal_0\submodelp\abold \kcal_1$.
}
\end{proof}
Remember that $\NNIL(\abold)$ indicates formulas in $\NNIL$ which are  boolean combinations of atomic formulas in $\abold$. 
The following theorem is a Kripke semantical 
characterization of $\NNIL$ formulas 
\citep{Visser-Benthem-NNIL}. 
\begin{theorem}\label{Theorem-NNIL-Submodel}
Given $\abold\subseteq\atom$, we have
$A\in\NNIL(\abold)$ iff the class of Kripke 
models of $A$ is closed under $\submodeli\abold$.
\end{theorem}
\begin{proof}
See \citep{Visser-Benthem-NNIL} or 
\citep{Visser02}.
\end{proof}

\begin{lemma}\label{Remark-NNIL-finiteness}
Modulo $\IPC$-provable equivalence, 
$\NNIL(\abold)$ is finite.
\end{lemma}
\begin{proof}
Observe that each formula can be written 
as $\bigvee\bigwedge C$ in which $C$ is an atomic or implication, 
which we call it a component. 
We assume that disjunction and conjunction over empty set of formulas are defined as $\bot$ and $\top$, 
respectively. 
Observe that the number of formulas in $n$ atoms,
$f(n)$, is less than or equal to $2^{2^{g(n)}}$, in which $g(n)$ is the 
number of components in $n$ atoms. Then observe that 
$g(n+1)\leq (n+1)f(n)+ n+1$, because one may assume that 
each component is either of the form  
$p\to A$ for some 
atomic $p$  and some $A$ in $n$ atoms,  
 or  it is  atomic. Hence the following recursive function is an upperbound for the 
number of all formulas in $n$ atoms:

\[f(0):=2\quad \quad , \quad 
\quad 
f(n+1):=2^{2^{(n+1)(f(n)+1)}}\qedhere
\] 
\end{proof}
Define $\ttbrace{\kcal}_{_\Gamma}:=\{A\in\Gamma: \kcal\Vdash A\}.$
\begin{theorem}\label{Theorem-NNIL-Submodel2}
Let $\kcal,\kcal'$ be two  Kripke models and $\abold\subseteq\atom$. 
Then $\kcal'\Vdash \ttbrace{\kcal}_{\nnil(\abold)}$ iff 
$\kcal'\submodel\abold \kcal$.
\end{theorem}
\begin{proof}
See \citep[theorem~7.1.2]{Visser-Benthem-NNIL}.
\end{proof}
Given a substitution 
$\theta$ 
and a Kripke model $\kcal=(W,\pce,\V)$, we define 
$\theta(\kcal):=(W,\pce,\V')$  
as follows. 
For every atomic  $a$, 
define  $w\V' a$ iff $\kcal,w\Vdash \theta(a)$. 

\begin{lemma}\label{Lem-sub-swap}
Given a general substitution $\theta$,
Kripke model $\kcal$ and $A\in\lcalz$, we have 

$$\kcal,w\Vdash \theta(A) \quad \quad 
\text{iff} \quad \quad 
\theta(\kcal),w\Vdash A.
$$
\end{lemma}
\begin{proof}
Use induction on the complexity of $A$.
\end{proof}
\begin{remark}
Let $\theta$ and $\tau$ be two general substitutions
and $(\theta\tau)$ is their composition. 
Above lemma, implies that 
$(\theta\tau)(\kcal)=\tau(\theta(\kcal))$. This conflicts 
with our standard notation for the composition of 
functions. This confliction could be resolved
by choosing another name, e.g.~$\theta^*$ for the operation 
on Kripke models corresponding to $\theta$. 
However, for the simplicity of notations, we prefer not to do so.
\end{remark}

In the rest of this section
(\Cref{relative-ext-prop}), we define some notions for
the Kripke semantics. The rationale behind them 
becomes
clearer when one looks at the proof of 
\Cref{Theorem-Ghil-Ext}.

Given   Kripke models $\kcal$ and $\kcal'$, 
we say that $\kcal'$ is an
$\abold$-\textit{variant} of $\kcal$, if 
$\kcal'$ and $\kcal$ share the same frame and the same 
atomic valuations, except possibly at the root and only 
for atoms not in $\abold$, for which we may 
have different  valuations. 
\sixtaghir{In other words, $\abold$-\textit{variant} Kripke models should have the same valuation for atomics in  $\abold$.} 
We also say that $\kcal'$ is
a variant of $\kcal$, 
if it is $\emptyset$-variant of $\kcal$.
Moreover for a Kripke model $\kcal$ with the root $w_0$ we define 
$\kcal\vdashm  A$ iff $\kcal,w\Vdash A$ for 
every $w\neq w_0$.
\\
We say that  a formula $A$ is  \textit{$\abold$-subextendible}
if for every Kripke model $\kcal\Vdash A$ and every 
$\kcal'\submodeli\abold \kcal$ with $\kcal'\vdashm A$
there is an  $\abold$-variant $\kcal''$ of $\kcal'$
such that $\kcal''\Vdash A$. 
  
\sixtaghir{
\begin{example}
Every $A\in\NNIL(\abold)$ is $\abold$-subextendible. Furthermore, for $p\in\parr$
and $x\nin\parr$, the formula 
$B:=p\wedge x$ is $\parr$-subextendible while $C:=\neg\neg p \wedge x$ is not.
\end{example}
\begin{proof}
First statement is due to this fact that 
validity of $\NNIL(\abold)$-formulas are preserved under  $\abold$-submodels (see \Cref{Theorem-NNIL-Submodel}).

To see why $B$ is $\parr$-subextendible, 
let $\kcal\Vdash A$ and $\kcal'\submodeli \parr \kcal$ and $\kcal'\vdashm B$
seeking to find some $\kcal''\Vdash B$ which is $\parr$-variant of $\kcal'$.
Since $\kcal'\submodeli\parr \kcal\Vdash p$, we get $\kcal'\Vdash p$. 
However, it might not be the case that $\kcal'\Vdash x$, in which case we 
define $\kcal''$ by changing the valuation of $\kcal'$ on $x$ at the root. 
Since $\kcal'\vdashm B$, the resulted  model $\kcal''$ is a Kripke model 
indeed (in the sense of having monotonicity of truth on atomics).  
Moreover obviously $\kcal''$ is a $\parr$-variant of $\kcal'$ 
and $\kcal''\Vdash B$. This finishes showing $\parr$-subextendibility of $B$.

Finally to see why $C$ is not $\parr$-subextendible, consider 
$\kcal:=(W,\prec,\V)$ as follows. Define $W:=\{0,1\}$ with $0\prec 1$
 as the only  accessibility relation. Also 
$1\V p$ and $0,1\V x$ are only valuations for atomics.
Then we have $\kcal\Vdash C$. Moreover the model $\kcal':=(W',\prec',\V')$
defined as $W':=\{0\}$ and $\prec':=\emptyset$ and 
$0\V x$ as only atomic valuation. Then $\kcal'\vdashm C$ (actually $\kcal'\vdashm D$ for every $D$) while there is no $\parr$-variant $\kcal''$ of $\kcal'$ such that 
$\kcal''\Vdash C$ (note that any $\parr$-variant of $\kcal'$ should validate 
$\neg p$ and hence not validating $\neg\neg p$). 
\end{proof}
}  
  
For later applications in this paper it is helpful to define 
$\abold$-subextendibility also for a class of Kripke models. 
Let $\scrk$ is a \sixtaghir{set} of  Kripke models. 
Define the Kripke model $\sum(\scrk)$ as the disjoint union of all Kripke models in $\scrk$ with a fresh root 
$w_0$ such that for every atomic $a$ we have 
$\sum(\scrk),w_0\Vdash a$ iff $\scrk\Vdash a$.
Also for a
Kripke model $\kcal$ with
$\scrk\submodeli{\abold}\kcal$  (see \Cref{relative-ext-prop})
define $\sum(\scrk,\kcal)$ 
as   disjoint union 
of the Kripke models in $\scrk$ 
with a fresh root $w_0$  and following valuation for atoms
$a\in\abold$:

$$\sum(\scrk,\kcal),w_0\Vdash a\quad \text{iff}\quad 
\kcal\Vdash a. $$
We say that $\scrk$ is  
\textit{$\abold$-subextendible}, if
for every finite $\scrk'\subseteq \scrk$ 
with $\scrk'\submodeli\abold \kcal\in\scrk$, 
there is an $\abold$-variant of $\sum(\scrk',\kcal)$ which
belongs to $\scrk$. We say that $\scrk$ is \textit{extendible}
if  it is nonempty and $\emptyset$-subextendible.
One may easily observe that $A$ is $\parr$-subextendible iff $\Mod A$ is so.

\sixtaghir{
We have the following correspondence between two notions of $\abold$-subsextendibility.
\begin{lemma}\label{remark-parr-subext-semant}
$A\in\lcalz$ is $\abold$-subextendible iff $\Mod A$ 
is so. 
\end{lemma}
\begin{proof}
For the left-to-right: 
let   $\scrk\submodeli\abold\kcal\Vdash A$ and $\scrk$ is a finite set 
of models of $A$. Then obviously  $\sum(\scrk,\kcal)\vdashm A$ and hence by 
$\abold$-subextendibility of $A$, there is a $\abold$-variant 
of $\sum(\scrk,\kcal)\vdashm A$ which validates $A$.

For the right-to-left direction: 
let $\kcal\Vdash A$ and $\kcal'\submodeli\abold \kcal$ and $\kcal'\vdashm A$. 
Without loss of generality, we also may assume that $\kcal$ and $\kcal'$ has the same valuations for atomics in $\abold$, otherwise we may replace $\kcal_{f(w')}$
for $\kcal$ with $f$ as the function witnessing $\kcal'\submodeli\abold \kcal$
and $w'$ as the root of $\kcal'$. Obviously, we also have 
$\kcal_{f(w')}\Vdash A$ and $\kcal'\submodeli\abold \kcal_{f(w')}$. 
Let $w_i$ for $0\leq i\leq n$ be all the immediate successors\footnote{an immediate successor of a node is some accessible node such that there is no other nodes in between.}
of the root of $\kcal'$.  Define $\scrk:=\{\kcal'_{w_i}: 0\leq i\leq n\}$. Then 
$\scrk\submodeli\abold \kcal$ and $\scrk\Vdash A$.
Hence by $\abold$-subextendibility of $\Mod A$, there is an $\abold$-variant 
$\kcal''$ of $\sum(\scrk,\kcal)$ such that $\kcal''\Vdash A$. 
Since $\kcal'$ and $\kcal$ validate the same atomics in $\abold$, 
this implies that $\kcal''$ is also an $\abold$-variant of $\kcal'$.
\end{proof}
}

There is a well-known another notion in the literature
\citep{Ghil99}
for extendibility.  It says that $A$ is \textit{extendible} 
if 
 ``for every Kripke model $\kcal$ such that $\kcal\vdashm  A$ there is 
a variant $\kcal'$ of $\kcal$ such that $\kcal'\Vdash A$".  The following remark shows the relationship between our notion of $\abold$-subextendibility and extendibility.
\begin{remark}\label{remark-extendibility-emptyset}
A formula $A$ is extendible
 iff  it is satisfiable and 
$\emptyset$-subextendible.
\end{remark}
\begin{proof}
First assume that $A$ is extendible. It is obvious that 
then $A$ is $\emptyset$-subextendible. It remains only to show that $A$ is satisfiable.
Let $\kcal$
be an arbitrary single-node Kripke model. 
Then obviously $\kcal\vdashm  A$ and hence 
by extendibility,  there is a variant $\kcal'$ of 
$\kcal$ such that 
$\kcal'\Vdash A$. Thus $A$ is satisfiable. 
\\
For the other direction, assume that $A$
is satisfiable and $\emptyset$-subextendible. Also assume 
that $\kcal'$ is a Kripke model such that 
$\kcal'\vdashm  A$. Define the Kripke model 
$\kcal$ as follows. $\kcal$ and $\kcal'$ share the same frame and the atomic valuation of $\kcal$ at all nodes are the same and equal to the valuation which classically validates $A$ (which is guaranteed by satisfiability condition on $A$). Then we have 
$\kcal'\submodeli\emptyset \kcal$ and hence by 
$\emptyset$-subextendibility of $A$ we get some 
variant $\kcal''$ of $\kcal'$ such that $\kcal''\Vdash A$.
\end{proof}

%

\subsection{$\altNNILpar$-projectivity and $\parr$-extendibility}
\label{rel-ext-proj}
In this section we will prove \Cref{Theorem-Ghil-Ext},
an extension of  Ghilardi's 
characterization of projective formulas
via the notion of extendibility 
(see \Cref{Theorem-Ghil}). 
\\
For a formula $A$ and 
a set $\xbold\subseteq \varr$,
define the substitution 
$\txa:\lcalz \longrightarrow\lcalz$ as follows:

\begin{equation*}
\txa(x):=\begin{cases}
A\to x \quad  &: x\in\varr\cap\xbold\\
A\wedge x &: x\in\varr\setminus\xbold
\end{cases}
\end{equation*}

Let $\xbold_{\!1},\ldots,\xbold_{\!s}$ be  a list of all subsets 
of $\varr$ such that
$\xbold_{\!i}\subseteq\xbold_{\!j}$ implies
$i\leq j$.  Finally define

\begin{equation*}
\ta{} :=\ta{\xbold_{\!s}}\ta{\xbold_{\!s-1}}\ldots
\ta{\xbold_{\!1}} 
\end{equation*}
The following theorem, is the main preliminary tool 
provided in \citep{Ghil99} to characterize the  unification type of $\IPC$. We refer the reader to 
\citep[theorem 5]{Ghil99} for  its proof. 
We will  prove a generalization of this in \Cref{Theorem-Ghil-Ext}. 

\begin{theorem}\label{Theorem-Ghil}
For $A\in\lcalz$,
the following conditions are equivalent:
\begin{enumerate}
\item $\ta{}$ is a unifier for $A$, 
i.e.~$ \vdash\ta{}(A)$, 
\item $A$ is projective, 
\item $A$ is extendible.
\end{enumerate}
\end{theorem}

\begin{lemma}\label{Lem-Ghil-1}
If   $\ta{}(\kcal)\vdashm  A$ and $A$ is valid in  a 
$\parr$-variant of $\ta{}(\kcal)$ then  $\ta{}(\kcal)\Vdash A$. 
\end{lemma}
\begin{proof}
See proof of the theorem 5 in \citep{Ghil99}.
\end{proof}

Before we continue with a generalization 
of above theorem, let us give another definition, \sixtaghir{which is essentially a calculation of $\altNNILpar$-projection of $A$ (see \cref{Gamma-proj-nonmodal})}. 
Let $\kcal$ be  a Kripke model and 
$\pbold\subseteq\parr$.    
Then define $A^\ddagger$ as follows: 

\begin{equation*}
A^\ddagger:=
\bigwedge_{\pbold\subseteq\parr}\left(
\bigwedge\pbold\to \bigvee_{\kcal\Vdash \pbold,A} 
\bigwedge\tbnpr{\kcal}
\right)
\end{equation*}
Note that since by \Cref{Remark-NNIL-finiteness} the set  
$\NNILpar$ is finite and $\tbnpr{\kcal}\subseteq\NNILpar$, 
the conjunction $\bigwedge \tbnpr{\kcal}$ is 
a formula and also the disjunction may considered as a finite disjunction.

\sixtaghir{Remember that previously we defined 
$ A^\dagger $ for $ \altNNILpar $-projective $ A $ as the unique 
$ A'\in\altNNILpar $ such that $ A\xrat{}{} A' $. 
We will see in \Cref{Theorem-Ghil-Ext} 
that these two definitions are the same up to 
$ \IPC $-provable equivalence. Hence by \Cref{Gamma-projectivity-pres}
$A^\ddagger$ is the best $\NNILpar$-approximation of $A$ from above.}

\begin{remark}\label{dagger-dist-arrow}
	Note that by above definition, 
	if $  \vdash A\to B $ then $  \vdash A^\ddagger\to B^\ddagger $.
\end{remark}

\begin{theorem}\label{Theorem-Ghil-Ext}
	For $A\in\lcalz $, the following conditions are equivalent:
	\begin{enumerate}
		\item $A\xtra{\ta{}}{} A^\ddagger$,
		\item $A\xrat{}{} \altNNILpar$,
		\item $A$ is $\parr$-subextendible. 
	\end{enumerate}
\end{theorem}
\begin{proof}
	$1\to 2$: 
	From the definitions of $A^\ddagger$, evidently 
	$A^\ddagger\in\NNILpar$. 
	Also  observe that 
	$\txa$ is $A$-identity. Then  
	since $A$-identity 
	substitutions are closed under composition, 
	$\ta{}$ is $A$-identity.
	\\
	$2\to 3$: Let $A\xratth A'\in\NNILpar$  
	and $\kcal\Vdash A$ and  
	$\kcal'\submodeli\parr \kcal$ and   
	$\kcal'\vdashm  A$. We are
	seeking some variant $\kcal''$ of $\kcal'$ such that 
	$\kcal''\Vdash A$.  
	Let $\kcal''=\theta(\kcal')$. 
	First note that 
	since  $\theta$ is $A$-identity, $\kcal''$ is a 
	$\parr$-variant of $\kcal'$. 
	Since $\kcal\Vdash A'$,  
	$A'\in\NNILpar$ and $\kcal'\submodeli\parr\kcal$, 
	\Cref{Theorem-NNIL-Submodel}  implies that 
	$\kcal'\Vdash A'$. Hence 
	$\kcal'\Vdash \theta(A)$, and by 
	\Cref{Lem-sub-swap}
	we have $\kcal''\Vdash A$.
	\\
	$3\to 1$: Let $A$ be $\parr$-subextendible. We
	show $ \vdash A^\ddagger\lr \ta{}(A)$. 
	We use induction on the height of 	Kripke model $\kcal$
	and show $\kcal\Vdash A^\ddagger\lr \ta{}(A)$. 
	Suppose	that $w_0$ is the root of $\kcal$. 
	By the induction hypothesis 
	$\kcal\vdashm  A^\ddagger\lr \ta{}(A)$, 
	and since $A^\ddagger$ does not have any 
	atomic variables,  we have 
	$\kcal\vdashm  \ta{}(A^\ddagger\lr A)$.
	Then by \Cref{Lem-sub-swap}, 
	$\ta{}(\kcal)\vdashm  A^\ddagger
	\lr A$. We show  $\ta{}(\kcal)\Vdash A^\ddagger\lr A$.  
	If 	$\ta{}(\kcal)\nVdash^- A^\ddagger$, then 
	$\ta{}(\kcal)\nVdash^- A$ and hence 
	$\ta{}(\kcal),w_0\nVdash A^\ddagger$
	and 
	$\ta{}(\kcal),w_0\nVdash A$. 
	Then 
	$\ta{}(\kcal),w_0\Vdash A^\ddagger
	\lr A$ and we are done. 
	So assume that $\ta{}(\kcal)\vdashm  A\wedge A^\ddagger$. 
	It is sufficient to show the following items:
	\begin{itemize}[leftmargin=*]
		\item   $\ta{}(\kcal),w_0\Vdash A$ 
		implies 
		$\ta{}(\kcal),w_0\Vdash A^\ddagger$.  
		We use the definition of $A^\ddagger$. 
		Consider some $\pbold\subseteq\parr$.  
		Also assume that $\ta{}(\kcal),w\Vdash \pbold$ 
		for some node $w$ of $\kcal$.
		Then since $\ta{}(\kcal),w\Vdash \pbold$		
		and $\ta{}(\kcal),w\Vdash A$,   the following formula is one of the disjuncts in the definition of $A^\ddagger$:
		
		\begin{equation*}		
		\bigwedge \tbnpr{\ta{}(\kcal)_w},
		\end{equation*}
		in which $\ta{}(\kcal)_w$ is the  restriction of 
		$\ta{}(\kcal)$ to $w$ and its upward nodes.
		Obviously $\ta{}(\kcal),w\Vdash \bigwedge \tbnpr{\ta{}(\kcal)_w}$ and hence 
		$\ta{}(\kcal),w_0\Vdash A^\ddagger$.
		\item $\ta{}(\kcal),w_0\Vdash A^\ddagger$ implies 
		$\ta{}(\kcal),w_0\Vdash A$. 
		Let $\ta{}(\kcal),w_0\Vdash A^\ddagger$. 
		Also assume that 
		$\ta{}(\kcal)$ is $\pbold$-model,  
		i.e.~$\ta{}(\kcal),w_0\Vdash\pbold$ and 
		$\ta{}(\kcal),w_0\nVdash \bigvee (\parr\setminus\pbold)$. 
		Since $\ta{}(\kcal),w_0\Vdash A^\ddagger$, 
		for some 
		$\kcal_1\in\Mod{A}$ with $\kcal_1\Vdash \pbold$
		we have 
		$\ta{}(\kcal)\Vdash \tbnpr{\kcal_1}$.
		\Cref{Theorem-NNIL-Submodel2} implies that 
		$\ta{}(\kcal) \submodel\parr\kcal_1$.  Then 
		\Cref{Remark-embed-sub} implies 
		$\ta{}(\kcal) \submodelp\parr\kcal_1$
		and thus by \Cref{Remark-embed-sub2}
		there is some $\kcal_2\Vdash A$  such that 
		$\ta{}(\kcal) \submodeli\parr\kcal_2$.		
		Since $A$ is $\parr$-subextendible, there is a  
		$\parr$-variant $\kcal'$ of $ \ta{}(\kcal)$
		such that $\kcal'\Vdash A$.		 
		Thus
		\Cref{Lem-Ghil-1} implies  $\ta{}(\kcal)\Vdash A$.
		\qedhere
	\end{itemize}
\end{proof}
\begin{corollary}\label{Decid-nnilpar-proj}
$\altNNILpar$-projectivity is decidable. In other words,
given $A\in\lcalz$, one may algorithmically decide $A\in\altdnnilpar$.
\end{corollary}
\begin{proof}
Given $A$, by \Cref{Theorem-Ghil-Ext} 
it is sufficient to decide $\IPC\vdash \ta{}(A)\lr A^\ddagger$, which 
is decidable since $\IPC$ is decidable.
\end{proof}

\subsection{Projective resolution}\label{proj-res}
The main result in \citep{Ghil99} is that 
the unification type 
of $\IPC$ is finitary. It means that for every 
$A\in\lcalzvar$, there exists 
a finite \textit{complete set of unifiers} for $A$, i.e.~a finite set 
$\Theta$ of unifiers for $A$ such that
every unifier of $A$ is less general than 
some $\theta\in \Theta$.  We say that $\theta$ is less general than $\gamma$ if there is some substitution 
$\lambda$ such that for every $x\in\varr$ we have 

\begin{equation*} 
\vdash \theta(x)\lr\lambda(\gamma(x)).
\end{equation*}
The proof of the above mentioned fact is based on projective approximations. 
Later  \citep{ghilardi2002resolution} provides a resolution/tableaux method for computation of the projective approximations.
The aim for this subsection is providing
a relativised version 
of projective approximations
in \Cref{Theorem-IPC-nnilp-Finitary}.
\begin{definition}\label{def-proj-res}
Given  $\Gamma,\Pi\subseteq\lcalz$ and $A\in\lcalz$, 
we say that $\Pi$  is 
$\Gamma$-projective resolution for $A$ if 
\begin{itemize}
\item $ \Pi $ is a \sixtaghir{finite} set of independent formulas, i.e.~for 
$ B,C\in\Pi $, $ \vdash B\to C $ implies $ B=C $. 
\item Every $B\in \Pi$ is $\Gamma$-projective.
\item $ A\arg\bigvee\Pi $. 
\item   $ \vdash \bigvee\Pi\to A$.
\end{itemize}
A  $\{\top\}$-projective resolution is also called 
projective resolution.
\end{definition}
\noindent  
Note that $\emptyset $ is a projective resolution for 
a formula which is not unifiable.
The greatest lower bound (glb) for a formula $A$ is defined in \Cref{glb}. Intuitively 
a glb for $A$ w.r.t.~$(\Gamma,\sft)$ is the best 
$\Gamma$-approximation from below inside the logic $\sft$.
\begin{remark}\label{remark-resol-glb}
If $\Pi$ is $\Gamma$-projective resolution of $A$ then 
$\bigvee\Pi$ is a glb for $A$ w.r.t.~$(\dgv,\IPC)$.
\end{remark}
\begin{proof}
By \Cref{pres-admis-rel} we have $A\prtdg\bigvee\Pi$ and hence by \Cref{vee-pres} we have $A\prtdgv\bigvee\Pi$.
Thus
\Cref{Gamma-approx-preserv} implies desired result.
\end{proof}
\begin{theorem}\label{Theorem-IPC-Finitary}
Whenever  $\parr=\emptyset$,
every  $A\in\lcalz $ has   projective
resolution. 
\end{theorem}
\begin{proof}
See \citep[theorem 5]{Ghil99}.
We will also prove a generalization of this
result in \Cref{Theorem-IPC-nnilp-Finitary}. 
\end{proof}

\noindent
First some preliminary definitions. We refer the reader 
to \citep{Ghil99} for more information on these notions.

Let 
$\kcal=(W,\pce,\V)$ and $\kcal'=(W',\pce',\V')$
are two Kripke models with the roots 
$w_0$ and $w'_0$. Also let 
$\kcal(w):=\{a\in\atom : \kcal,w\Vdash a\} $ and  $\lcalz(\kcal)$ be defined as $\lcalz(\bigcup_{w\in W}\kcal(w))$.
We say that  $\kcal$ \textit{has finite valuations} if for every 
$w\in W$ we have $\kcal(w)$ is finite. Also define:

\begin{align*}
\kcal\sim_0 \kcal' \quad &\text{iff}\quad 
\kcal(w_0)=\kcal'(w_0')\\
\kcal\sim_{n+1} \kcal' \quad &\text{iff}\quad 
\fa w\in W\,\ex w'\in W' (\kcal_w\sim_n\kcal'_{w'})
\text{ and  vice versa}\\
\kcal\leq_0\kcal' \quad &\text{iff} \quad 
\kcal(w_0)\supseteq
\kcal'(w'_0)\\
\kcal\leq_{n+1}\kcal' \quad 
&\text{iff} \quad 
\fa w\in W\,\ex w'\in W' (\kcal_w\sim_n\kcal'_{w'})
\end{align*}
Evidently $\sim_n$ is an equivalence relation and 
$\leq_n$ is reflexive transitive. One may easily 
observe by induction on $n$ that $\kcal\sim_{n+1}\kcal'$
implies $\kcal\sim_n\kcal'$. Hence 
$\kcal\sim_n\kcal'$ ($\kcal\leq_n\kcal'$) implies 
$\kcal\sim_m\kcal'$ ($\kcal\leq_m\kcal'$) for every $m\leq n$. 
\\
Let $\cto(A)$ indicate the maximum number of nested implications in $A$:
\begin{itemize}
\item $\cto(a)=\cto(\top)=\cto(\bot)=0$ for atomic $a$.
\item $\cto(A\circ B):=\max\{\cto(A),\cto(B)\}$, for $\circ\in\{\vee,\wedge\}$.
\item $\cto(A\to B):=1+\max\{\cto(A),\cto(B)\}$.
\end{itemize}
Remember that by default we assume the set $\atom$ to be a finite set. 
\begin{remark}\label{Remark-Finiteness-c(A)}
Modulo $\IPC$-provable equivalence, there are finitely many 
formulas $A\in \lcalz$ with $\cto(A)\leq n$.
\end{remark}
\begin{proof}
By induction on $n$, we define an upper bound $f(n)$
for the number of formulas $A\in\lcalz$ with 
$\cto(A)\leq n$. 
\begin{enumerate}[leftmargin=*]
\item $f(0):$ Observe that any $A$ with $\cto(A)=0$ is $\IPC$-equivalent to a disjunction of conjunctions of atoms. 
Hence $f(0)=2^{2^m}$ is an obvious upper bound, 
in which $m$ is the number of atoms in $\atom$.
\item $f(n+1):$ For every implication $B\to C$
with $\cto(B\to C)\leq n+1$, we have $\cto(B),\cto(C)\leq n$,
and hence $f(n)^2$ is an upper bound for the number of 
inequivalent such formulas. Then since 
\mipc
every formula is a disjunction of conjunctions of atoms or implications, the following definition is an upperbound:

\begin{equation*}
f(n+1):=2^{2^{[m+f(n)^2]}}.\qedhere
\end{equation*}
\end{enumerate}
\end{proof}

\begin{lemma}\label{Lem-nnilpn-c(n)}
Every $A\in\NNIL$ has an $\IPC$-provable equivalent 
$A'\in\NNIL$ with $\cto(A')\leq \#\atom$. 
\end{lemma}
\begin{proof}
\sixtaghir{First we show that} every $A\in\NNIL$ has an $\IPC$-equivalent $B\in\NNIL$ 
such that \sixtaghir{in $B$ only atoms of $A$ are appeared and} $B=\bigwedge_i\bigvee_j B_i^j$
and every implication in 
$B^j_i$ is of the form $a\to C$, with $a\in\atom$ and 
$C$ does not contain $a$.  
\sixtaghir{We will prove the above claim by induction on complexity of $A$. The cases for atomic, conjunction and disjunction are easy and left to the reader. So assume that $A=E\to F$. Since $A\in\NNIL$, $E$ does not have implications. 
Hence it is equivalent to some formula of the form $\bigvee\Gamma $ such that every $D\in\Gamma$ is of the form $\bigwedge \Delta_D$ and $\Delta_D$ is a non-empty set of atomic formulas for every $D\in \Gamma$. 
Thus $A$ is equivalent to $\bigwedge_{D\in \Gamma}(\bigwedge \Delta_D\to F)$. 
For every $D\in \Gamma$, fix some  $p_D\in \Delta_D$ and let 
$\Delta'_D:=\Delta_D\setminus\{p_D\}$. Hence $A$ is equivalent to 
$\bigwedge_{D\in\Gamma}(p_D\to (\bigwedge\Delta'_D\to F[p_D:\top]))$, in which 
$F[p_D:\top]$ is the replacement of $\top$ for $p_D$ in $F$. This formula is of required shape and we are done.} 

Then one may easily prove the statement of this lemma by induction on the number of elements in $\atom$: \sixtaghir{by induction hypothesis, for every $B^i_j=p\to C$, we have $\cto(C)\leq \#\atom -1$ and hence $\cto(B^i_j)\leq\#\atom$. This implies 
$\cto(B)\leq \# \atom$.}
\end{proof}

\begin{lemma}\label{Lem-Characteristic-Kripke}
For every  Kripke model $\kcal$, 
there exists a formula $\charkn\in\lcalz$
 with the following properties:
\begin{itemize}
\item $\kcal'\Vdash \charkn$ iff $\kcal'\leq_n \kcal$.
\item $\cto(\charkn)\leq n$.
\end{itemize}
\end{lemma}
\begin{proof}
Here we only give the definition of $\charkn $ by induction on $n$, and refer the 
reader to \citep[proposition 1]{Ghil99} for its proof.
\\
Let
$\kcal=(W,\pce,\V)$
and define $\charac{\kcal}{0}:=\bigwedge\kcal(w_0)$ and 

\begin{equation*}
\charac{\kcal}{n+1}:=\bigwedge_{\{\kcal':\fa w\in W(\kcal'\nsim_n\kcal_w)\}} 
\left(
\charac{\kcal'}{n}\to 
\bigvee_{\{\kcal'': \kcal'\nleq_n\kcal''\}}
\charac{\kcal''}{n} 
\right)
.\qedhere
\end{equation*}
\end{proof}

\begin{corollary}\label{Cor-Kripke-bisim-property}
For every Kripke models $\kcal$ and $\kcal'$, 
we have $\kcal'\leq_n\kcal$ iff 
for every $A\in\lcalz$ with $\cto(A)\leq n$ we have 
$\kcal\Vdash A$ implies $\kcal'\Vdash A$. 
\end{corollary}
\begin{proof}
For the left to right direction, use induction on $n$. 
For the right to left, \sixtaghir{use \Cref{Lem-Characteristic-Kripke}.}
\end{proof}
\begin{corollary}\label{Corollary-n-sim-Kripke}
$\kcal'\sim_n\kcal$ iff for every $A$ with $\cto(A)\leq n$
we have 
\begin{center}
$\kcal\Vdash A$ iff $\kcal'\Vdash A$.
\end{center}
\end{corollary}
\begin{proof}
First observe that $\kcal\sim_n \kcal'$ is equivalent to 
$\kcal\leq_n\kcal'\leq_n\kcal$ and then use 
\Cref{Cor-Kripke-bisim-property}.
\end{proof}
\begin{lemma}\label{Lem-Mod(A)-Kripke}
A class  $\scrk$ of Kripke models 
is of the form $\Mod{A}$ 
with $\cto(A)\leq n$, iff $\scrk$ is $\leq_n$-downward 
closed, i.e.~for every Kripke model $\kcal'$ 
with 
$\kcal'\leq_n\kcal\in\scrk$  we have $\kcal'\in\scrk$.
\end{lemma}
\begin{proof}
For the left-to-right direction, let $\kcal'\leq_n\kcal\in\Mod{A}$ for some $A$ with $\cto(A)\leq n$. Since 
$\kcal\Vdash A$, \Cref{Cor-Kripke-bisim-property} implies 
$\kcal'\Vdash A$ and hence $\kcal'\in\Mod{A}$. 
\\
For the other direction, let $\scrk$ be 
$\leq_n$-downward closed and define 

\begin{equation*}
A:=\bigvee_{\kcal\in\scrk}\charac{\kcal}{n}.
\end{equation*}
By 
\Cref{Remark-Finiteness-c(A)}, the disjunction is finite 
and hence $A$ is indeed a formula.  
One may easily observe that \Cref{Lem-Characteristic-Kripke} implies that 
$\cto(A)\leq n$  and 
$\scrk=\Mod{A}$.
\end{proof}

\begin{lemma}\label{Lem-extendible}
If a class of Kripke models $\scrk$ is $\parr$-subextendible and $\theta$
is a substitution, then $\theta(\scrk)$ 
is also $\parr$-subextendible.
\end{lemma}
\begin{proof}
Easy and left to the reader.
\end{proof}
\noindent 
We say that a class $\scrk$ of Kripke models 
is stable, if for every $\kcal\in \scrk$ and every node 
$w$ in $\kcal$ we have $\kcal_w\in\scrk$. 
\begin{remark}\label{Remark-stable}
A class $\scrk$ of Kripke models
 is $\parr$-subextendible iff for 
every finite stable 
class of models $\scrk'$ which is $\parr$-submodel of 
some $\kcal\in\scrk$, a $\parr$-variant of 
$\sum(\scrk',\kcal)$ belongs to $\scrk$.
\end{remark}
\begin{proof}
Easy and left to the  reader.
\end{proof}
\noindent Define 
$\scrkn:=\{\kcal: \ex\kcal'\in\scrk\ 
(\kcal\leq_n\kcal')\text{ and } \kcal 
\text{ is  a Kripke model}\}$.
\begin{lemma}\label{Lem-stable-extendible}
If $\scrk$ is $\parr$-subextendible
stable class of   Kripke models, then so is 
$\scrkn$, for every $n> \#\parr$. \uparan{$\#\parr$ indicates the number of elements in $\parr$.}
\end{lemma}
\begin{proof}
We only prove here the $\parr$-subextendibility
 of $\scrkn$ and leave other properties to the reader.\\
Let $\scrf'=\{\kcal'_i\}_i$ be a finite set of models 
in $\scrkn$, 
which are  $\parr$-submodels of some $\kcal'\in\scrkn$.
By \Cref{Remark-stable} we may also assume 
that $\scrf'$ is stable.
We must show that a $\parr$-variant of 
$\sum(\scrf',\kcal')$ belongs to $\scrkn$.
Since $\kcal'\in\scrkn$ and $\scrk$ is stable, 
there is some 
$\kcal\in\scrk$ such that $\kcal'\sim_{n-1}\kcal$. 
Similarly, since $\kcal'_i\in\scrkn$, there is some 
$\kcal_i\in\scrk$
such that $\kcal'_i\sim_{n-1}\kcal_{i}$. Let 
$\scrf:=\{\kcal_i\}_i$. 
\\
First we show that  $\scrf$ is a 
$\parr$-submodel of $\kcal$. 
Since $\scrf'$ is a $\parr$-submodel of $\kcal'$, 
by \Cref{Theorem-NNIL-Submodel2} 
we have $\kcal'_i\Vdash \tbnpr{\kcal'}$. 
From $\kcal'_i\sim_{n-1}\kcal_i$, 
\Cref{Lem-nnilpn-c(n),Corollary-n-sim-Kripke} we get 
$\kcal_i\Vdash \tbnpr{\kcal'}$. Also 
since $\kcal'\sim_{n-1}\kcal$, by 
\Cref{Lem-nnilpn-c(n),Corollary-n-sim-Kripke} we have 
$ \tbnpr{\kcal'}=\tbnpr{\kcal}$. Hence $\kcal_i\Vdash \tbnpr{\kcal}$, and by \Cref{Theorem-NNIL-Submodel2}
we have $\kcal_i$ is a $\parr$-submodel of 
$\kcal$. Hence $\scrf$ is a $\parr$-submodel of 
$\kcal$.  \\
We go back to the main proof. 
Since $\scrf$ is 
$\parr$-submodel of $\kcal$, by extendibility 
of $\scrk$, there exist  a $\parr$-variant $\kcalh$
of $\sum(\scrf,\kcal)$ in $\scrk$. 
Let $w_0$ is the root of $\kcal$ 
which is also the root of $\kcalh$ and  
$w'_0$ is the root of $\kcal'$. 
Define the $\parr$-variant $\kcalh'$ of 
$\sum(\scrf',\kcal')$ 
for atomic  $x\not\in \parr$ as follows:

\begin{equation*}
\kcalh',w'_0\Vdash x \quad \Longleftrightarrow \quad 
\kcalh,w_0\Vdash x.
\end{equation*}
It is sufficient to show that $\kcalh'\in\scrkn$. 
For this aim it is sufficient to show 
$\kcalh'\leq_n \kcalh$. From the definition of 
$\kcalh$ and $\kcalh'$, it is clear that 
it is sufficient to show that $\kcalh'\sim_{n-1}\kcalh$.
We use  induction on 
$k\leq n-1$  and show $\kcalh'\sim_{k}\kcalh$.
\\
If $k=0$, we must show that for every atomic $a$
we have 

\begin{equation*}
\kcalh',w'_0\Vdash a \quad \Longleftrightarrow \quad 
\kcalh,w_0\Vdash a.
\end{equation*}
For atomic variables  $x$, by definition of $\kcalh'$, we already have this. 
Also since $\kcalh$   is a $\parr$-variant of 
$\sum(\scrf,\kcal)$,
$\kcalh'$ is a $\parr$-variant of 
$\sum(\scrf',\kcal')$ and $\kcal\sim_{0}\kcal'$, 
for every $p\in\parr$ we also have 

\begin{equation*}
\kcalh',w'_0\Vdash p \quad \Longleftrightarrow \quad 
\kcalh,w_0\Vdash p.
\end{equation*}
Then let $0<k<n$ and show $\kcalh'\sim_k\kcalh$. 
We have the following items to prove:
\begin{itemize}[leftmargin=*]
\item For every  node $w'$ in $\kcalh'$, there is some 
$w$ in $\kcalh$ such that $\kcalh'_{w'}
\sim_{k-1}\kcalh_w$. If $w'$ is the root of $\kcalh'$, 
take $w$ also the root of $\kcalh$
and we have desired result by
induction hypothesis. If $w'$ is not the root of 
$\kcalh'$, since $\scrf'$ is stable, we may let 
$w'$ as a root $w'_i$ of some $\kcal'_i$. 
Take $w=w_i$.
Then by definition of $\kcal_i$, we have 

\begin{equation*}
\kcalh'_{w'}=\kcal'_i\sim_{n-1}\kcal_i=\kcalh_{w}
\end{equation*}
Since $k-1\leq n-1$, we have the desired result. 
\item For every  node $w$ in $\kcalh$, there is some 
$w'$ in $\kcalh'$ such that $\kcalh'_{w'}
\sim_{k-1}\kcalh_w$. Again if $w$ is the root, take $w'$
also the root and we are done by induction hypothesis.
If $w$ is not the root, there is some $i$ such that 
$w$ is a node of $\kcal_i$. Since 
$\kcal_i\sim_{n-1}\kcal'_i$, there is some $w'$ in 
$\kcal'_i$ such that 
$(\kcal'_i)_{w'}\sim_{n-2}(\kcal_i)_w$. Since 
$k-1\leq n-2$, we have 

\begin{equation*}
\kcalh'_{w'}=(\kcal'_i)_{w'}\sim_{k-1}(\kcal_i)_w
=\kcalh_w,
\end{equation*}
as desired.\qedhere
\end{itemize}
\end{proof}
\begin{theorem}\label{Theorem-IPC-nnilp-Finitary}
Every  $A\in\lcalz$ has  
$\altNNILpar$-projective resolution $\Pi$. 
Moreover for every $B\in\Pi$ we have $\cto(B)\leq
\max\{\cto(A),1+\#\parr\}$ and $\Pi$ is a computable 
function of $A$. 
\end{theorem}
\begin{proof}
Given a substitution $\theta$
and $A'\in\NNILpar$ such that 
$ \vdash A'\to \theta(A)$,
we will find some $\bat\in\lcalz$,
with the following properties:
\begin{enumerate}
\item $\bat$ is $\altNNILpar$-projective.
\item $ \vdash A'\lr \theta(\bat)$.
\item $ \vdash \bat\to A$.
\item $\cto(\bat)\leq n$ for   $n:=\max\{\cto(A),1+\#\parr\}$ ($\#\parr$ indicates the number of atoms in $\parr$). 
\end{enumerate}
Then by items 1-3 (an independent subset of) the following set is a 
 $\altNNILpar$-projective resolution for $A$:

\begin{equation*}
\Pi:=\{
\bat: 
A'\in\NNILpar \text{ and $\theta$ a substitution such that }
 \vdash A'\to \theta(A)
\}.
\end{equation*}
Moreover \Cref{Remark-Finiteness-c(A)} and 
item (4) implies that $\Pi$ is finite, as desired.
So it remains to find  $\bat$
with mentioned properties.
Define 

\begin{equation*}
\scrk:=\theta(\Mod{A'}):=\{\theta(\kcal): \kcal\in\Mod{A'}\}.
\end{equation*}
Since $\scrkn$ has downward $\leq_n$-closure condition, 
we may apply  \Cref{Lem-Mod(A)-Kripke} 
and find some formula, e.g.~$\bat$,
such that $\cto(\bat)\leq n$ (so item 4 is satisfied)
and $\scrkn=\Mod{\bat}$. 
Since $A'\in\NNILpar$, evidently it is 
$\altNNILpar$-projective. Hence by \Cref{Theorem-Ghil-Ext}, 
$A'$ is $\parr$-subextendible. 
Hence by \Cref{Lem-extendible}, 
$\scrk$ is $\parr$-subextendible. 
Since $\scrk$ is stable
and $n> \#\parr$, \Cref{Lem-stable-extendible}
implies that $\scrkn$ is also $\parr$-subextendible.
Hence $\bat$ is $\parr$-subextendible and by 
\Cref{Theorem-Ghil-Ext}, $\bat$ is $\NNILpar$-projective. 
So item (1) is satisfied.
\\
To show item 3 for $\bat$, 
it is sufficient to show $\kcal\Vdash \bat\to A$ for every finite rooted model $\kcal$. If $\kcal\Vdash \bat$, we have $\kcal\in\scrkn$. Hence $\kcal\leq_n \kcal'$ for 
some $\kcal'\in\scrk$. Then $\kcal'=\theta(\kcal'')$
for some finite rooted $\kcal''$ such that 
$\kcal''\Vdash A'$. Since $\vdash A'\to \theta (A)$, 
we have $\kcal''\Vdash \theta(A)$, and by 
\Cref{Lem-sub-swap} we get $\theta(\kcal'')\Vdash A$,
whence $\kcal'\Vdash A$. Since 
$\cto(A)\leq n$ and $\kcal\leq_n\kcal'$, 
\Cref{Cor-Kripke-bisim-property} implies that
$\kcal\Vdash A$, as desired.
\\
It remains to show that item 2 holds.  
It is sufficient to show  $\kcal\Vdash 
A'\lr \theta(\bat)$ for arbitrary finite rooted $\kcal$.
If $\kcal\Vdash A'$, then $\theta(\kcal)\Vdash A'$ and 
hence $\theta(\kcal)\in\scrk\subseteq\scrkn=\Mod{\bat}$. 
Then  $\theta(\kcal)\Vdash \bat$ and hence $\kcal\Vdash
\theta(\bat)$. For the other direction, let 
$\kcal\Vdash \theta(\bat)$. Hence $\theta(\kcal)\Vdash \bat$ and then $\theta(\kcal)\in\Mod{\bat}=\scrkn$. 
So there is some $\kcal'\in\scrk$ such that 
$\theta(\kcal)\leq_n\kcal'$.  Since $\kcal'\in\scrk$,
there is some $\kcal''$ such that $\kcal'=\theta(\kcal'')$ and $\kcal''\Vdash A'$. Since 
$A'=\theta(A')$, we have $\kcal''\Vdash \theta(A')$
and hence $\kcal'\Vdash A'$. By 
\Cref{Lem-nnilpn-c(n)} $\cto(A')<n$ and the 
\Cref{Cor-Kripke-bisim-property} implies $\kcal\Vdash A'$.

Finally we provide an algorithm which computes 
$\Pi$. Given $A$, compute the finite set  

\begin{equation*}
\Pi':=\{B\in\lcalz: \cto(B)\leq\max\{\cto(A),1+\#\parr\} \text{ and } \vdash B\to A \text{ and } B\in\altdnnilpar\}.
\end{equation*}
Note that $\Pi'$ is computable since $\IPC$ is decidable and 
by \Cref{Decid-nnilpar-proj} we can decide 
$B\in\altdnnilpar$. Finally one may easily find $\Pi\subseteq\Pi'$ 
which includes pairwise $\IPC$-independent formulas, as required 
for projective resolutions.
\end{proof}

\subsection{Projective resolution for \NNIL}\label{NNIL-resol}
In this subsection, we will  see that 
 projective resolution of 
a $\NNIL$-formula gets a more elegant form. 
We will use this form later for  characterization of 
 $\altNNILpar$-admissible rules  of $\IPC$, specifically 
 for the validity of disjunction rule. 
By \Cref{Theorem-IPC-Finitary} or equivalently
\Cref{Theorem-IPC-nnilp-Finitary} with empty $\parr$,
there is a finite  projective resolution for every formula $A$, i.e.~a set $\{A_1,\ldots ,A_n\}$,  
with the following properties:
\begin{itemize}
\item Every unifier of $A$, is also a unifier of some 
$A_i$, in other words $A\ar \bigvee A_i$. 
\item $ \vdash \bigvee A_i\to A$.
\item $A_i$  is projective for every $i\leq n$.
\end{itemize}
We will prove here that if $A\in\NNIL$, the projective 
resolution can be chosen such that 
every $A_i$ is $\NNIL$ and moreover $ \vdash A\lr 
\bigvee A_i$. 
\\
Given $A\in\NNIL$ \sixtaghir{and a set $X$ of formulas}, we say that $A$ is an   
$X$-component if $A=\bigwedge \Gamma\wedge\bigwedge\Delta $
with the following properties:
\begin{itemize}
	\item Every $B\in \Gamma$ is atomic.
	\item Every $B\in\Delta$ is an implication
	$C\to D$ for some atomic   $C$ such that 
	$X\nvdash \bigwedge\Gamma\to C$.
\end{itemize}
\sixtaghir{For simplicity reasons, $\empty$-components are also called components.}

\begin{lemma}\label{ipc-component-dec} 
	Given  a set $X\subseteq\lcalz$,
	every $A\in\NNIL$  can be decomposed to $X$-components, 
	i.e.~there is a finite set of 
	$X$-components $\Gamma_A$ such that 
	$X\vdash A\lr \bigvee\Gamma_A$. Moreover,
	 if $A\in\NNILa$ then $\Gamma_A\subseteq\NNILa$.
\end{lemma}
\begin{proof}
We use induction on 
$\suba A$ (the set of atomic 
formulas in $A$) ordered by $\supset$ and find some finite set $\Gamma_A$ of 
$X$-components with 
$\suba{\Gamma_A}\subseteq\suba A$ and $X\vdash \bigvee\Gamma_A\lr A$. 
As induction hypothesis assume that for every $\sft$ and 
$B\in\NNIL$ with 
$\suba B\subset \suba A$ there is a finite set 
$\Gamma_B$ of 
$X$-components such that $X\vdash B\lr \bigvee\Gamma$
and $\suba{\Gamma_B}\subseteq\suba B$.  
For the induction step, assume that $A\in\NNIL$ is given.
Using derivation in $\IPC$ one may easily find finite sets 
$\Gamma_i$ and $\Delta_i$ for $1\leq i\leq n$ such that 
\begin{itemize}
\item $\IPC\vdash A\lr \bigvee_{i=1}^n A_i$,
in which $A_i:=\bigwedge\Gamma_i\wedge\bigwedge\Delta_i$.
\item $\Delta_i$ includes only atomic   formulas.
\item $\Gamma_i$ includes implications  
with atomic  antecedents.
\item $\suba{\Gamma_i\cup\Delta_i}\subseteq \suba A$.
\end{itemize}
It is sufficient to decompose every $A_i$ to $X$-components. 
If $X\nvdash \bigwedge\Delta_i\to E$ for every antecedent 
$E$ of an implication in $\Gamma_i$, then $A_i$ already is a 
$X$-component and we are done. Otherwise, there is some 
$E\to F\in\Gamma_i$ such that $X\vdash \bigwedge\Delta_i\to E$.
Then let  $A_i':=A_i[E:\top]$, i.e.~the replacement of every 
occurrences of $E$ in $A_i$ with $\top$. Also let $X':=X\cup\{E\}$.
Hence $\suba{A'_i}\subsetneqq \suba{A}$ and by induction hypothesis 
we may decompose $A'_i$ to $X'$-components:

\begin{equation*}
X'\vdash A'_i\lr \bigvee_j B_j
\end{equation*}
It is not difficult to observe that if $B_j$ is a $X'$-component then 
$B'_j:=E\wedge B_j$ is a $X$-component. Moreover 
$X\vdash E\wedge A'_i\lr \bigvee_j B'_j$ and since 
$\IPC\vdash (E\wedge A'_i)\lr (E\wedge A_i) $ and $X\vdash A_i\to E$, 
we get 

\begin{equation*}
X\vdash A_i\lr \bigvee_j B'_j.
\end{equation*}
Hence we have decomposed $A_i$ to $X$-components $B'_j$ with 
$\suba{B'_j}\subseteq\suba{A}$, as desired.
\end{proof}
\begin{lemma}\label{Lemma-NNIL-normal-projective}
Components are extendible.
\end{lemma}
\begin{proof}
Let $B=\bigwedge B_i$ is a component and 
$\scrk$ be a finite set of finite rooted Kripke models
for $B$. We must show that a variant of 
$\sum(\scrk)$ is a model of $B$.
Let $w_0$ be the root of $\sum(\scrk)$ and 
define a variant $\kcal$ of $\sum(\scrk)$ as follows. 
$\kcal,w_0\Vdash a$ iff $a=B_i$ for some $i$. 
Then it is easy to observe that 
$\kcal,w_0\Vdash B$.
\end{proof}
\begin{corollary}\label{Corollary-proj-res-NNIL}
For $A\in\NNIL$ there is a finite set $\Delta$
of projective
and $\NNIL$ formulas with 

\begin{equation*} 
\vdash A\lr \bigvee \Delta.
\end{equation*}
\end{corollary}
\begin{proof}
Use \Cref{Lemma-NNIL-normal-projective,ipc-component-dec,Theorem-Ghil}.
\end{proof}

\begin{lemma}\label{Remark-extendible-disjun}
Every extendible $A$ is prime, i.e.~if 
$ \vdash A\to (B\vee C)$, then 
either $ \vdash A\to B$ or $ \vdash A\to C$. 
\end{lemma}
\begin{proof}
We prove this by contraposition. Let $ \nvdash A\to B$
and $ \nvdash A\to C$. Then there are some Kripke 
models $\kcal_1 $ and $\kcal_2$ such that 
$\kcal_1\Vdash A$, $\kcal_1\nVdash B$, 
$\kcal_2\Vdash A$ and $\kcal_2\nVdash B$. Since 
$A$ is extendible, there is some variant 
$\kcal'$ of $\sum(\{\kcal_1,\kcal_2\})$ 
such that $\kcal\Vdash A$. 
Since 
$\kcal_1\nVdash B$ we have 	 $\kcal\nVdash B$. Similarly 
$\kcal\nVdash C$. Hence $\kcal\nVdash B\vee C$ 
and then $\kcal\nVdash A\to (B\vee C)$. 
\end{proof}

\begin{theorem}\label{component-extendible-prime}
Given \sixtaghir{$A\in\NNIL$}, the following are equivalent:
\begin{enumerate}
\item $A$ is a component, \uparan{modulo $\IPC$-provable equivalence}
\item $A$ is extendible,
\item $A$ is prime.
\end{enumerate}
\end{theorem}
\begin{proof}
$1\Rightarrow 2$: \Cref{Lemma-NNIL-normal-projective}.
$2\Rightarrow 3$: \Cref{Remark-extendible-disjun}.
$3\Rightarrow 1$:   Let $A$ is   prime. 
By \Cref{ipc-component-dec} it can be decomposed to 
components $\Gamma_A$.  
Thus $\vdash A\lr \bigvee\Gamma_A$ and by  primality of $A$
we have $\vdash A\to B$ for some $B\in\Gamma_A$. Then $\vdash A\lr B$
and hence $A$ is $\ipc$-equivalent to some component.
\end{proof}

Remember that $\pNNILpar$ indicates the set of prime 
and $\NNILpar$-formulas.
\begin{corollary}\label{Lem-nnil-normal-form}
Up to $\IPC$-provable equivalence, we have 
$\NNIL=\pNNIL^\vee$ and $\NNILpar=\pNNILpar^\vee$.
\end{corollary}
\begin{proof}
By \Cref{ipc-component-dec} every $A\in\NNIL$ can be decomposed to 
components $ \Gamma_A$ such that 
$A\in\NNILpar$ implies 
$\Gamma_A\subseteq\NNILpar$. Then \Cref{component-extendible-prime}
implies that every $E\in\Gamma_A$ is  prime. Hence 
$\bigvee\Gamma_A\in\pNNIL^\vee$ and moreover 
$A\in\NNILpar$ implies $\bigvee\Gamma_A\in\pNNILpar^\vee$.
\end{proof}

\begin{corollary}\label{pnnilpv-nnilp}
${\altarpn}={\altarn}\  $.
\end{corollary}
\begin{proof}
\Cref{Lem-nnil-normal-form,vee-pres}.
\end{proof}
A consequence of the results in this subsection is that now 
we have uniqueness of the projective resolutions:
\begin{theorem}[\bf Projective Resolution]\label{Projec-resol}
	Every $A\in\lcalz$ has a $ \altpNNILpar $-projective resolution. 
	Moreover this resolution is computable and 
	unique up to $ \IPC $-provable equivalency, 
	i.e.~for every two $ \altpNNILpar $-projective resolutions 
	$ \Delta=\{B_1,\ldots,B_m\}$ and $\Delta'=\{C_1,\ldots,C_n\}  $ 
	for $ A $, we have  
	$ m=n$ and there is some permutation $\sigma$ such that
	for every $ i $,  $ \vdash B_i\lr C_{\sigma(i)} $. 
\end{theorem}
\begin{proof} 
Given $A$, by \Cref{Theorem-IPC-nnilp-Finitary}	there is a 
$\altNNILpar$-projective resolution $\Delta$ for $A$. Then define 

\begin{equation*}
\Pi_0:=\{E\wedge E': E\in \Delta\text{ and }  E'\in \Gamma_{E^\dagger}\}
\end{equation*}
in which $E^\dagger\in\NNILpar$ is the $\altNNILpar$-projection of $E$
and $\Gamma_{E^\dagger}$ is the decomposition of $E^\dagger$ to 
components, as provided by \Cref{ipc-component-dec}.  Finally 
let $\Pi\subseteq\Pi_0$ be some $\subseteq$-minimal set  with 
$\vdash \bigvee \Pi_0\lr\bigvee\Pi$.
Then by 
	\Cref{pnnilpv-nnilp} and the following fact 
	one may easily observe that $\Pi$ is a 
	$\altpNNILpar$-projective resolution for $A$: 
	if \sixtaghir{$ E\xratth F\in\altNNILpar $ and $ E'\in\pNNILpar $, 
	then $ (E\wedge E')\xratth (F\wedge E')$}. 
	\\
	For the uniqueness, it is sufficient to show that for every
	$ \altpNNILpar $-projective $ E $, if 
	$ E \altarpn \bigvee_i F_i$ 
	then  for some $ i $ we have 
	$ \vdash E\to F_i $. Let $ E\xratth E^\dagger $.
	Then by $ E \altarpn \bigvee_i F_i$ we have 
	$ \vdash \theta(E^\dagger\to \bigvee_iF_i)  $.
	Hence $ \vdash E^\dagger\to\bigvee_i\theta(F_i) $ and
	since $E^\dagger $ is prime, we have 
	$ \vdash E^\dagger\to \theta(F_i) $ for some $ i $. 
	Thus \Cref{proj-pres-admiss} implies $ \vdash E \to F_i$, as desired.
\end{proof}

\section{$\NNILpar$-admissible rules of $\IPC$}
\label{sec-admis}
In  \citep{IemhoffT}, 
the admissibility relation $\ar $  is characterized by means of preservation relation $\rhd$ and its Kripke semantics, called 
${\sf AR}$-models. In this section we will characterize
and prove the decidability of 
{$\arn$}\ , the \textit{$\altNNILpar$-admissible rules of $\IPC$}
(see \Cref{pres-admis}).
 For this end, we imitate the route in 
\citep{IemhoffT}, i.e.~we define a system $\ARN$ for the 
$\altNNILpar$-admissible rules of $\IPC$ and also introduce 
a Kripke semantic for it and prove the soundness and completeness. 
Finally using this  and the results
in \Cref{Sec-Ghil} we prove that $\ARN$ is sound and complete for 
both $\altNNILpar$-admissibility and $\altdNNILpar$-preservativity, 
i.e.~$\ARN\vdash A\rhd B$ iff $A\altarn    B $ iff $A\altprdnpar B$.
\subsection{The system $\ARN$}\label{ARp}
 
$\ARN$ is a system which proves  formulas 
in the form $A\rhd B$, and $A,B\in\lcalz$. 
Before we continue with the axioms and rules of the 
system $\ARN$, let us first define a notation. 

\begin{equation*}
\itp{A}{B}:=\begin{cases}
B \quad &: B\in\parr\cup\{\bot\}\\
A\to B &: \text{otherwise}
\end{cases}
\end{equation*}
\noindent  Then $\ARN$ is defined as $\BAR\IPC$ (as defined in \Cref{pres-admis}) plus the following axiom and rules:

\begin{itemize}[leftmargin=1cm]
\item[$\VAR:$]\quad $B\to C\rhd \bigvee_{i=1}^{n+m}
\itp{B}{E_i}$, in which  
$B=\bigwedge_{i=1}^n (E_i\to F_i)$ and 
$C=\bigvee_{i=n+1}^{n+m} E_i$.
\end{itemize}	
\begin{center}
\seventaghir{%
\Ax{$B\rhd A$}
		\Ax{$C\rhd A$}
		\RLa{Disj}
		\BI{$B\vee C\rhd A$}
		\DP}
\quad \quad\quad \quad
\Ax{$A\rhd B$}
\RLa{$\mont(\parr)$}
\LLa{($p\in\parr$)}
\UI{$p\to A\rhd p\to B$}
\DP
\end{center}

\begin{remark}
The system $\AR$, as defined in \citep{IemhoffT}, 
is $\ARN$  with   $\parr=\emptyset$. The Visser rule $\VAR$ in this case 
is proved to be of central importance \citep{iemhoff2005intermediate}.
\end{remark}
\begin{remark}
As we will see in \Cref{montagna-nnil}, 
the following extension of the Montagna's rule is 
admissible in $\ARN$: 
\begin{prooftree}
\Ax{$A\rhd B$}
\LLa{\uparan{$E\in\NNILpar$}}
\RLa{.}
\UI{$E\to A\rhd E\to B$}
\end{prooftree}
\end{remark}
\begin{remark}\label{remark-arp}
$\ARN$ is closed under general substitutions $\theta$ with 
$\theta(p)\in\{\top,\bot\}\cup{\parr}$ for every $p\in\parr$, i.e.~$\ARN\vdash A\rhd B$ implies $\ARN\vdash 
\theta(A)\rhd\theta(B)$.
\end{remark}
\begin{proof}
Use induction on the complexity of proof $\ARN\vdash A\rhd B$. All cases are easy and left to the reader.
\end{proof}
The following theorem is from \citep{IemhoffT}. 
\begin{theorem}
$A\ar B$ iff $\AR\vdash A\rhd B$.    
\end{theorem}
\noindent 
\begin{lemma}\label{Lem-ARN-implies-arpn}
$\ARN\vdash A\rhd B$ implies $A\altarpn B$.
\end{lemma}
\begin{proof}
We use induction on the complexity of  the proof 
$\ARN\vdash A\rhd B$. All cases are easy except for
the axiom $\VAR$ and the rules  ${\mont(\parr)}$ and Disj.
\begin{itemize}[leftmargin=*]
\item \textit{$\VAR$:} Let $C=\bigwedge_{i=1}^n 
(E_i\to F_i)$ and  $D=\bigvee_{i=n+1}^{n+m} E_i$. 
We show $C\to D\altarn \bigvee_{i=1}^{n+m}\itp{C}{E_i}$. 
So assume that $\theta$ is a substitution 
and $G\in\pNNILpar$ and show that
$  \vdash G\to \theta(C\to D)$ implies 
$  \vdash G\to 
\theta(\bigvee_{i=1}^{n+m}\itp{C}{E_i})$. 
We reason by contraposition.  
Let $  \nvdash G\to 
\theta(\bigvee_{i=1}^{n+m}\itp{C}{E_i})$.
Hence for every $i\leq n+m$ we have 
$  \nvdash G\to \theta(\itp{C}{E_i})$. Then 
for every $i$ there is some Kripke model $\kcal_i$
with the root $w_i$
such that $\kcal_i\Vdash G$ and 
$\kcal_i\nVdash \theta({E_i})$ and 
moreover for every $i$ 
with $E_i\not\in\parr\cup\{\bot\}$
we have $\kcal_i\Vdash \theta(C)$.  Then
\Cref{component-extendible-prime} implies that  $G$ is extendible. 
Let $\kcal$ be a variant of $\sum(\{\kcal_i\}_i)$
such that $\kcal\Vdash G$ and let $w_0$ be its root. 
Also define  $\kcal'$  as follows: 
for every $i$ such that $E_i\in\parr\cup\{\bot\}$, 
eliminate  $\kcal_i$ (we mean its nodes) from $\kcal$.
Then evidently $\kcal'$ is a submodel of $\kcal$. 
Since $G$ is $\NNIL$, by \Cref{Theorem-NNIL-Submodel}, 
we have $\kcal'\Vdash G$. 
It is sufficient to show that $\kcal'\Vdash \theta(C)$
and 
$\kcal'\nVdash \theta(D)$. Since for every node $w$ in 
$\kcal'$ other than $w_0$, we have $\kcal',w\Vdash \theta(C)$, 
if we show $\kcal',w_0\nVdash \theta(E_i)$ for every $i$, 
we have both  $\kcal'\Vdash \theta(C)$
and $\kcal'\nVdash \theta(D)$ and we are done.
We have two cases. (1)  $E_i\in\parr\cup\{\bot\}$. 
In this case we have $\theta(E_i)=E_i$, and since 
$\kcal,w_i\nVdash E_i$ we have $\kcal,w_0\nVdash E_i$
and hence $\kcal',w_0\nVdash E_i$. (2) 
$E_i\not\in\parr\cup\{\bot\}$. Since $w_i$ in this case 
is a node of $\kcal'$ and $\kcal',w_i\nVdash 
\theta(E_i)$, we have $\kcal',w_0\nVdash \theta(E_i)$.
\item $\mont(\parr)$:
Let $A\altarn B$ and show 
$p\to A\altarn p\to B$ for every $p\in\parr$. 
Let $\theta$ be a substitution and $E\in\NNILp$ such that 
$  \vdash E\to \theta(p\to A)$. 
Hence $  \vdash (E\wedge p)\to \theta(A)$. Then 
by $A\altarn B$ we have $  \vdash (E\wedge p)\to 
\theta(B)$ and hence $  \vdash E\to \theta(p\to B)$, as desired.
\item \textit{Disj:} Let $B\altarn A$ and 
$C\altarn A$ and show $B\vee C\altarn A$. 
\Cref{pnnilpv-nnilp} and $B\altarn A$ and 
$C\altarn A$ imply $B\altarpn A$ and 
$C\altarpn A$. 
Let  $E\in\pNNILp$ and $\theta$ a substitution 
such that $  \vdash E\to \theta(B\vee C)$. 
Since $E$ is prime,
either we have $  \vdash E\to\theta(B)$
or $  \vdash E\to \theta(C)$. In either of the cases,
by $B\altarpn A$ and $C\altarpn A$ we have 
$  \vdash E\to \theta(A)$.
So by this argument we may conclude that $(B\vee C)\altarpn A$ and then 
by \Cref{pnnilpv-nnilp} we have $(B\vee C)\altarn A$.
\qedhere
\end{itemize}
\end{proof}
\begin{corollary}\label{Cor-ARN-implies-arn}
$\ARN\vdash A\rhd B$ implies $A\altarn B$. 
\end{corollary}
\begin{proof}
Use \Cref{Lem-ARN-implies-arpn,pnnilpv-nnilp}.
\end{proof}

\begin{corollary}\label{Cor-ARN-implies-IPC-deriv}
For every $A\in\NNILpar$ and $B\in\lcalz$, 
if $\ARN\vdash A\rhd B$ then 
$  \vdash A\to B$.
\end{corollary}
\begin{proof}
Let $\ARN\vdash A\rhd B$. Then by 
\Cref{Cor-ARN-implies-arn}, we have 
$A\altarn B$. Let $\theta$ be identity substitution. Then we have 
$  \vdash A\to \theta(A)$. Hence $  \vdash A\to \theta(B)$, which implies 
$  \vdash A\to B$, as desired.
\end{proof}
\subsection{$\ARN$-models}\label{sec-ARmod}
Before we define $\ARN$-models, the Kripke models for 
which $\ARN$ is sound and complete, let us present some 
definitions.  
Let $\kcal=(W,\pce,\V)$ is a Kripke model, not necessarily  finite tree. 
All over the rest of this subsection we assume that 
in general a Kripke model is not necessarily  finite tree. 
Given  a set $\Gamma$   of formulas,
two nodes $v,w\in W$ are called 
\textit{$\Gamma$-similar},
notation $v\eqg w$, if for every $A\in\Gamma$
we have $\kcal,v\Vdash A$ iff $\kcal,w\Vdash A$. 
Let $W'\subseteq W$ is a set of nodes and $w\in W$.
The notation $w\pce W'$, means $w\pce w'$ for every 
$w'\in W'$.  
We say that 
$w\in W$ is a \textit{tight predecessor} of $W'$, if 
$w\pce W'$ and for every $u\succcurlyeq w$, 
either $u=w$ or $u\succcurlyeq v$ for some $v\in W'$.
A node $w$ is called a \textit{base}, if for every 
finite set $W'\subseteq W$ such that $w\pce W'$, there
is  some $w'\in W$ such that:
$w\pce w'\pce W'$ and $w\eqp w'$ and 
$w'$ is a tight predecessor of $W'$. And finally,
a  Kripke model $\kcal=(W,\pce,\V)$ is an $\ARN$-model if
it is rooted (let $w_0$ be its root) and 
there is some set  $\basisworlds W\subseteq W$ with the following 
properties:
\begin{itemize}
\item $w_0\in \basisworlds{W}$,
\item every $w\in \basisworlds{W}$ is a base,
\item for every $w'\in \basisworlds{W}$ and $w\succcurlyeq w'$, 
there is some $v\in \basisworlds{W}$ such that $v\eqp w$ and 
$w'\pce v\pce w$.
\end{itemize}
Such $\basisworlds{W}$ is called a \textit{base-set} for $\kcal$. 
\\
We say that $\kcal$ is \textit{good}, if for every 
finite set of nodes $W'$, and every 
$X\subseteq \parr$
such that $\kcal,W'\Vdash X$,
there is some $w'\in \basisworlds{W}$ such that $w'\pce W'$ and 
$\kcal(w')\cap\parr= X$.

\begin{remark}\label{Remark-ARN-models}
Let $\kcal=(W,\pce,\V)$ is an $\ARN$-model
with a base-set $\basisworlds{W}$, and $w\in \basisworlds{W}$.
Then $\kcal_w$ is also an $\ARN$-model with the 
base-set $\basisworlds{W}_w:=\{v\in \basisworlds{W}: v\succcurlyeq w\}$.
\end{remark}
\begin{theorem}\label{Theorem-ARN-Soundness}
\textup{(\textbf{Soundness})}
$\ARN\vdash A\rhd B$ implies $\kcal\Vdash B$, 
for every $\ARN$-model $\kcal$ with $\kcal\Vdash A$.
\end{theorem}
\begin{proof}
We use induction on the proof $\ARN\vdash A\rhd B$. All 
cases are trivial except for the axiom $\VAR$ and the 
rule $\mont$. First we treat Montagna's Rule. 
As induction hypothesis, let $\kcal\Vdash  A$
implies $\kcal\Vdash B$, for every $\ARN$-model $\kcal$. 
Also let $\kcal\Vdash p\to A$ for some $p\in\parr$
and $\ARN$-model $\kcal=(W,\pce,\V)$ with the base-set
$\basisworlds{W}$. 
We will show $\kcal\Vdash p\to B$. Let $w\in W$ 
such that $\kcal,w\Vdash p$. 
Since $\kcal$ is an $\ARN$-model, there is some 
$w'\in \basisworlds{W}$ such that $w\eqp w'$ and $w'\pce w$.
Then $\kcal,w'\Vdash p$ and hence $\kcal,w'\Vdash A$.
Observe that  $\kcal_{w'}$ is also an $\ARN$-model and 
$\kcal_{w'}\Vdash A$.  Hence by induction hypothesis 
$\kcal_{w'}\Vdash B$, which implies 
$\kcal,w\Vdash B$, as desired. 
\\
Next we show 
$\kcal\Vdash \VAR$ for every $\ARN$-model 
$\kcal=(W,\pce,\V)$ with the root $w_0$. 
Let $B=\bigwedge_{i=1}^n (E_i\to F_i)$ and 
$C=\bigvee_{i=n+1}^{n+m} E_i$. 
Also assume 
that $\kcal,w_0\nVdash \bigvee_{i=1}^{n+m}\itp{B}{E_i}$.
We show that 
$\kcal,w_0\nVdash B\to C$.
By definition of $\itp{B}{E_i}$, for every 
$E_i\in\parr\cup\{\bot\}$, we have 
$\kcal,w_0\nVdash E_i$, 
and for every $E_i\not\in \parr\cup\{\bot\}$, there is 
some $w_i\succcurlyeq w_0$ 
such that $\kcal,w_i\Vdash B$ and $\kcal,w_i\nVdash E_i$. 
Let $W':=\{w_i: E_i\not\in\parr\cup\{\bot\} \}$. There 
is some $w'\in {W}$   which is a tight predecessor 
of $W'$ and $w'\eqp w_0$. We show that $\kcal,w'\nVdash 
B\to C$ by showing $\kcal,w'\Vdash B$ and 
$\kcal,w'\nVdash C$. Let $E_i$ be some disjunct in $C$. 
If $E_i\in\parr\cup\{\bot\}$, then since $w'\eqp w_0$ 
and $\kcal,w_0\nVdash E_i$, we have 
$\kcal,w'\nVdash E_i$. Otherwise, since 
$\kcal,w_i\nVdash E_i$ and $w'\pce w_i$, we have 
$\kcal,w'\nVdash E_i$. This finishes showing 
$\kcal,w'\nVdash C$. Then we show $\kcal,w'\Vdash B$.
Let $E_i\to F_i$ is a conjunct in $B$. Consider some 
$w\succcurlyeq w'$ such that $\kcal,w\Vdash E_i$. 
Since $w'$ is a tight predecessor
of $W'$, either we have $w=w'$ or $w\succcurlyeq w_j$ for 
some $w_j\in W'$. If $w\succcurlyeq w_j$, since 
$\kcal,w_j\Vdash B$, we have $\kcal,w\Vdash B$ and then 
$\kcal,w\Vdash E_i\to F_i$, whence $\kcal,w\Vdash F_i$. 
Also if $w=w'$, then by the following argument, we have  
$\kcal,w'\nVdash E_i$, a contradiction with our first
assumption $\kcal,w\Vdash E_i$. 
Finally,  the argument for 
$\kcal,w'\nVdash E_i$: if $E_i\in\parr\cup\{\bot\}$,
then since $w'\eqp w_0$ 
and $\kcal,w_0\nVdash E_i$, we have 
$\kcal,w'\nVdash E_i$. Otherwise, since 
$\kcal,w_i\nVdash E_i$ and $w'\pce w_i$, we have 
$\kcal,w'\nVdash E_i$.
\end{proof}
\noindent
We follow the methods in 
\citep{IemhoffT} to prove the completeness theorem. 
This proof is almost identical to the one for \citep[proposition 7.2.2]{IemhoffT}. 
First some definitions and lemmas.
A set $w$ of formulas is $\IPC$-saturated if 
\begin{itemize}
\item $w\vdash   A$ implies $A\in w$,
\item $\bot\not\in w$,
\item $A\vee B\in w$ implies either $A\in w$ or $B\in w$.
\end{itemize}
Also  $w$ is called  $\ARN$-saturated if 
it is $\IPC$-saturated and 
\begin{itemize}
\item If $\ARN\vdash A\rhd B$ and $A\in w$, then $B\in w$.
\end{itemize} 
Let $*(.)$ is a property on sets of formulas. 
We say that $*(.)$ is an \textit{extendible} property 
if the following conditions hold:
\begin{itemize}
\item If $*(w)$ and 
$w\vdash   A$, then $*(w\cup\{A\})$.
\item If $*(w\cup\{A\vee B\})$ then 
either $*(w\cup\{A\})$ or $*(w\cup\{B\})$ hold.
\item \eighttaghir{%
If $\{w_i\}_{i\in\nat}$ is an infinite sequence s.t.~$*(w_i)$ and $w_i\subseteq w_{i+1}$ for every $i$, then $*(\bigcup_{i\in\nat}w_i)$.}
\end{itemize}
If also the following condition holds, we say that 
$*(.)$ is $\ARN$-extendible property. 
\begin{itemize}
\item If $*(w)$ and $\ARN\vdash w\rhd A$, then $*(w\cup\{A\})$.
\end{itemize}
In the above expression, $\ARN\vdash w\rhd A$ is 
a shorthand for $\ARN\vdash (\bigwedge_i B_i)\rhd A$ for some \textit{finite} set	$\{B_i\}_i\subseteq w$.
\begin{lemma}\label{Lem-Max-Sat}
For every extendible property $*(.)$, if $*(w)$
for some set $w$ of formulas holds, 
there is some maximal 
$\IPC$-saturated $w'\supseteq w$ such that 
$*(w')$. Moreover if $*(.)$ is $\ARN$-extendible,
then $w'$ is also $\ARN$-saturated. 
\end{lemma}
\begin{proof}
Let $A_1,A_2,\ldots $ be a list of all formulas 
such that each formula occurs  infinitely often. 
We define a sequence $w=w_0\subseteq w_1\subseteq 
w_2\subseteq \ldots$ and then define $w':=\bigcup_i w_i$. 

\begin{equation*}
w_{n+1}:=\begin{cases}
w_n\cup\{A_n\} \quad &: *(w_n\cup\{A_n\})
\\
w_n   &:\text{otherwise}
\end{cases}
\end{equation*}
It can be easily proved that this $w'$ satisfies all required conditions.
We required that each formula appears in the list infinitely often for the following reason. 
It might be the case that a formula $A$ together with $w_n$ at the stage $n$ does 
not satisfy the property $*()$, while at some future step it will. 
Without this infinite repetition condition, we lose the maximality of $w'$.  
\end{proof}
\begin{corollary}\label{Lem-ARN-saturation}
If $\ARN\nvdash A\rhd B$, then there is some 
$\ARN$-saturated $w$ such that $A\in w$ and $B\not\in w$.
\end{corollary}
\begin{proof}
Define the property $*(.)$ as follows:

\begin{equation*}
*(y): \ARN\nvdash y\rhd B.
\end{equation*}
Then it is straightforward to observe that 
$*(.)$ is $\ARN$-extendible and $*(\{A\})$ holds. Hence 
\Cref{Lem-Max-Sat} implies the desired result.
\end{proof}

\begin{theorem}\label{Theorem-ARN-Completeness}
\textup{(\textbf{Completeness})}
$\ARN$ is complete for good $\ARN$-models, i.e.~if 
for every good $\ARN$-model $\kcal$, we have
$\kcal\Vdash A$ implies $\kcal\Vdash B$, 
then $\ARN\vdash A\rhd B$.
\end{theorem}
\begin{proof}
As usual, we reason contrapositively. Let 
$\ARN\nvdash A\rhd B$. 
Define the Kripke model $\kcal=(W,\pce,\V)$ as follows.
Since $\ARN\nvdash A\rhd B$, by \Cref{Lem-ARN-saturation}
there is some $\ARN$-saturated set $w_0$ such that $A\in w_0$ and 
$B\not\in w_0$.  Then define

\begin{equation*}
W:=\{w\supseteq w_0: w\text{ is a 
$\IPC$-saturated set of formulas}\}.
\end{equation*}
Also define $u\pce v$ iff $u\subseteq v$. Finally define 
$w\V a$ iff $a\in w$ for atomic $a$. We will show that this model is a good $\ARN$-model 
such that $\kcal\Vdash A$
and $\kcal\nVdash B$.
First note that by a standard argument, 
one may easily prove by induction on the complexity 
of $A$ that  $A\in w$ iff $\kcal,w\Vdash A$. 
Then since $A\in w_0$ and $B\not\in w_0$, we have 
$\kcal\Vdash A$ and $\kcal\nVdash B$. 
So it remains to show that $\kcal$ is 
a good  $\ARN$-model. Let $\basisworlds{W}$ as follows:

\begin{equation*}
\basisworlds{W}:=\{w\in W: w \text{ is $\ARN$-saturated}\}.
\end{equation*}
We will show that $\basisworlds{W}$ is a base-set, i.e.~has 
the following properties:
\begin{itemize}
\item $w_0\in \basisworlds{W}$,
\item every $w\in \basisworlds{W}$ is a base,
\item for every $w'\in \basisworlds{W}$ and $w\succcurlyeq w'$, 
there is some $v\in \basisworlds{W}$ such that $v\eqp w$ and 
$w'\pce v\pce w$.
\end{itemize}
The first property is obvious. For the second property 
we will need $\VAR$ and for the third one we will use 
\mont's rule.\\
Let $w\in \basisworlds{W}$ and $w\pce\{ w_1,\ldots,w_n\}$. We find 
some tight predecessor $u$ such that $u\eqp w$ and 
$w\pce u\pce \{w_1,\ldots,w_n\}$. 
Let $\what:=\bigcap_i w_i$ and define

\begin{equation*}
\Delta:=\{  
E\to F: E\to F\in \hat{w}\text{ and }
(E\not\in \hat{w}\vee E\in\parr\setminus w)
 \}.
 \end{equation*}
Define the property $*(.)$ as follows:

\begin{equation*}
*(y): y\vdash   \bigvee_i A_i \vee\bigvee_i p_i
\text{ and $\fa i \ (p_i\in \parr)$ implies }\ex i\
  (A_i\in \what) \vee \ex i\ (p_i\in w).
\end{equation*}
Note that by letting the second disjunction as empty, 
from $*(y) $ we have 
$y\vdash   \bigvee_i A_i$ 
 implies $\ex i\ (A_i\in \what)$. Similarly and by considering the first disjunction as empty disjunction, 
from $*(y) $ we get
$y\vdash   \bigvee_i p_i$ 
 implies $\ex i\ (p_i\in w)$. 
It is not difficult to observe that $*(.)$ is an 
extendible property. 
Then we show  $*(w\cup\Delta)$. 
Let $w\cup\Delta\vdash   \bigvee_i C_i\vee 
\bigvee_i p_i$ and $p_i\in \parr$. 
Then 
$w\vdash   G \to (\bigvee_i C_i\vee 
\bigvee_i p_i)$ in which 
$G=\bigwedge_i(E_i\to F_i)$ and $E_i\to F_i\in \Delta$.  
Since $w\in \basisworlds{W}$, and 

\begin{equation*}
\ARN\vdash \left(G\to (\bigvee_i C_i\vee 
\bigvee_i p_i)\right)\rhd \left(
\bigvee_i \itp{G}{E_i}\vee
\bigvee_i\itp{G}{C_i}\vee\bigvee_i\itp{G}{p_i}\right)
,
\end{equation*}
we have $w\vdash   \bigvee_i \itp{G}{E_i}\vee
\bigvee_i\itp{G}{C_i}\vee\bigvee_i\itp{G}{p_i}$. 
Since $w$ is $\IPC$-saturated,
either $w\vdash  \itp{G}{E_i}$  or 
$w\vdash  \itp{G}{C_i}$ or 
$w\vdash  \itp{G}{p_i}$, for some $i$. If 
$w\vdash   \itp{G}{E_i}$, then 
$w,\Delta\vdash   E_i$
and since $w\cup\Delta\subseteq \what$, we have 
$E_i\in \what$, a 
contradiction. So either we have 
$w\vdash  \itp{G}{C_i}$ or 
$w\vdash  \itp{G}{p_i}$, in 
which we have $w,\Delta\vdash C_i$ or $w\vdash p_i$. 
Hence either $C_i\in \what$ or $p_i\in w$.
This finishes showing $*(w\cup\Delta)$.
\\
Now let $u\supseteq (w\cup \Delta)$ be a maximal 
$\IPC$-saturated set such that $*(u)$, as provided by 
\Cref{Lem-Max-Sat}.
Then we show that $u$ satisfies all required conditions:
\begin{itemize}
\item $w\pce u\pce \{w_1,\ldots, w_n\}$.  Since 
$w\subseteq u$, we have $w\pce u$. Also from
$*(u)$, we get $u\subseteq\what$ and hence for every $i$
 we have $u\pce w_i$. 
\item $w\eqp u$. 
Since $w\subseteq u$, we have 
$w\cap\parr\subseteq u\cap\parr$. For the other direction, let $p\in \parr\cap u$. Then 
$u\vdash p$ and from $*(u)$ we have $p\in w$. 
\item $u$ is a tight predecessor of $\{w_1,\ldots,w_n\}$.
We reason by contraposition. 
Let $v\supsetneqq u$ such that for every $i$, 
$w_i\not\subseteq v$. 
Then for every $i$ there is some 
$C_i\in  w_i\setminus v$ and hence 
$\bigvee C_i\in \what\setminus v$.
 On the other hand, 
since $u$ is a maximal saturated set with
$*(u)$ and $v\supsetneqq u$ and $v$ is $\IPC$-saturated,
we have $\neg *(v)$. Hence 
$v\vdash \bigvee_i A_i\vee \bigvee_ip_i$ and 
for every  $i$ we have $p_i\in\parr$ and  $A_i\not\in \what$ and $p_i\not\in w$. From
$v\vdash \bigvee_i A_i\vee \bigvee_ip_i$, 
there is  some $E$ such that
either  we have $E\in v\setminus \what$ or 
 $E\in v\cap(\parr\setminus  w)$. In 
either of the cases, by definition of $\Delta$
we have $E\to \bigvee C_i\in\Delta$. Hence
$E\to \bigvee C_i\in v$ and 
then $\bigvee C_i\in v$, a contradiction.
\end{itemize}
It finishes showing the second property 
of base-set $\basisworlds{W}$. Next we show that $\basisworlds{W}$ satisfies 
the third condition. 
Let $w'\in \basisworlds{W}$ and $w'\pce w$. 
Define the property $*(.)$ as follows.

\begin{equation*}
*(y): \text{ for every $C$, if }
\ARN\vdash y\rhd C \text{, then } C\in w.
\end{equation*}
We show that $*(.)$ is an $\ARN$-extendible property and 
$*(w'\cup w_\parr)$, in which 
$w_\parr:=w\cap\parr$. First let us show why 
this finishes the proof. From \Cref{Lem-Max-Sat} 
we get some $\ARN$-saturated 
$v\supseteq (w'\cup w_\parr)$ such that $*(v)$. 
Hence by definition $v\in \basisworlds{W}$. 
Since $v\supseteq w'$, we have $w'\pce v$. 
Then we show $v\pce w$. Let $C\in v$. From $*(v)$
and $\ARN\vdash v\rhd C$, we have $C\in w$, as desired. 
So we have  $v\pce w$.  Finally we show $v\eqp w$. 
We must show $v_\parr=w_\parr$, which holds 
because $v\supseteq w_\parr$ and $v\subseteq w$. 
\\
So it remains to show that $*(.)$ is an $\ARN$-extendible 
property and $*(w'\cup w_\parr)$. First we show that 
$*(.)$ satisfies all required conditions for 
 $\ARN$-extendibility:
\begin{itemize}
\item If $*(y)$ and $y\vdash   E$. We must show 
$*(y\cup\{E\})$. Let $\ARN\vdash y\cup\{E\}\rhd C$. 
Hence $\ARN\vdash E\wedge \bigwedge_i F_i\rhd C$ for some 
finite  set $\{F_i\}_i\subseteq y$. Then since 
$y\vdash   E $, we have $\ARN\vdash y\rhd 
E\wedge \bigwedge_i F_i$. Hence   $\ARN\vdash y\rhd C$. 
Then from $*(y)$ we have $C\in w$, as desired.
\item If neither 
$*(y\cup\{E\})$ nor $*(y\cup\{F\})$ hold, then 
we show that $*(y\cup\{E\vee F\})$ does not hold. 
Let $C,D$ such that $\ARN\vdash y\cup\{E\}\rhd C$ and  
$\ARN\vdash y\cup\{F\}\rhd D$ and $C\not\in w$ and 
$D\not\in w$. Hence by disjunction rule, 
we have $\ARN\vdash y\cup\{E\vee F\}\rhd C\vee D$. Since 
$w$ is $\IPC$-saturated, we also have $C\vee D\not\in w$. 
Hence $*(y\cup\{E\vee F\})$ does not hold.
\item Let $*(y)$ and $\ARN\vdash y\rhd E$. We 
must show that 	$*(y\cup\{E\})$. Let 
$\ARN\vdash y\cup\{E\}\rhd C$. Then from $\ARN\vdash y\rhd E$ we have 
$\ARN\vdash y\rhd C$. Then from $*(y)$  we  have 
$C\in w$.
\end{itemize}
Finally we  show that $*(w'\cup w_\parr)$. Let 
$\ARN\vdash w'\cup w_\parr\rhd C$. 
Hence $\ARN\vdash \bigwedge w_\parr\wedge E\rhd C$, for some $E\in w'$. 
Then  by 
Montagna's Rule we have $\ARN\vdash 
\bigwedge w_\parr\to E\rhd  \bigwedge w_\parr\to C$.
Since $E\in w'$, we have $\bigwedge w_\parr\to E\in w'$
and hence by $\ARN$-saturatedness of $w'$ we have 
$\bigwedge w_\parr\to C\in w'$. Since $w'\subseteq w$,
we have $\bigwedge w_\parr\to C\in w$ and hence
$C\in w$.
\\
It only remains to show that $\kcal$ is good. Let 
$w_1,\ldots,w_n\in W$ and $\what:=\bigcap_iw_i$.
Also assume that  $X\subseteq \what\cap\parr$. 
We find some 
$w\in \basisworlds{W}$ such that 
$ w\subseteq \what$ and $w\cap\parr={X}$.
Define 

\begin{equation*}
*(y): \text{For every $C_i$ and $p_i\in\parr$, if }
\ARN\vdash y\rhd \bigvee_i C_i\vee\bigvee_i p_i \text{, then }
\ex i\  C_i\in \what \vee \ex i\ p_i\in {X}.
\end{equation*}
We show that $*(.)$ is an $\ARN$-extendible 
property and $*({X})$. Then by 
\Cref{Lem-Max-Sat} we have some $\ARN$-saturated $w$
such that ${X}\subseteq w$ and $*(w)$ holds. 
From $*(w)$ it is clear that $w\subseteq \what$. 
Also if $p\in \parr\cap w$, then by $*(w)$ we have 
$p\in {X}$ and hence $w\cap\parr={X}$. 
Hence  $w$ satisfies all required conditions. 
It remains only to show that $*(.)$ is 
$\ARN$-extendible property and $*({X})$. 
First the $\ARN$-extendibility of $*(.)$: 
\begin{itemize}
\item If $*(y)$ and $y\vdash   E$. We must show 
$*(y\cup\{E\})$. Let $C=\bigvee_i C_i\bigvee_i p_i$ and 
$p_i\in\parr$ and 
$\ARN\vdash y\cup\{E\}\rhd C$. 
Hence $\ARN\vdash E\wedge \bigwedge_i F_i\rhd C$ for some 
finite  set $\{F_i\}_i\subseteq y$. Then since 
$y\vdash   E $, we have $\ARN\vdash y\rhd 
E\wedge \bigwedge_i F_i$. Hence   $\ARN\vdash y\rhd C$. 
Then from $*(y)$ we have $C_i\in \what$
or $p_i\in {X}$, for some $i$.
\item If neither 
$*(y\cup\{E\})$ nor $*(y\cup\{F\})$ hold, then 
we show that $*(y\cup\{E\vee F\})$ does not hold. 
Let $C=\bigvee_i C_i\vee\bigvee_ip_i$ and 
$D=\bigvee_iD_i\vee\bigvee_iq_i$ and $p_i,q_i\in\parr$
such that $\ARN\vdash y\cup\{E\}\rhd C$ and  
$\ARN\vdash y\cup\{F\}\rhd D$ and for all $i$ we have 
$C_i,D_i\not\in \what$ and 
$p_i,q_i\not\in {X}$. 
Hence by disjunction rule, 
we have $\ARN\vdash y\cup\{E\vee F\}\rhd C\vee D$,
while for all $i $,  $C_i,D_i\not\in \what$ and $p_i,q_i\not\in {X}$.  
Hence $*(y\cup\{E\vee F\})$ does not hold.
\item Let $*(y)$ and $\ARN\vdash y\rhd E$. We 
must show that 	$*(y\cup\{E\})$. Let 
$C=\bigvee_i C_i\vee\bigvee_ip_i$ and $p_i\in \parr$
and 
$\ARN\vdash y\cup\{E\}\rhd C$. 
Then from $\ARN\vdash y\rhd E$ we have 
$\ARN\vdash y\rhd C$. Then from $*(y)$  we  have 
$C_i\in \what$ or $p_i\in {X}$
for some $i$.
\end{itemize}
It finishes showing that $*(.)$ is an $\ARN$-extendible 
property. Then we show $*({X})$. Let 
$C=\bigvee_iC_i\vee\bigvee_ip_i$
and $p_i\in\parr$ and 
$\ARN\vdash \parr\cap w\rhd C$. 
Then by \Cref{Cor-ARN-implies-IPC-deriv} 
we have ${X}\vdash   C$.  
Since $\bigwedge ({X})$ is extendible,
by \Cref{Remark-extendible-disjun}
for some $i$ we have 
${X}\vdash   C_i$ or 
${X}\vdash   p_i$. 
Since ${X}\subseteq\what$, for some $i$
either we have $C_i\in\what$ or $p_i\in{X}$, as desired.
\end{proof}
\subsection{$\altNNILpar$-admissibility}

\begin{lemma}\label{Lem-ARN-models-arn}
For every good $\ARN$-model $\kcal$ and $n\in\mathbb{N}$,
 there is some 
$\parr$-subextendible stable class of finite rooted models 
$\scrk$ such that for every 
formula $A$ with $c(A)\leq n$ we have 
$\kcal\Vdash A$ iff $\scrk\Vdash A$. 
\end{lemma}
\begin{proof}
Given a good $\ARN$-model $\kcal=(W,\pce,\V)$ with 
$\basisworlds{W}\subseteq W$ as its base-set, we define a stable 
$\parr$-subextendible class $\scrk$ of finite 
rooted Kripke models as 
follows. $\scrk$ includes all Kripke models
$\kcal'=(W',\pce',\V')$ with the following properties:
\begin{itemize}
\item $\kcal'$ is finite rooted with tree frame.
\item $\kcal'$ is embeddable in $\kcal$, 
i.e.~there is a function $f:W'\longrightarrow W$ such 
that $w'\V' a$ iff $f(w')\V a$; and 
$w'\pce' v'$ implies $f(w')\pce f(v')$.
\item For all $A$ with $c(A)\leq n$ and for every 
$w'\in W'$  we have $\kcal',w'\Vdash A$
iff $\kcal,f(w')\Vdash A$.
\end{itemize}
Obviously $\scrk$ is stable and $\kcal\Vdash A$ implies 
$\scrk\Vdash A$ for every $A$ with $c(A)\leq n$. 
It remains to show:
\begin{enumerate}[leftmargin=*]
\item $\scrk\Vdash A$ implies $\kcal\Vdash A$ for every 
$A$ with $c(A)\leq n$. It is sufficient to show that 
for a given $n$ and  $w_0\in W$, there is a 
finite rooted (with the root $w'_0$) tree-frame Kripke 
model $\kcal'=(W',\pce',\V')$ which is embeddable in 
$\kcal$ with the embedding $f$ such that
$f(w'_0)=w_0$ and  for every 
$w'\in W'$ and $A$ with $c(A)\leq n$ we have 
$\kcal',w'\Vdash A$ iff $\kcal,w\Vdash A$. 
First we inductively 
define sets  $W_i$ of sequences of implications 
$B\to C$ with $c(B\to C)\leq n$,
 for $0\leq i\leq n$  and the function
$f$ from $W_i$ to $W$.
Then let $W':=\bigcup_{i=0}^n W_i$.
 Let $W_0:=\{\langle \rangle\}$  and 
 $f(\langle\rangle):=w_0$. Assume that we already 
defined $W_i$ and define $W_{i+1}$ as follows. For 
every sequence $\sigma\in W_i$ and 
implication $B\to C$ with $c(B\to C)\leq n$ such that 
$\kcal,f(\sigma)\nVdash B\vee (B\to C)$, add the new node 
$\sigma*\langle B\to C\rangle$ to $W_{i+1}$ and 
define $f(\sigma*\langle B\to C\rangle)=u$
for some $u$ such that 
$u\sce f(\sigma)$ and $\kcal,u\Vdash B$
and $\kcal,u\nVdash C$.  
This finishes definition of $W_i$ and $W'$ and the 
embedding $f:W'\longrightarrow W$. Finally define 
$\sigma\pce'\gamma$ iff $\sigma$ is an initial segment of 
$\gamma$. Since there are only finitely many 
inequivalent formulas $A$ with $c(A)\leq n$, 
one may easily observe that $\kcal'$ is finite. The other 
required properties for $\kcal'$ are easy and left to the 
reader. 
\item $\scrk$ is $\parr$-subextendible. 
Let $\scrk':=\{\kcal_{1},\ldots,\kcal_{n}\}
\subseteq \scrk$ be finite such that 
$\scrk'$ is a $\pbold$-submodel of some 
$\kcal_{0}\in \scrk$ and $w'_i$ be 
the root of $\kcal_i$.
Let $f_i$ be the embedding of $\kcal_i$ in $\kcal$ 
and $w_i:=f_i(w'_i)$.
Since $\kcal$ is good, there is some $u\in W$ such that 
$u\eqp w_0$ and $u\pce w_1,\ldots,w_n$ and $u\in \basisworlds{W}$. 
Since $u$ is a base, there is some tight predecessor 
$v\in W$ for the set $\{w_1,\ldots,w_n\}$ such 
that $u\eqp v$ and $u\pce v\pce w_1,\ldots,w_n$. 
Define a $\pbold$-variant $\kcal''$
of $\kcal':=\sum(\scrk',\kcal_0)$ in this way:
$\kcal'',w'_0\Vdash a$ iff $\kcal,v\Vdash a$, for every atomic $a$. Then it is not difficult to observe that 
$\kcal''\in\scrk$.
\end{enumerate}
\end{proof}
\begin{lemma}\label{Lem-arn-p-extendible}
If $A\altarn B$ and $\scrk\subseteq\Mod{A}$ is   $\parr$-subextendible 
and stable, then $\scrk\Vdash B$.
\end{lemma}
\begin{proof}
Let $A\altarn B$ and $\scrk$ is a stable class of 
finite rooted models with tree frames. 
Let $\scrk'$ be the restriction of $\scrk$ to the atoms
appeared in $A,B,\parr$. Obviously $\scrk'\subseteq\Mod A$ also
is a $\parr$-subextendible
stable class.
Let $n:=\max\{c(A),\#\parr\}$. Then 
 \Cref{Lem-stable-extendible} implies that 
 $\langle \scrk'\rangle_n$ is also a 
$\parr$-subextendible stable class of finite rooted models 
with tree frames. \Cref{Lem-Mod(A)-Kripke} implies 
$\langle \scrk'\rangle_n=\Mod{C}$ for some 
$C$ with $c(C)\leq n$. Moreover, by 
\Cref{Theorem-Ghil-Ext}, there is a substitution
$\theta$ and $C'\in\NNILpar$ such that 
$  \vdash C'\lr \theta(C)$ and $C\vdash   E\lr \theta(E)$ for every formula $E$. 
On the other hand, \Cref{Cor-Kripke-bisim-property}
implies $\langle\scrk'\rangle_n\Vdash A$. Hence $A$ is 
valid in $\Mod{C}$, which implies $  \vdash C\to A$. 
Hence $  \vdash \theta(C)\to \theta(A)$ and then 
$  \vdash C'\to \theta(A)$. From $A\altarn B$ infer
$  \vdash C'\to \theta(B)$, or equivalently 
$  \vdash \theta(C\to B)$. Hence for every $\kcal$, 
and of course for every   
$\kcal\in \langle\scrk'\rangle_n$ we have $\kcal\Vdash 
\theta(C\to B)$. Let $\kcal$ be a model in $\langle\scrk'\rangle_n$.
Since $\kcal\Vdash C$ and $\theta$ is 
$C$-identity, we have $\kcal\Vdash C\to B$, and hence 
$\kcal\Vdash B$. Thus we have $\langle\scrk'\rangle_n\Vdash B$.
Since 
$\scrk'\subseteq  \langle\scrk'\rangle_n$, we 
also have $\scrk'\Vdash B$. Whence $\scrk\Vdash B$, as 
desired.
\end{proof}

\begin{theorem}\label{Characterization-admissibility}
The following statements are equivalent:
\begin{enumerate}
\item $\ARN\vdash A\rhd B$.
\item $A\altarn B$.
\item $B$ is valid in every $\parr$-subextendible 
stable class of 
Kripke models of $A$.
\item $B$ is valid in every good $\ARN$-model of $A$.
\end{enumerate}
\end{theorem}
\begin{proof}
$1\to 2$: \Cref{Cor-ARN-implies-arn}.\\
$2\to 3$: \Cref{Lem-arn-p-extendible}.\\
$3\to 4$: \Cref{Lem-ARN-models-arn}.\\
$4\to 1$: \Cref{Theorem-ARN-Completeness}.
\end{proof}

\begin{corollary}
\label{montagna-nnil}
The following rule is admissible in $\ARN$:
\Ax{$A\rhd B$}
\LLa{\uparan{$E\in\NNILpar$}}
\UI{$E\to A\rhd E\to B$}
\DP.
\end{corollary}
\begin{proof}
Since $\rhd={\altarn}\ \ $ and 
\Ax{\begin{tabular}{c}
{$A\altarn B$}
\end{tabular}
}
\UI{$E\to A\altarn E\to B$}
\DP, we have the desired result.
\end{proof}

\subsection{$\altdNNILpar$-preservativity logic}\label{sec-pres-1}

In the following theorem we show that the other direction of 
\Cref{pres-admis-rel} holds when $ \Gamma=\altNNILpar $:
 \begin{theorem}\label{arnp-character}
 		$ {\altprdnpar} = {\altarn}$ .
 \end{theorem}
\begin{proof}
\Cref{pres-admis-rel} implies that if 
$A\altarn B$ then $A\altprdnpar B$. For the other direction, 
assume that $A\altprdnpar B$ seeking to show $A\altarn B$. 
By \Cref{pnnilpv-nnilp} it is sufficient to show $A\arpn B$. 
Let $E\in\pNNILpar $ and substitution $ \theta $ such that 
$   \vdash E\to \theta(A) $. Let $ \Pi_A $ be the $ \pNNILpar $-projective resolution for $ A $, as guarantied by 
\Cref{Projec-resol}. Since $A\arpn \bigvee \Pi_A$
we have  $\vdash E\to \bigvee \theta(\Pi_A) $ and hence 
by primality of $E$, 
for some $ F\in\Pi_A $ we have
$\vdash E\to \theta(F)$. On the other hand, since $\Pi_A$
is a projective resolution for $A$ we have 
  $   \vdash  F\to A $. Then   
by $ A\altprdnpar B $ we get $   \vdash F\to B $. Hence 
$   \vdash \theta(F)\to \theta(B) $, which implies 
$   \vdash E\to\theta(B) $, as desired.
\end{proof}
\begin{remark}\label{argt=prtg}
For every $\Gamma$ and a logic $\sft\supseteq\IPC$ which admits 
$\Gamma$-projective resolutions,
i.e.~every $A\in\lcalz$ has a $\Gamma$-projective resolution in $\sft$,
the above proof works and we have ${\argt}={\prtdg}\ $. Hence we have
${\altarpn}={\altprdpnpar}$.
\end{remark}
\begin{remark}
By \Cref{pnnilpv-nnilp,arnp-character,Characterization-admissibility,argt=prtg}  we have:
\begin{center}
 $\ARN\vdash A\rhd B$ \quad  
iff\quad $A\altarn B$ \quad 
iff\quad $A\altarpn B$
\quad iff\quad $A\altprdnpar B$
\quad iff\quad $A\altprdpnpar B$.
\end{center}
\end{remark}

\subsection{$\NNILpar$-preservativity logic}\label{sec-pres-2}
In this subsection we  axiomatize the 
$\nnilpar$-preservativity and 
show $ {\prnpar}=\ARNN $ in which 
$ \ARNN  $ is defined as $ \ARN $ plus the following  axiom schema 
(the substitution axiom):

\begin{equation*}
\VA:  A\rhd \theta(A) \text{ for every substitution $ \theta $ (which by default is identity on parameters)}.
\end{equation*}

The main point of the axiom $ \VA $ is that we may annihilate 
occurrences of atomic variables, and together with other axioms of 
$ \ARN $ we may simplify formulas to $ \NNILpar $-formulas. 

\vspace{3mm}
\noindent\textbf{Important notice.}
In this subsection we assume that the set $\varr$ is \textit{infinite}. 
We require this condition for the proof of \Cref{variant-visser-prop} 
and its consequent theorems and lemmas.

\vspace{3mm}
Before we continue with providing such simplifying algorithm,
let us  define $\nitap{A}{B}$ and $\nitapp{A}{B}$, two variants of 
$\nita{A}{B}$:

\begin{equation*}
\nitap{A}{B}:=\begin{cases}
B\quad &: B\text{ is $\bot$ or parameter}\\
A \to B &: B\in\varr\\
\itap{A}{C}\circ\itap{A}{D} &: 
B=C\circ D\text{ and }\circ\in\{\vee,\wedge\}\\
(C\wedge{A{\downarrow}C})\to {(D{\downarrow}C)} &: B=C\to D 
\end{cases}
\end{equation*}
\begin{equation*}
\nitapp{A}{B}:=\begin{cases}
B\quad &: B\text{ is $\bot$  or parameter}\\
A \to \bot &: B\in\varr\\
\itapp{A}{C}\circ\itapp{A}{D} &: 
B=C\circ D\text{ and }\circ\in\{\vee,\wedge\}\\
 (C\wedge {A{\downarrow}C})\to {(D{\downarrow}C)} &: B=C\to D 
\end{cases}
\end{equation*}
In which $A{\downarrow}C$ is defined as follows: start from $A$ and replace every occurrence of an implication $C\to E$ with its consequent. 
More precisely, $A{\downarrow}C$
is defined by induction on the complexity of $A$:
\begin{itemize}
\item $A$ is atomic or $A=\bot$: $A{\downarrow}C:=A$.
\item $(A_1\circ A_2){\downarrow}C:=(A_1{\downarrow}C)\circ(A_2{\downarrow}C)$ for $\circ\in\{\vee,\wedge\}$.
\item $(A_1\to A_2){\downarrow}C:=\begin{cases}
(A_1{\downarrow}C)\to(A_2{\downarrow}C)\quad &: A_1\neq C\\
A_2{\downarrow}C &: A_1=C
\end{cases}$.
\end{itemize}
Then define the following variants of Visser rule:
\begin{itemize}
	\item[$\VARBP:$]\quad $B\to C\rhd \bigvee_{i=1}^{n+m}
	\itap{B}{E_i}$, in which  
	$B=\bigwedge_{i=1}^n (E_i\to F_i)$ and 
	$C=\bigvee_{i=n+1}^{n+m} E_i$.
\end{itemize}
\begin{itemize}
	\item[$\VARBPP:$]\quad $B\to C\rhd \bigvee_{i=1}^{n+m}
	\itapp{B}{E_i}$, in which  
	$B=\bigwedge_{i=1}^n (E_i\to F_i)$ and 
	$C=\bigvee_{i=n+1}^{n+m} E_i$.
\end{itemize}

\begin{lemma}\label{tech-lem}
$\ARN\vdash \VARBP$.
\end{lemma}
\begin{proof}
Let $B$ and $C$ and $E_i$ and $F_i$ as in $\VARBP$.
Define $X:=\{E_i: 1\leq i\leq n\}$ and $Y:=\{E_i: n+1\leq i\leq n+m\} $ and $Z:=\{E_i\to F_i: 1\leq i\leq n\}$. 
We say that $E$ is a \textit{basic phrase} if it is either $\bot$, atom or implication.
Also we say that $E$ is a \textit{phrase} if it is a finite conjunction of basic phrases. 
We prove the statement of this lemma in following steps:
\begin{enumerate}[leftmargin=*]
\item For every formula $E$, we can  inductively 
define a finite set $X_E$ of phrases such that $\vdash E\lr \bigvee X_E$ 
and $\vdash \itap{B}{E}\lr \bigvee_{F\in X_E}\itap{B}{F}$ as follows:
\begin{itemize}
\item $X_E:=\{ E\}$ if $E$ is $\bot$ or atomic or implication ($E$ is a basic phrase).
\item $X_E:=X_F\cup X_G$ if $E=F\vee G$.
\item $X_E:=\{F'\wedge G':F'\in X_F \text{ and } G'\in X_G\}$ for $E=F\wedge G$.
\end{itemize}
\item 
If $\vdash B\lr B'$  then $\vdash \itap{B}{E}\lr \itap{B'}{E}$:
	The proof is by straightforward induction on the complexity of $E$.
	For the case of $E=F\to G$, note that $(F\wedge B'{\downarrow}F)\to (G{\downarrow}F)$ is equivalent to 
	$B'\to E$. 
\item $\vdash A\to A'$ implies $\ARN\vdash A\rhd A'$. This holds by the definition of $\ARN$.
\item Without loss of generality, we may assume that every $E\in X$ is a phrase:
\\
Let $B':=\bigwedge_{E\to F\in Z}\bigwedge_{E'\in X_E}(E'\to F)$. Item 1 implies
$\vdash B\lr B'$. Let $X':=\bigcup_{E\in X}X_E$. Clearly, every $E\in X'$
is a phrase. Then by
assuming that $\ARN\vdash (B'\to C)\rhd \bigvee_{E\in X'\cup Y}\itap{B'}{E}$,
items 2 and 3 imply $\ARN\vdash (B\to C)\rhd \bigvee_{E\in X'\cup Y}\itap{B}{E}$.
On the other hand, by item 1 
the formula $\bigvee_{E\in X'\cup Y}\itap{B}{E}$ is equivalent (in the intuitionistic logic) 
to $\bigvee_{E\in X\cup Y}\itap{B}{E}$. Thus by item 3 we have
$\ARN\vdash (B\to C)\rhd \bigvee_{E\in X\cup Y}\itap{B}{E}$, as desired.
\item Without loss of generality we also may assume that every $E\in Y$ is a phrase:
\\
Let $Y':=\bigcup_{E\in Y}X_E$ and $C':=\bigvee Y'$. Clearly, every $E\in Y'$ 
is a phrase. Also assume that 
$\ARN\vdash (B\to C')\rhd \bigvee_{E\in X\cup Y'}\itap{B}{E}$.
Then by items 1 and 3 we have $\ARN\vdash (B\to C)\rhd \bigvee_{E\in X\cup Y}\itap{B}{E}$, 
as desired.
\item By the items 4 and 5, we may now assume that $E_i$ is a phrase for every 
$i\leq n+m$. So let $X_i$ be the set of basic  phrases such that $E_i=\bigwedge X_i$.
We claim that for every $i\leq n+m$ if we pick some $E'_i\in X_i$, then we have 

\begin{equation*}
\ARN\vdash B\to C\rhd \bigvee_{i=1}^{n+m}\itap{B}{E'_i}.
\end{equation*}
First note that $\ARN\vdash B\to C\rhd \bigvee_{i=1}^{n+m}\ita{B}{E_i}$.
On the other hand, by definition we have $\vdash \ita B{E_i}\to\ita B{E'_i}$. Thus by item 3
we get 
$\ARN\vdash B\to C\rhd \bigvee_{i=1}^{n+m}\ita{B}{E'_i}$.
Then since for every basic phrase $E$ we have 
$\IPC\vdash\ita B E\lr\itap B E$, we get desired result.
\item By previous item we have 

\begin{equation*}
\ARN\vdash B\to C\rhd \bigwedge\bigvee_{i=1}^{n+m}\itap{B}{E'_i},
\end{equation*}
in which the conjunction in the right hand varies over all choices 
$E'_i$'s from $X_i$'s for $1\leq i\leq n+m$. Then since $\itap{B}{E}$ commutes with $\wedge$ in its consequent, we get 
$\ARN\vdash B\to C\rhd \bigvee_{i=1}^{n+m}\itap{B}{E_i},$ as desired.\qedhere
\end{enumerate}
\end{proof}

\begin{lemma}\label{lem-76}
Given $B,E\in\lcalz$, let $\theta$ be the substitution which replaces 
every variable $x\in\sub B\cup\sub E$ with $x'\vee x''$ in which $x'$ and $x''$ 
are fresh atomic variables. Then 

\begin{equation*}
\ARNN\vdash \itap{\theta(B)}{\theta(E)}\rhd \itapp{B}{E}. 
\end{equation*}
\end{lemma}
\begin{proof}
By induction on the complexity of $E$:
\begin{itemize}[leftmargin=*]
\item $E$ is a parameter or $E=\bot$. Then by definition we have $\itap{\theta(B)}{\theta(E)}=E=\itapp{B}{E}$. 
\item $E=x\in\varr$. Then by definition we have $\nitap{\theta(B)}{\theta(E)}=(\theta(B)\to x')\vee(\theta(B)\to x'')$. Then define the substitution $\gamma'$ and $\gamma''$ 
as follows: 
\begin{itemize}
\item $\gamma'(x'):=\bot$ and $\gamma'(x''):=x$.
\item $\gamma''(x'):=x$ and $\gamma''(x''):=\bot$.
\end{itemize}
Then we have 
$\ARNN\vdash (\theta(B)\to x')\rhd \gamma'\theta(B)\to\gamma'(x')$
and $\ARNN\vdash (\theta(B)\to x'')\rhd \gamma''\theta(B)\to\gamma''(x'')$.
Obviously $\vdash \gamma'\theta(B)\lr B$ and $\vdash \gamma''\theta(B)\lr B$.
Thus $\ARNN\vdash (\theta(B)\to x')\rhd B\to\bot$
and $\ARNN\vdash (\theta(B)\to x'')\rhd B\to\bot$.
Hence by disjunction rule in $\ARNN$ we have $\ARNN\vdash \itap{\theta(B)}{\theta(E)}\rhd
B\to \bot$. On the other hand, by definition we have 
$\itapp{B}{E}=B\to \bot$. Thus $\ARNN\vdash \itap{\theta(B)}{\theta(E)}\rhd
\itapp{B}{E}$, as desired.
\item $E$ is a conjunction or disjunction. Use induction hypothesis.
\item $E=F\to G$. Let $B':=\theta(B)$ and $F':=\theta(F)$ and $G':=\theta(G)$.
Then by definition we have $\itap{\theta(B)}{\theta(E)}=(F'\wedge B'{\downarrow}F')\to(G'{\downarrow}F')$
and $\itapp{B}{E}=(F\wedge B{\downarrow}F)\to(G{\downarrow}F)$.
Let $\gamma$ be a  substitution with $\gamma(x'):=\gamma(x''):=x$. Then by substitution 
axiom in $\ARNN$ we have $\ARNN\vdash \big((B'{\downarrow}F')\to(F'\to G')\big)\rhd
\gamma\big((B'{\downarrow}F')\to(F'\to G')\big)$. 
By induction on $B$, one may easily observe that 
$\vdash \gamma (B'{\downarrow}F') \lr 
(B{\downarrow}F)$. Thus 
$\ARNN\vdash \big((B'{\downarrow}F')\to(F'\to G')\big)\rhd
\big((B{\downarrow}F)\to(F\to G)\big)$ and 
$\ARNN\vdash \itap{\theta(B)}{\theta(E)}\rhd
\itapp{B}{E}$, as desired. \qedhere
\end{itemize}
\end{proof}
\begin{remark}
A closer look at the proof of \Cref{lem-76} shows that 
the other direction also holds:

\begin{equation*}
\ARNN\vdash \itapp{B}{E}\rhd \itap{\theta(B)}{\theta(E)}. 
\end{equation*}
\end{remark}
\begin{lemma}\label{variant-visser-prop}
	$\ARNN\vdash\VARBPP$.
\end{lemma}
\begin{proof}
The idea is that we must show that another variant of Visser's rule, holds in 
$\ARNN$ in which, the variables must be annihilated (replaced by falsity). 
However it looks easy to directly use axiom $\VA$ to achieve this, there is a difficulty 
which prevents us doing it in a straightforward way: 
annihilating a variable, might cause damage to other occurrences of the variable.

Let $B$ and $C$ as in $\VARBP$. \Cref{tech-lem} implies 
$\ARNN\vdash (B\to C)\rhd \bigvee_{i=1}^{n+m}\itap{B}{E_i}$.
Let $X:=\{ E_i\in\varr: 1\leq i\leq n+m\} $ 
and $Y:=\{E_i\nin\varr: 1\leq i\leq n+m\} $. 
Then obviously we have $\ARNN\vdash (B\to C)\rhd \bigvee_{x\in X}B\to{x}\vee
\bigvee_{E\in Y}\itap{B}{E}$.
By the following argument, on the other hand we have 

\begin{equation}\label{eqqq}
\ARNN\vdash \bigvee_{x\in X}(B\to x)\rhd\big(B\to \bot\vee\bigvee_{E_i\in Y \& i\leq n}
\itapp{B}{E_i}\big).
\end{equation}
Then since for $E\in X$ we have $\itapp B E=B\to \bot$, 
we may conclude 

\begin{equation*}
\ARNN\vdash (B\to C)\rhd \bigvee_{i=1}^{n+m}\itapp B {E_i},\text{ as desired}.
\end{equation*}
So it remains only to show \cref{eqqq}.  
Let $\theta$ be a substitution such that $\theta(x)=x'\vee x''$ 
for every atomic variables $x$, and such that $x',x''\in\varr$ are some fresh variables. 
Note that since the set $\varr$ is assumed to be infinite, 
we can freely pick such fresh variables.
Then for every atomic variable $x$ by axiom 
$\VA$ we have 
$\ARNN\vdash  B\to x\rhd \theta(B)\to (x'\vee x'')$. 
By \Cref{tech-lem}  we have

\begin{equation*}
\ARNN\vdash  \theta(B)\to (x'\vee x'')\rhd  \theta(B)\to x' \vee 
\theta(B)\to x''\vee
\bigvee_{i=1}^n\itap{\theta(B)}{\theta(E_i)}.
\end{equation*}
Then thanks to the axiom $\VA$  we have  
(by replacing $\bot $ for $y'$ and $y$ for $y''$ for every variable $y$)
$\ARNN\vdash \theta(B)\to x'\rhd B\to \bot$. Similarly, we also have 
$\ARNN\vdash \theta(B)\to x''\rhd B\to \bot$. 
Then with the aid of  \Cref{lem-76} we may infer desired result in eq.~(\ref{eqqq}).
\end{proof} 

\begin{definition}\label{comp-meas-def}
Define the complexity number 
	$\ofrak A\in\nat^3$ as follows: 
	\begin{itemize}[leftmargin=*]
		\item $\Ifrak:=\{E\to F: E\to F\in\sub{A} 
\}$%
%
. \uparan{Remember that $A\in\sub A$.}
		\item $\ifrak A:=\max\{\#\Ifrak B:B\in \Ifrak A\}$. 
		\uparan{$\#X$ indicates the number of elements in the set $X$}
		\item $\cfrak A$ is defined as the number of occurrences of connectives 
		$\{\vee,	\wedge,\to\}$ in $A$.
		\item $\cfrakz A:=\max\{\cfrak B: B\in \Ifrak A\}$.
		\item $\vfrak A $ is defined as the number of occurrences of 
		variables in $A$.
		\item $\ofrak A:=(\ifrak A,\cfrakz A, \vfrak A ,\cfrak A)$. 
					Finally, we 
					let $\leq$ be the lexicographic order on tuples in $\nat^4$.
	\end{itemize}
\end{definition}

\begin{lemma}\label{star-prop}
	For every $A\in\lcalz$ one may effectively compute
	$A^\star\in\NNILpar$ such that:
	\begin{enumerate}
		\item $ \IPC\vdash A^\star\to A $,
		\item $ \ARNN\vdash A\rhd A^\star$,
    	\item $\suba{A^\star}\subseteq \suba{A}$.
	\end{enumerate}
\end{lemma}
\begin{proof}
First   by  recursion on $\ofrak A$ we define $A^\star$. Then by induction on 
	$\ofrak A$ we show that this $A^\star$ satisfies required properties.  
	So assume that 
	for every $B$ with $\ofrak {B}< \ofrak A$ we already defined $B^\star\in\NNILpar$ satisfying 
		\begin{enumerate}
		\item $ \IPC\vdash B^\star\to B $,
		\item $ \ARNN\vdash B\rhd B^\star$,
    	\item $\suba{B^\star}\subseteq \suba{B}$.
	\end{enumerate} 
	 Then we have the following cases for $A$. The required properties for $A^\star$ are
	  evident in many cases below and we skip the argument for them, except for few 
	  cases/properties which we do provide some additional argument right after 
	  the definition.
	\begin{itemize}[leftmargin=*]
		\item $A\in\varr$: Define $A^\star:=\bot$. Note that by
		the substitution axiom $\VA$, we have  $\ARNN\vdash A\rhd A^\star$.
		\item  $A\in\parr$: Define $A^\star:=A$. 
		\item $A=A_1\wedge A_2$: Define $A^\star:=A_1^\star\wedge A_2^\star$. 
		Observe that $\ifrak {A_i}\leq\ifrak A$ and $\cfrakz {A_i}\leq \cfrakz A$ and 
		$\vfrak{A_i}\leq \vfrak A$ and finally $\cfrak{A_i}<\cfrak{A}$.  
		Thus $\ofrak {A_i}<\ofrak{A}$ for $i=1,2$, and we have a legitimate inductive 
		definition. 
		\item $A=A_1\vee A_2$: Define $A^\star:=A_1^\star\vee A_2^\star$.
		Observe that $\ifrak {A_i}\leq\ifrak A$ and $\cfrakz {A_i}\leq \cfrakz A$ and 
		$\vfrak{A_i}\leq \vfrak A$ and finally $\cfrak{A_i}<\cfrak{A}$.  
		Thus $\ofrak {A_i}<\ofrak{A}$ for $i=1,2$, and we have a legitimate inductive 
		definition.  		
		\item $A=B\to C$: We have several sub-cases:
		\begin{itemize}[leftmargin=*]
			\item $B$ has an outer disjunction, i.e.~a disjunction 
			$E\vee F$ which is not in 
			the scope of  $\to$. Then,  there is some 
			formula $B_0(x)$ with the following properties: 
			(1)  $x$ is a variable not appearing in $B$,
			(2) $x$ occurs only once in $B_0$,
			(3) $x$ has an outer occurrence in $B_0$, 
			i.e.~$x$ is not in the scope of arrows,
			(4) $B=B_0[x: E\vee F]$. Then, we define $B_1:=B_0[x:E]$ and 
			$B_2:=B_0[x:F]$ and let 
			
			\begin{equation*}
			A^\star:=A_1^\star\wedge A_2^\star \quad \text{with} \quad 
			A_i:=(B_i\to C)\quad  \text{for } i=1,2.
			\end{equation*}
			In the following items, we show   that $\ofrak {A_i}<\ofrak A$ for $i=1,2$.
			\begin{itemize}[leftmargin=*]
			\item $\ifrak{A_i}\leq\ifrak A$. It is enough to define a 1-to-1 function 
			$f$ mapping the set of implications in ${A_i}$ to implications in $A$.
			Let $D$ be an implication in $A_i$. Then either  $D=A_i$ or it is an
			implication in $B_i$ or an implication in $C$.  
			In the first case we define $f(D)=A$.
			In the second and third cases, we define $f(D):=D$. Note that since 
			$E\vee F$ is an outer occurrence of disjunction in $B$, $D$ can not include 
			that occurrence of $E\vee F$ and thus $D$ is also a subformula of $B$. 
			\item $\cfrakz{A_i}<\cfrakz A$. Note that since $A_i$ and $A$ are both
			implications, then $\cfrakz{A_i}=\cfrak{A_i}$ and $\cfrakz{A}=\cfrak A$.
			So we must only show $\cfrak{A_i}<\cfrak{A}$, which is obvious by definition
			of $A_i$: it has at least one occurrence $\vee$ less than $A$.
			\end{itemize}
			\item $C$ has an outer conjunction, i.e.~a conjunction  
			$E\wedge F$ which is not in 
			the scope of  $\to$. Then there is some 
			formula $C_0(x)$ with the following properties: 
			(1)  $x$ is a variable not appearing in $C$,
			(2) $x$ occurs only once in $C_0$,
			(3) $x$ has an outer occurrence in $C_0$, 
			i.e.~$x$ is not in the scope of arrows,
			(4) $C=C_0[x: E\wedge F]$. Then, we define $C_1:=C_0[x:E]$ and 
			$C_2:=C_0[x:F]$ and let 
			
			\begin{equation*}			
			A^\star:=A_1^\star\wedge A_2^\star \quad \text{ with } \quad 
			A_i:=B\to C_i.
			\end{equation*}
			Then,  with an argument similar to the previous case, we can show that 
			$\ofrak{A_i}<\ofrak A$.
			\item 
				$B=\bigwedge^n_{i=1}B_i$ and 
				$C=\bigvee^{n+m}_{j=n+1}E_j$ in which 
				every $B_i$ and $E_j$ is either atomic   or implication.  
				Again we have several sub-cases:
			\begin{itemize}[leftmargin=*]
				\item $B_i=\bot$ for some $i$. Then define $A^\star:=\top$.
				\item  $B_i\in\varr$ for some $i$. Let $\theta$ be a substitution
				such that $\theta(B_i):=\top$ and $\theta$ is identity elsewhere. 
				Then define $A_1:=\theta (A)$ and $A^\star:=A_1^\star$. 
				We must show first that $\ofrak{A_1}<\ofrak A$. 
				The important point here is that $\top=\bot\to \bot$ 
				and hence this replacement includes addition of new implication to 
				the formula, possibly resulting in bigger $\ifrak{A_1}$. However, 
				it is not a real issue, since our definition of $\ifrak{A_1}$ does 
				not count $\bot\to\bot $. Hence we have 
				$\ifrak{A_1}\leq\ifrak{A}$. On the other hand  since $A$ and $A_1$ 
				both are implications, $\cfrakz{A_1}=\cfrak{A_1}$	and 
				$\cfrakz{A}=\cfrak A$. 
				Furthermore, $A_1$ and $A$ have a bijection of occurrences of $\vee$ and $\wedge$, resulting in $\cfrak{A}=\cfrak {A_1}$.				
				Finally note that $\vfrak{A_1}<\vfrak A$ and hence 
				$\ofrak {A_1}<\ofrak A$.
				
				To show that $\ARNN\vdash A\rhd A^\star$, by the induction hypothesis and 
				cut rule in $\ARNN$, 	it is enough to show that $\ARNN\vdash A\rhd A_1$,
				which itself is an instance of the axiom $\VA$.
				\item $B_i\in\parr$  for some $i$.  Let $B'$ be the result of 
				removing $B_i$ from the conjunction $B$ and define
				
				\begin{equation*}				
				 A_1:=B'\to C \quad \text{ and }\quad 
				A^\star:=B_i\to A_1^\star.
				\end{equation*}
				Then obviously $\ifrak{A_1}=\ifrak{A}$ and 
				$\cfrakz{A_1}=\cfrak{A_1}<\cfrak A=\cfrakz A$. 
				This finishes showing that $\ofrak{A_1}<\ofrak A$. 
				
				Then, by the induction hypothesis, we have $\ARNN\vdash A_1\rhd A_1^\star$,
				and thus, by Montagna's rule, 
				$\ARNN\vdash (B_i\to A_1)\rhd (B_i\to A_1^\star)$.
				This obviously implies $\ARNN\vdash A\rhd A^\star$, as desired.
				\item $B_i=E_i\to F_i$, for every $1\leq i\leq n$ and there is some 
				$1\leq i\leq n+m$ and $E,F,F'$ such that 
				$E\to F\in \Ifrak{E_i}$ and $E\to F'\in \Ifrak{F}$.
				In this case, then define 
				\item $B_i=E_i\to F_i$, for every $1\leq i\leq n$ and there is not any
				$1\leq i\leq n+m$ and $E,F,F'$ such that   
				$E\to F\in \Ifrak{E_i}$ and $E\to F'\in \Ifrak{F}$.
				Then, we define $ B{\downarrow}E_i:=F_i\wedge 
				\bigwedge_{j\neq i}(E_j\to F_j)$ and  
				
				\begin{equation*}
				G_i:=(B{\downarrow}E_i) \to C
				\quad       \text{and} \quad  
				H_i:=\itapp{B}{E_i}
				\quad       \text{and} \quad 
				A^\star:=(\bigwedge_{i=1}^n G_i^\star\wedge \bigvee_{i=1}^{n+m}H_i^\star).
				\end{equation*}
			\end{itemize}
		\end{itemize}
	\end{itemize}
	For the last case, we reason for following facts:
	\begin{itemize}[leftmargin=*]
		\item $\ofrak{G_i}<\ofrak A$
		for every $0\leq i\leq n$. 
		Observe that  
		$\ifrak{G_i}=\#\Ifrak{G_i}\leq \#\Ifrak A=\ifrak A$. Note that equality can not be excluded,
		since it still could be the case that $E_i\to F_i\in\sub{G_i}$.
		Furthermore, $\cfrakz{G_i}=\cfrak{G_i}<\cfrak A=\cfrakz A$, and thus 
		$\ofrak{G_i}<\ofrak A$.
		\item $\ofrak{ H_i}<\ofrak A$ for every 
		$1\leq i\leq n+m$. We have the following cases for $E_i$:
		\begin{itemize}[leftmargin=*]
		\item $E_i\in\varr$: Then we have 
		$H_i=B\to\bot$. Hence obviously $\ifrak{ H_i}\leq\ifrak A$
		and 		$\cfrakz{H_i}=\cfrak{H_i}\leq\cfrak A=\cfrakz A$ and
		$\vfrak{ H_i}<\vfrak A$. 
		Thus $\ofrak{ H_i}<\ofrak A$, as desired.  
		\item $E_i\in\parr$ or $E_i=\bot$. Then we have $H_i=E_i$ 
		and then   obviously $\ifrak {H_i}=0< \ifrak{A}$. Thus $\ofrak{H_i}<\ofrak{A}$.
		\item $E_i=E\circ F$ and $\circ\in\{\vee,\wedge\}$:  
		Then  \Cref{Lem-t1} implies $\ifrak{H_i}\leq \ifrak A$. 
		On the other hand,	by definition we have 
		$H_i=\itapp{B}{E}\circ\itapp B F$ and obviously 
		$\cfrakz{B\to E},\cfrakz{B\to F}< \cfrakz{A}$. Hence
		\Cref{Lem-t1} implies $\cfrakz{\nitapp{B}E},\cfrakz{\nitapp B F}<\cfrakz{A}$.
		Thus $\cfrakz{H_i}<\cfrakz{A}$ and $\ofrak{H_i}<\ofrak{A}$, as desired.
		\item $E_i=E\to F$. Then $H_i:=(E\wedge{B{\downarrow}E})\to (F{\downarrow}E)$. 
		It is enough to show  
		$\ifrak{H_i}<\ifrak{B\to(E\to F)}$. For this aim,
		we define a one-to-one function $f:\Ifrak{H_i}\longrightarrow
		\big(\Ifrak{B\to(E\to F)}\setminus\{ E\to F\}\big)$. 
		We define $f(G)$ for every  $G\in \Ifrak{H_i}$ as follows.		
		\begin{itemize}
		\item[i.] $G=H_i$. Then define   $f(G):=(B\to(E\to F))$. Obviously, in this case 
		we do have $E\to F\neq f(G)$.
		\item[ii.] 
		$G\in \Ifrak{E}\cup \Ifrak{F{\downarrow}E}\cup \Ifrak{B{\downarrow}E}$.
		Then 
		there must exist some $G'\in  \Ifrak E\cup \Ifrak F\cup \Ifrak B$ 
		such that $G= (G'{\downarrow}E)$. Then define $f(G):=G'$ for some such $G'$ with 
		minimum length.
		Note that in this case also we can not have $f(G)=E\to F$; otherwise, 
		we could have defined $f(G):=F$, since $(E\to F){\downarrow}E=F{\downarrow}E$,
		which contradicts the minimality of the length of $f(G)$. 
		\end{itemize}
Since $G'_1=G'_2$ implies $(G'_1{\downarrow}E)=(G'_2{\downarrow}E)$, the above-defined function $f$,
is injective.
		\end{itemize}
		\item $\ARNN\vdash A\rhd A^\star$: 
		Note that $\vdash A\to G_i$ for every $i\leq n$. 
		This implies $\ARNN\vdash A\rhd G_i$ for every $i\leq n$. On the other hand, \Cref{variant-visser-prop}
		implies $\ARNN\vdash A\rhd\bigvee_{i=1}^{n+m}H_i$.
		Finally,		
		by induction hypothesis we have desired conclusion.	
		\item $\IPC\vdash  A^\star\to A$: 
		By induction hypothesis, 
		it is sufficient to show  

		\begin{equation*}		
		\IPC\vdash (\bigwedge_{i=1}^n G_i\wedge \itapp{B}{E_j})\to A
		\end{equation*}
		for every $1\leq j\leq n+m$.
		So we reason inside $\IPC$. Assume $\bigwedge_{i=1}^n G_i$ and $\nitapp{B}{E_j}$
		and $B$. We want to derive $C$. If $j>n$, then by definition we have 
		$\itapp{B}{E_j}\to C$ and we are done. If $j\leq n$, by 
		$G_j$ we have $(B{\downarrow}E_j)\to C$ and, hence, 
		it is sufficient to show $B{\downarrow}E_j$. 
		Then, by $B$, it is sufficient to show 
		$E_j$, which holds by $B$ and $\nitapp{B}{E_j}$.\qedhere
	\end{itemize}
\end{proof}
\begin{lemma}\label{Lem-t1}
Given $B,E\in\lcalz$ and $\ifrak{.}$ and $\cfrakz{.}$ as  defined in 
the  \Cref{comp-meas-def}, we have 
\begin{enumerate}
\item $\ifrak{\nitapp{B}{E}}\leq \ifrak{B\to E}$.
\item $\cfrakz{\nitapp{B}{E}}\leq \cfrakz{B\to E}$.
\end{enumerate}
\end{lemma}
\begin{proof}
The proof is by  induction on the complexity of $E$.
We only reason for the following case and only for 
the first item and leave the rest to reader.

If $E$ is an implication $C\to D$. Then 
$\itapp B E := (C\wedge B{\downarrow}C)\to (D{\downarrow}C)$. It is enough to define an injective function 
$f:\Ifrak{(C\wedge B{\downarrow}C)\to (D{\downarrow}C)}\longrightarrow \Ifrak{B\to(C\to D)}$.
We define $f((C\wedge B{\downarrow}C)\to (D{\downarrow}C)):=B\to(C\to D)$ and for every $G\in \Ifrak C\cup \Ifrak{B{\downarrow}C}\cup \Ifrak{D{\downarrow}C}$, there exist
some $G'\in \Ifrak{C}\cup \Ifrak B\cup \Ifrak D$ such that $G=G'{\downarrow}C$. Define $f(G)$ as any such $G'$.
Then it is obvious that this $f$ is injective.
\end{proof}

\begin{theorem}\label{Theorem-nnilp-Pres}
	For every $A,B$, following items are equivalent:
	\begin{enumerate}
		\item $ \ARNN\vdash A\rhd B $,
		\item $ A\prnpar B $,
		\item $\vdash A^\star\to B $.
	\end{enumerate}
\end{theorem}
\begin{proof}
		We show $1 \Rightarrow 2 \Rightarrow
	3\Rightarrow 1$:
	\begin{itemize}
		\item  $1\Rightarrow 2$: By \Cref{Lem-con1-ipc}.   
		\item   $2\Rightarrow 3$: Let $ A\prnpar B $.  
		By \Cref{star-prop} we have $ A^\star\in\NNILpar $ and 
		$ \vdash A^\star\to A $.  
		Then by $ A\prnpar B $  
		we get $ \vdash A^\star\to B $.  
		\item $3\Rightarrow 1$: 
		From $ \vdash A^\star\to B $ we get 
		$ \ARNN\vdash A^\star\rhd B $. Also by \Cref{star-prop} 
		we have $\ARNN\vdash A\rhd A^\star$ and then 
		Cut  implies desired result.\qedhere
	\end{itemize}
\end{proof}
\begin{lemma}\label{prnp=prpnp}
${\prnpar}={\prpnpar}$ .
\end{lemma}
\begin{proof}
\Cref{Lem-nnil-normal-form}\footnote{Notice that the statement of 
\Cref{Lem-nnil-normal-form} is for the finite language. One may easily observe that its 
proof is not relying on finiteness of atomic formulas.
Nevertheless, it is also easy to 
observe that the case for infinite language can be derived from the finite case.} 
and \Cref{vee-pres}\footnote{Notice that the \Cref{vee-pres} holds with the same proof for 
a language with infinite set of atomics.}. 
\end{proof}

\begin{lemma}\label{VA-validity}
	$B\to C\prnpar \bigvee_{i=1}^{n+m}
	\itp{B}{E_i}$, in which  
	$B=\bigwedge_{i=1}^n (E_i\to F_i)$ and 
	$C=\bigvee_{i=n+1}^{n+m} E_i$.
\end{lemma}
\begin{proof}
	We reason by  contraposition. Let $E\in\NNILpar$
	be such that $\nvdash E\to (\bigvee_{i=1}^{n+m}
	\itp{B}{E_i})$. Hence there is some finite rooted 
	$\kcal=(W,\pce,\V)$ such that 
	$\kcal,w_0\Vdash E$ and $\kcal,w_0\nVdash \bigvee_{i=1}^{n+m}	\itp{B}{E_i}$. Let $I$ be the set of indexes $i$ such that $E_i\in\parr$ or $E_i=\bot$.
	Also let $J$ be the complement of $I$. Thus 
	for every $i\in I$ we have $\kcal,w_0\nVdash E_i$
	and for every $j\in J$, there is some $w_j\sce w_0$
	such that $\kcal,w_j\Vdash B$ and 
	$\kcal,w_j\nVdash E_j$.
	Let $W'$ defined   as follows:
	
	\begin{equation*}
	W':=\{w_0\}\cup \{v\in W: \exists j\in J(w_j\pce v)\}
	\end{equation*}
	and  define $\kcal':=(W',\pce,\V)$.
	Then since $E\in\NNIL$, \Cref{Theorem-NNIL-Submodel}\footnote{Notice that the 
	statement of this theorem is actually independent of the finiteness or infiniteness of the set of atomic formulas in the language.} 
	implies $\kcal',w_0\Vdash E$. 
	Moreover, it is not difficult to observe that 
	$\kcal',w_0\Vdash B$ and $\kcal',w_0\nVdash C$. 
	Thus $\kcal',w_0\nVdash E\to (B\to C)$ and then 
	$ \nvdash E\to (B\to C)  $.
\end{proof}

\begin{lemma}\label{Lem-con1-ipc}
	$\ARNN\vdash A\rhd B$ implies $ A\prnpar B $. 
\end{lemma}
\begin{proof}
	We use induction on complexity of the proof 
	$ \ARNN\vdash A\rhd B $. All steps trivially hold except:
	\begin{itemize}[leftmargin=*]
		\item $\VA$: This axiom holds because $\IPC$
		is closed under substitutions and $ \theta(E)=E $ for every $ E\in\NNILpar $.
		\item $\VAR$: \Cref{VA-validity}.
		\item Disj: 
		Let $A\prnpar C$ and $B\prnpar C$ seeking to show
	    $A\vee B\prnpar C$.  By \Cref{prnp=prpnp} it is sufficient to show 
	    $A\vee B\prpnpar C$.
		Let $E\in \pNNILpar$ such that $\vdash E\to (A\vee B)$. 
		Since $E$ is  prime, either we have 
		$\vdash E\to A$ or $\vdash E\to B$. 
		Then by $A\prnpar C$ and $B\prnpar C$, 
		in either of the cases  we have  
		$\vdash E\to C$.		
		\qedhere
	\end{itemize}
\end{proof}

\section*{Acknowledgement}
We extend our heartfelt gratitude to the following individuals for their 
valuable contributions, stimulating discussions, and communications on topics of preservativity, unification, and admissibility: Mohammad Ardeshir,
Majid Alizadeh,
Philippe Balbiani, Lev Beklemishev, Silvio Ghilardi,
Rosalie Iemhoff,
Emil Jeřábek,
Dick de Jongh, Fedor Pakhomov, 
Konstantinos Papafilippou, Deniz Tahmouresi and Rineke Verbrugge. 
Their valuable insights have significantly enriched the quality of this research, 
and the author is deeply appreciative of their support.

Additionally, we would like to express our sincere appreciation to 
Lev Beklemishev, Albert Visser and the anonymous remainder of the referee team,  
for their diligent efforts in reviewing and enhancing various sections of the paper, 
as well as providing invaluable corrections and suggestions for improved notations.
Their meticulous feedback and suggestions have played a vital role in refining this work, 
making it more robust and comprehensive. We are immensely thankful for their guidance.

\section*{Funding}
This work is partially funded by FWO grant G0F8421N and BOF grant BOF.STG.2022.0042.01.

\newpage
\begin{appendices}
In the following table, we assume that $X$, $X_1,X_2,\ldots, X_n$ are arbitrary sets of formulas.
\renewcommand{\figurename}{Table}
\begin{figure}[h]
\bgroup
\def\arraystretch{1.8}
\begin{center}
\begin{tabular}{|c|c|c|}
    \hline
    \textbf{ Symbol} 
    & \textbf{Description}
    & \textbf{Section}
\\ \hline\hline 
     $\varr$
    & atomic variables
    &  \ref{lang}
\\ \hline
    $\parr$
    & atomic parameters
    & \ref{lang}
\\ \hline
     $\atom$
    & $\parr\cup\varr$
    &  \ref{lang}
\\ \hline    
     $\lcalz(X)$
    &  boolean combinations of formulas in $X$  
    & \ref{lang}
\\ \hline
     $\IPC$
    &  Intuitionistic Propositional Logic 
    &  \ref{lang}
\\ \hline    
     $\vdash$
    &  Derivability in Intuitionistic Logic 
    & \ref{lang}
\\ \hline 
	  $\Gamma^\vee$  
    & Disjunctive closure of $\Gamma$
    &   \ref{notation-set}
\\ \hline    
      $\NNIL$   
    &   set of formulas with No Nested Implications in the Left
    & \ref{notation-set}
\\ \hline
 ${\sf P}$ 
& 
 set of formulas $A$ such that $\IPC+A$ still has Disjunction Property
& \ref{notation-set} 
\\ \hline
	  $\Gamma(X)$
    &  $\Gamma\cap\lcalz(X)$
    &  \ref{notation-set} 
    \\ \hline
 $A$-identity
    &  \mspan{a substitution $\theta$ is $A$-identity if   
    $\theta$ is identity substitution, modulo}{$\IPC+A$-provable eqiovalence}
    &   \ref{Gamma-proj-nonmodal}
    \\ \hline    
   $  A\xratth E $   
    &  $\theta$ is $A$-identity and $\vdash \theta(A)\lr E$ and $E\in\lcalzpar$
    & \ref{Gamma-proj-nonmodal}
\\ \hline    
      $\darrow \Gamma$
    & $\{A:\exists\,\theta\exists\, E\in\Gamma(\parr)\ A\xratth E \}:=$
    the set of $\Gamma$-projective formulas
    & \ref{Gamma-proj-nonmodal}
\\ \hline	
$X_1\ldots X_n$
    & $X_1\cap \ldots\cap X_n$  
    &  \ref{notation-set}
\\ \hline
$\ap\Gamma {} A$
& the greatest lower bound  of $A$ in $\Gamma$ (in the intuitionistic logic)
& \ref{glb}
\\ \hline
	$A\adsm{}{\Gamma} B$ 
    &   $\forall\,\theta$
	 $\forall\, C\in\seventaghir{\Gamma(\parr)}$: 
	$ \vdash C\to\theta(A)$ 
	implies $ \vdash C\to \theta(B)$
    &\ref{pres-admis} 
\\ \hline
	$A\pres{}{\Gamma} B$ 
    &    
	 $\forall\, C\in\seventaghir{\Gamma}$: 
	$ \vdash C\to A$ 
	implies $ \vdash C\to B$
    &\ref{pres-admis}     
\\ \hline
    \end{tabular}
\caption[Table  of symbols and notations]{List of symbols and notations\label{Table-Axioms}}
\end{center}
\egroup
\end{figure}

\newpage

\noindent\textbf{\Large The system $\ARN$:} (see \cref{ARp})
\\[5mm]
\begin{equation*}
\itp{A}{B}:=\begin{cases}
B \quad &: B\in\parr\cup\{\bot\}\\
A\to B &: \text{otherwise}
\end{cases}
\end{equation*}
\\[5mm]
\textbf{Axioms}
\begin{itemize}[leftmargin=1.5cm]
	\item[${\sf Ax}:$] \quad $A\rhd B$, 
	for every $  \vdash A\to B$. 
\item[$\VAR:$]\quad $B\to C\rhd \bigvee_{i=1}^{n+m}
\itp{B}{E_i}$, in which  
$B=\bigwedge_{i=1}^n (E_i\to F_i)$ and 
$C=\bigvee_{i=n+1}^{n+m} E_i$.
\end{itemize}	
\textbf{Rules}
\begin{center}
	\bgroup
	\begin{tabular}{c c}
		\Ax{$A\rhd B$}
		\Ax{$A\rhd C$}
		\RLa{Conj}
		\BI{$A\rhd (B\wedge C)$}
		\DP \quad \quad \quad 
		&
		\Ax{$A\rhd B$}
		\Ax{$B\rhd C$}
		\RLa{Cut}
		\BI{$A\rhd C$}
		\DP  
		\\[5mm]
\Ax{$B\rhd A$}
		\Ax{$C\rhd A$}
		\RLa{Disj}
		\BI{$(B\vee C)\rhd A$}
		\DP
&
\Ax{$A\rhd B$}
\Ax{$p\in\parr$}
\RLa{$\mont(\parr)$}
\BI{$(p\to A)\rhd (p\to B)$}
\DP
	\end{tabular}
	\egroup
\end{center}
\vspace{5mm}
\noindent
\textbf{\Large The system $\ARNN$:} (see \cref{sec-pres-2})

\vspace{5mm}

All axioms and rules listed above plus the following axiom:
\begin{equation*}
\VA:  A\rhd \theta(A) \text{ for every substitution $ \theta $ (which by default is identity on parameters)}.
\end{equation*}
\end{appendices}


\begin{thebibliography}{}

\bibitem[Ardeshir and Mojtahedi, 2018]{Sigma.Prov.HA}
Ardeshir, M. and Mojtahedi, M. (2018).
\newblock The {$\Sigma_1$}-{P}rovability {L}ogic of {${\sf HA}$}.
\newblock {\em Annals of Pure and Applied Logic}, 169(10):997--1043.

\bibitem[Ghilardi, 1997]{Ghilardi97}
Ghilardi, S. (1997).
\newblock Unification through projectivity.
\newblock {\em J. Log. Comput.}, 7(6):733--752.

\bibitem[Ghilardi, 1999]{Ghil99}
Ghilardi, S. (1999).
\newblock Unification in {I}ntuitionistic {L}ogic.
\newblock {\em Journal of {S}ymbolic {L}ogic}, 64(2):859--880.

\bibitem[Ghilardi, 2000]{Ghil2000modal}
Ghilardi, S. (2000).
\newblock {B}est solving modal equations.
\newblock {\em Annalas of {P}ure and {A}pplied {L}ogic}, 102(3):183--198.

\bibitem[Ghilardi, 2002]{ghilardi2002resolution}
Ghilardi, S. (2002).
\newblock A resolution/tableaux algorithm for projective approximations in ipc.
\newblock {\em Logic Journal of the IGPL}, 10(3):229--243.

\bibitem[Harrop, 1960]{Harrop60}
Harrop, R. (1960).
\newblock {C}oncerning formulas of the types {$A\to B \vee C$}, {$A\to
  (Ex)B(x)$} in intuitionistic formal systems.
\newblock {\em Journal of Symbolic Logic}.

\bibitem[Iemhoff, 2001a]{Iemhoff-admissibility}
Iemhoff, R. (2001a).
\newblock {O}n the {A}dmissible {R}ules of {I}ntuitionistic {P}ropositional
  {L}ogic.
\newblock {\em The Journal of Symbolic Logic}, 66(1):281--294.

\bibitem[Iemhoff, 2001b]{IemhoffT}
Iemhoff, R. (2001b).
\newblock {\em Provability Logic and Admissible Rules}.
\newblock PhD thesis, University of Amsterdam.

\bibitem[Iemhoff, 2003]{Iemhoff.Preservativity}
Iemhoff, R. (2003).
\newblock {P}reservativity {L}ogic. ({A}n analogue of interpretability logic
  for constructive theories).
\newblock {\em Mathematical Logic Quarterly}, 49(3):1--21.

\bibitem[Iemhoff, 2005]{iemhoff2005intermediate}
Iemhoff, R. (2005).
\newblock Intermediate logics and {V}isser's rules.
\newblock {\em Notre Dame Journal of Formal Logic}, 46(1):65--81.

\bibitem[Iemhoff et~al., 2005]{Iemhoff2005}
Iemhoff, R., De~Jongh, D., and Zhou, C. (2005).
\newblock Properties of intuitionistic provability and preservativity logics.
\newblock {\em Logic Journal of the IGPL}, 13(6):615--636.

\bibitem[Iemhoff and Metcalfe, 2009]{iemhoff2009proof}
Iemhoff, R. and Metcalfe, G. (2009).
\newblock Proof theory for admissible rules.
\newblock {\em Annals of Pure and Applied Logic}, 159(1-2):171--186.

\bibitem[Ilin et~al., 2020]{NNIL-rev}
Ilin, J., de~Jongh, D., and Yang, F. (2020).
\newblock {NNIL}-formulas revisited: {U}niversal models and finite model
  property.
\newblock {\em Journal of Logic and Computation}, 31(2):573--596.

\bibitem[Je{\v{r}}{\'a}bek, 2005]{jevrabek2005admissible}
Je{\v{r}}{\'a}bek, E. (2005).
\newblock Admissible rules of modal logics.
\newblock {\em Journal of Logic and Computation}, 15(4):411--431.

\bibitem[Mojtahedi, 2021]{mojtahedi2021hard}
Mojtahedi, M. (2021).
\newblock Hard provability logics.
\newblock In {\em Mathematics, Logic, and their Philosophies}, pages 253--312.
  Springer.

\bibitem[Mojtahedi, 2022]{PLHA}
Mojtahedi, M. (2022).
\newblock {O}n {P}rovability {L}ogic of {H}eyting {A}rithmetic.
\newblock {\em In preparation}.
\newblock In preparation.

\bibitem[Rybakov, 1987a]{Rybakov87}
Rybakov, V.~V. (1987a).
\newblock {B}ases of admissible rules of the modal system {G}rz and of
  intuitionistic logic.
\newblock {\em Mathematics of the USSR-Sbornik}, 56(2):311--331.

\bibitem[Rybakov, 1987b]{Rybakov_1987}
Rybakov, V.~V. (1987b).
\newblock {D}ecidability of admissibility in the modal system {G}rz and in
  intuitionistic logic.
\newblock {\em Mathematics of the {USSR}-Izvestiya}, 28(3):589--608.

\bibitem[Rybakov, 1992]{rybakov_1992}
Rybakov, V.~V. (1992).
\newblock Rules of inference with parameters for intuitionistic logic.
\newblock {\em Journal of Symbolic Logic}, 57(3):912–923.

\bibitem[Rybakov, 1997]{Rybakov_Book}
Rybakov, V.~V. (1997).
\newblock {\em {A}dmissibility of logical inference rules}.
\newblock Elsevier.

\bibitem[Troelstra and van Dalen, 1988]{TD}
Troelstra, A.~S. and van Dalen, D. (1988).
\newblock {\em Constructivism in mathematics. {V}ol. {I}}, volume 121 of {\em
  Studies in Logic and the Foundations of Mathematics}.
\newblock North-Holland Publishing Co., Amsterdam.

\bibitem[Visser, 1996]{visser1996uniform}
Visser, A. (1996).
\newblock Uniform interpolation and layered bisimulation.
\newblock {\em G{\"o}del'96: Logical foundations of mathematics, computer science and physics---Kurt G{\"o}del's legacy, Brno, Czech Republic, August 1996, proceedings}, 6:139--165.
\newblock  Association for Symbolic Logic.

\bibitem[Visser, 2002]{Visser02}
Visser, A. (2002).
\newblock Substitutions of {$\Sigma_1^0$} sentences: explorations between
  intuitionistic propositional logic and intuitionistic arithmetic.
\newblock {\em Ann. Pure Appl. Logic}, 114(1-3):227--271.
\newblock Commemorative Symposium Dedicated to Anne S. Troelstra
  (Noordwijkerhout, 1999).

\bibitem[Visser et~al., 1995]{Visser-Benthem-NNIL}
Visser, A., van Benthem, J., de~Jongh, D., and R.~de Lavalette, G.~R. (1995).
\newblock {${\rm NNIL}$}, a study in intuitionistic propositional logic.
\newblock In {\em Modal logic and process algebra ({A}msterdam, 1994)},
  volume~53 of {\em CSLI Lecture Notes}, pages 289--326. CSLI Publ., Stanford,
  CA.

\bibitem[Zhou, 2003]{Zhou-PhD}
Zhou, C. (2003).
\newblock {\em {S}ome {I}ntuitionistic {P}rovability and {P}reservativity
  {L}ogics (and their interrelations)}.
\newblock PhD thesis, ILLC, Amsterdam.

\end{thebibliography}
\end{document}